\documentclass[11pt]{amsart}

\usepackage[hidelinks]{hyperref}
\hypersetup{
  colorlinks=true,
  linkcolor=black,     % For TOC, cross-references, etc.
  citecolor=blue,      % For citations only
  urlcolor=black,      % For URLs
  filecolor=black      % For file links
}
\usepackage[foot]{amsaddr}
\usepackage[letterpaper,top=3.5cm,bottom=3.5cm,left=3.5cm,right=3.5cm,marginparwidth=1.75cm]{geometry}
\usepackage[utf8]{inputenc}            % allow utf-8 input
\usepackage[T1]{fontenc}               % use 8-bit T1 fonts

\usepackage{microtype}
\usepackage{graphicx}
\usepackage{subfigure}
\usepackage{booktabs} % for professional tables
\usepackage{enumitem}
\usepackage{color}
\usepackage{soul}
\usepackage{wrapfig}
\usepackage{overpic}
\usepackage{amsmath}
\usepackage{amssymb}
\usepackage{amsthm}
\usepackage{commands}
\usepackage[showonlyrefs]{mathtools}

\usepackage{siunitx} % For aligning numbers by decimal point

\usepackage[capitalize,noabbrev]{cleveref}

\newtheorem{theorem}{Theorem}[section]
\newtheorem{lemma}[theorem]{Lemma}
\newtheorem{assumption}[theorem]{Assumption}
\newtheorem{proposition}[theorem]{Proposition}
\newtheorem{definition}[theorem]{Definition}
\newtheorem{example}[theorem]{Example}

\theoremstyle{remark}
\newtheorem{remark}[theorem]{Remark}
\numberwithin{equation}{section}

\newcommand{\xqed}[1]{%
    \leavevmode\unskip\penalty9999 \hbox{}\nobreak\hfill
    \quad\hbox{\ensuremath{#1}}}
\newcommand{\Endofdef}{\xqed{\lozenge}}

\newcommand{\cN}{\mathcal{N}}

\newcommand{\bE}{\mathbb{E}}
\newcommand{\bP}{\mathbb{P}}
\newcommand{\bR}{\mathbb{R}}

\newcommand{\rd}{\mathrm{d}}

\newcommand{\sfN}{\mathsf{N}}

\title[Variational Optimality of Föllmer Processes in Generative Diffusions]{Variational Optimality of Föllmer Processes in Generative Diffusions}

\author{Yifan Chen\textsuperscript{1}}
\address{\textsuperscript{1}Department of Mathematics, University of California, Los Angeles, CA, USA}
\email{yifanchen@math.ucla.edu}
\author{Eric Vanden-Eijnden\textsuperscript{2,3}}
\address{\textsuperscript{2}Machine Learning Lab, Capital Fund Management, Paris, France}
\address{\textsuperscript{3}Courant Institute, New York University, NY, USA}
\email{eve2@nyu.edu}

\begin{document}
\setcounter{tocdepth}{1}
\begin{abstract}
We construct and analyze generative diffusions that transport a point mass to a prescribed target distribution over a finite time horizon using the stochastic interpolant framework. The drift is expressed as a conditional expectation that can be estimated from independent samples without simulating stochastic processes. We show that the diffusion coefficient can be tuned \emph{a~posteriori} without changing the time-marginal distributions. Among all such tunings, we prove that minimizing the impact of estimation error on the path-space Kullback--Leibler divergence selects, in closed form, a F\"ollmer process---a diffusion whose path measure minimizes relative entropy with respect to a reference process determined by the interpolation schedules alone. This yields a new variational characterization of F\"ollmer processes, complementing classical formulations via Schr\"odinger bridges and stochastic control, and provides a conditional-expectation representation of the F\"ollmer drift that enables simulation-free estimation from data. We further establish that, under this optimal diffusion coefficient, the path-space Kullback--Leibler divergence becomes independent of the interpolation schedule, rendering different schedules statistically equivalent in this variational sense. We provide numerical experiments to illustrate the impact of path-space variational optimality of F\"ollmer's processes in probabilistic forecasting and data assimilation applications.
\end{abstract}

\maketitle

\tableofcontents
\newpage
\section{Introduction}

The central problem in generative modeling is to produce samples from a probability distribution $\mu_\star$ given only a finite collection of observations. A prominent class of methods is based on continuous-time dynamics: ordinary or stochastic differential equations (ODEs or SDEs) whose time marginals interpolate between a tractable reference measure and the target. These include diffusion and score-based generative models \cite{sohl2015deep,ho2020denoising,song2020score}, flow matching and rectified flows \cite{lipman2022flow,liu2022flow}, and stochastic interpolant methods \cite{albergo2022building,albergo2023stochastic}. Despite their empirical success, many aspects of these constructions---in particular questions of optimality, regularity, and the effect of design choices on statistical performance---are not yet systematically understood. We address several such questions in the case where the reference measure is a Dirac mass, so that all variability in the generative process is created by the diffusion.

The Dirac-source case is not merely a special instance: it is classical in probability theory and natural in several applied settings. Diffusions that start from a fixed point and are constrained to reach a prescribed distribution at a finite terminal time include F\"ollmer processes \cite{follmer1986time} and Schr\"odinger bridges with degenerate initial data \cite{schrodinger1932theorie,leonard2014survey,chen2021stochastic}. They also arise naturally in probabilistic forecasting and data assimilation, where the point source represents the most recent observation of a dynamical system and the terminal law encodes the conditional distribution of future states. Such constructions have proved effective in applications ranging from turbulent fluid dynamics \cite{chen2024probabilistic} and weather prediction \cite{kossaifi2026demystifying} to cosmological simulation \cite{cuesta-lazaro2024joint}.

In the remainder of this introduction, we formulate the problem precisely, state our main results, discuss related work, and outline the organization of the paper.

\subsection{Setting and main results}
\label{sec:setup-main-results}

Let $\mu_\star$ be a probability distribution on $\bR^d$ with
\[
    \int_{\bR^d} |x|^2 \,\mu_\star(\rd x) < \infty .
\]
We seek to construct a diffusion process $(X_t)_{t\in[0,1]}$ satisfying
\begin{equation}
    \label{eq:intro:sde}
    \rd X_t = b_t(X_t)\,\rd t + g_t\,\rd W_t, \quad X_0 = 0, \quad X_1 \sim \mu_\star ,
\end{equation}
where $(W_t)$ is a standard Wiener process, $b_t$ is the drift, and $g_t$ the diffusion coefficient. In this point-source setting diffusion is essential: an ODE cannot transport a Dirac mass to a general distribution.

\subsubsection{Stochastic interpolants and a baseline diffusion}

We construct such diffusions through the stochastic interpolant framework \cite{albergo2023stochastic}, which directly specifies the marginal evolution and yields simulation-free regression objectives for estimation. Fix interpolation schedules $\beta,\sigma\in C^1([0,1])$ such that
\[
    \beta_0 = 0, \quad \beta_1 = 1, \quad \sigma_1 = 0, \quad
    \dot\beta_t > 0, \ \dot\sigma_t < 0 \ \text{for } t\in [0,1] .
\]
Given $x_\star \sim \mu_\star$, independent of a Wiener process $(W_t)$, we define the \emph{stochastic interpolant}
\begin{equation}
    \label{eq:intro:interpolant}
    I_t = \beta_t \, x_\star + \sigma_t \, W_t,\qquad t\in[0,1] ,
\end{equation}
so that $I_0 = 0$ and $I_1 = x_\star \sim \mu_\star$. For each $t\in(0,1)$, the law of $I_t$ is the convolution of $\mu_\star$ with a non-degenerate Gaussian, and therefore admits a smooth density $\rho_t$ with respect to Lebesgue measure.

The process $(I_t)$ is in general not Markovian, but one can construct a Markov diffusion with the same marginals at each time; this can be seen as a consequence of Gy\"ongy's mimicking theorem \cite{gyongy1986mimicking}. For the interpolant, this Markovian projection admits an explicit drift formula.

\begin{theorem}[Baseline generative diffusion; informal, see Theorem~\ref{thm:1:b} for the precise statement]
\label{thm:1:b_informal}
The SDE \eqref{eq:intro:sde} with $g_t = \sigma_t$ and drift
\begin{equation}
    \label{eq:intro:drift}
    b_t(x) = \bE[\dot\beta_t \, x_\star + \dot\sigma_t \sqrt{t}\, z \;|\; \beta_t \, x_\star + \sigma_t\sqrt{t}\, z = x] ,
\end{equation}
where $z \sim \sfN(0,\mathrm{I})$ is independent of $x_\star$, satisfies $\mathrm{Law}(X_t) = \mathrm{Law}(I_t)$ for all $t \in [0,1]$, and in particular $X_1 \sim \mu_\star$.
\end{theorem}

The representation \eqref{eq:intro:drift} expresses $b_t$ as a conditional expectation with respect to the joint law of $(x_\star,z)$. As a result, $b_t$ can be characterized as the minimizer of a square loss and estimated by regression from i.i.d.\ samples of $x_\star$ and $z$, without simulating any SDE. The same representation also relates $b_t$ to the \textit{score} $\nabla\log\rho_t$ in closed form (see Theorem~\ref{thm:tunning_g}), which sets up the variational analysis of the next subsection.

\subsubsection{Tuning the diffusion coefficient and F\"ollmer processes}

A key structural feature is that the choice $g_t = \sigma_t$ is not intrinsic. Given a continuous function $g_t>0$, one may change the diffusion coefficient and adjust the drift without altering the marginal laws.

To formulate this, let $\rho_t$ denote the common time-$t$ density of $I_t$ and $X_t$. For any such $g$, set
\begin{equation}
    \label{eq:intro:tuned-drift}
    b_t^g(x) = b_t(x) + \tfrac{1}{2}(g_t^2 - \sigma_t^2)\,\nabla\log\rho_t(x) .
\end{equation}
At the level of the Fokker--Planck equation, the identity $\nabla\cdot(\rho\,\nabla\log\rho) = \Delta\rho$ implies that the diffusion with drift $b_t^g$ and diffusion coefficient $g_t$ has the same time marginals as the baseline diffusion. We make this precise in Theorem~\ref{thm:tunning_g}, including assumptions on $g$ to ensure non-singular drifts.

Thus the interpolant construction produces a whole family of diffusions of the form \eqref{eq:intro:sde}, parametrized by the scalar function $g$. This raises a natural variational question: can one select $g$ in a principled way? Our criterion is based on the effect of estimation error. Let $\hat b_t$ be an estimator of $b_t$ obtained from data; the tuned drift $\hat{b}_t^g$ is then derived from $\hat{b}_t$ via \eqref{eq:intro:tuned-drift}, so the estimator transfers across the family at no additional training cost (see Theorem \ref{thm:tunning_g}). Writing $\bP_{X^g}$ and $\bP_{\hat X^g}$ for the path measures of the exact and estimated diffusions at diffusion coefficient $g_t$, we choose the criterion to be the path-space Kullback--Leibler (KL) divergence
\begin{equation}
    \label{eq:intro:kl}
    \mathrm{KL}\!\left(\bP_{X^g} \,\big\|\, \bP_{\hat X^g}\right),
\end{equation}
and ask for which $g$ this quantity is minimized.

\begin{theorem}[KL-optimal diffusion coefficient; informal, see Theorem~\ref{prop-min-KL} for the precise statement]
\label{prop-min-KL_informal}
The function $g$ that minimizes the path-space KL divergence \eqref{eq:intro:kl} is
\begin{equation}
    \label{eq:intro:gF}
    g_t^{\rm F} = \left|2t\sigma_t^2\,\frac{\rd}{\rd t}\log\frac{\beta_t}{\sqrt{t}\,\sigma_t}\right|^{1/2} .
\end{equation}
\end{theorem}

The structure behind this formula is clarified by the next result, which shows that the diffusion with optimally tuned coefficient $g^{\rm F}$ is a F\"ollmer process in the sense of relative entropy minimization with respect to a linear reference diffusion.

\begin{theorem}[Variational characterization of F\"ollmer processes; informal, see Theorem~\ref{th:foll:app} for the precise statement]
\label{th:foll:app_informal}
Assume that $\beta_t/(\sqrt{t}\,\sigma_t)$ is non-decreasing on $[0,1]$. Let $X^{g^{\rm F}}$ be the diffusion obtained by inserting $g_t=g_t^{\rm F}$ into \eqref{eq:intro:tuned-drift}. Then $X^{g^{\rm F}}$ is the F\"ollmer process associated with the reference process
\begin{equation}
    \label{eq:intro:ref}
    \rd Y_t = a_t\, Y_t\,\rd t + g_t^{\rm F}\,\rd W_t, \quad Y_0 = 0, \quad a_t = \frac{\rd}{\rd t}\log\frac{\beta_t^2 + t\sigma_t^2}{\beta_t} .
\end{equation}
That is, among all diffusions $\check X$ with $\check X_0 = 0$ and $\check X_1 \sim \mu_\star$, the process $X^{g^{\rm F}}$ uniquely minimizes $\mathrm{KL}(\bP_{\check X} \| \bP_Y)$.
\end{theorem}

The reference process $Y$ depends only on the schedules $(\beta_t,\sigma_t)$, not on the target $\mu_\star$, which enters solely through the terminal constraint. The proof proceeds by time-reversing the tuned diffusion and showing that, at the optimal $g^{\rm F}$, the reversed drift no longer involves the score $\nabla\log\rho_t$; changing the initial condition in this reversed SDE then produces a reference diffusion whose bridges are shared with $X^{g^{\rm F}}$, and the variational optimality follows from the usual disintegration of relative entropy. The construction can equivalently be read as a variational optimality of time reversal in this setting.

This is a new variational characterization of F\"ollmer processes. In the classical formulation \cite{follmer1986time,lehec2013representation} the reference process is fixed \textit{a~priori}---typically Brownian motion---and the F\"ollmer process is the entropy minimizer subject to endpoint constraints; here, the reference process itself emerges from an optimization within a parametric family of generative diffusions sharing the same marginal laws. The two parametrizations are equivalent in expressivity: one may also start from a F\"ollmer process with prescribed reference dynamics and recover the stochastic interpolant via the score. As a byproduct, we further show (Proposition~\ref{prop-Follmer-SDE}) that the F\"ollmer drift admits a conditional-expectation representation analogous to the one for the baseline drift in \eqref{eq:intro:drift}, so it too can be estimated from data by simulation-free square-loss regression.

\subsubsection{Schedule invariance}

With $g_t = g_t^{\rm F}$ fixed by Theorem~\ref{prop-min-KL_informal}, the remaining design freedom is the choice of schedules $(\beta_t,\sigma_t)$. A natural question is whether some schedules are intrinsically better than others, in terms of the statistical error induced by estimating the drift from finitely many samples. Our final main result shows that, at the level of path-space KL divergence, the answer is no.

\begin{theorem}[Schedule invariance; informal, see Theorem~\ref{thm:schedule-invariance} for the precise statement]
\label{thm:schedule-invariance_informal}
After KL-optimal tuning of the diffusion coefficient, the minimized path-space KL divergence takes the form
\[
    \mathrm{KL}^\star = 2\int_0^\infty r\,\bE\!\left[\bigl|\nabla\log q_r(x_\star + rz) - \hat s_r(x_\star + rz)\bigr|^2\right]\rd r ,
\]
where $q_r$ is the density of $\mu_\star * \cN(0, r^2\mathrm{I})$ and $\hat s_r$ is the estimated score at noise level $r$ derived from $\hat{b}_t$. In particular, assuming the same accuracy of the estimated score $\hat s_r$, $\mathrm{KL}^\star$ is independent of the schedules $(\beta_t,\sigma_t)$.
\end{theorem}

Thus, after optimal tuning, different interpolation schedules are statistically equivalent in the sense of path-space relative entropy: the minimized KL divergence depends only on the accuracy of the score estimates across noise scales. Schedules can still affect optimization landscapes and numerical behavior in practice; these effects are not the focus of the present analysis, and we include a brief empirical illustration in Appendix~\ref{sec-dotbeta-0-improved-training} for completeness.

\subsubsection{Numerical illustrations}

We complement the theory with two experiments (Section~\ref{sec-numerical illustration}): a one-dimensional Ornstein--Uhlenbeck forecasting target, for which the path-space KL divergence between the exact and learned generative SDEs can be evaluated directly and the F\"ollmer choice is confirmed to dominate the alternatives by one to two orders of magnitude; and a two-dimensional Kolmogorov-flow data-assimilation task, where the F\"ollmer schedule yields the smallest posterior-mean error among the three diffusion coefficients tested.

\subsection{Related work}

\subsubsection{Generative diffusions and flow matching}

Diffusion models \cite{sohl2015deep,ho2020denoising} and score-based generative models \cite{song2020score} construct generative dynamics by reversing a noising process \cite{anderson1982reverse,haussmann1986time}, with the drift expressed in terms of score functions \cite{song2019generative}. Flow matching \cite{lipman2022flow}, rectified flows \cite{liu2022flow}, and stochastic interpolant methods \cite{albergo2022building,albergo2023stochastic} estimate velocity fields or stochastic dynamics that reproduce prescribed marginals along an interpolation path. The stochastic interpolant viewpoint \cite{albergo2023stochastic} highlights in particular the freedom to modify the diffusion coefficient \emph{a~posteriori} without changing the marginals; this is the starting point for our variational analysis.

Most works in this area use a Gaussian reference at the initial time. The point-source setting studied here is motivated by probabilistic forecasting and data assimilation \cite{chen2024probabilistic}, and has been adopted in applications such as weather prediction \cite{kossaifi2026demystifying}, cosmological simulation \cite{cuesta-lazaro2024joint}, and other scientific problems \cite{lim2024elucidating,mucke2025physics,sabti2025generative,chen2025flowdas,schiodt2025generative,yasuda2025probabilistic,horowitz2025baryonbridge}. Our contribution is to give a unified variational and analytic treatment of this regime, and to connect it to F\"ollmer processes. The variational optimality of F\"ollmer processes was stated informally in the appendix of our previous work \cite{chen2024probabilistic}; we formulate it rigorously here. The optimal diffusion coefficient identified is also related to the memoryless noise schedule used in stochastic optimal control for fine-tuning flow and diffusion generative models \cite{domingo2024adjoint}.

\subsubsection{F\"ollmer processes and Schr\"odinger bridges}

F\"ollmer \cite{follmer1986time} introduced the eponymous process by time-reversing a Wiener process conditioned on its terminal value, and showed that it minimizes path-space relative entropy subject to endpoint constraints. This fits into the broader framework of the Schr\"odinger bridge problem \cite{schrodinger1932theorie,leonard2014survey,chen2021stochastic}, in which one seeks the path measure closest in relative entropy to a given reference process, under prescribed marginal constraints. Classical analyses characterize the optimal drift through Hamilton--Jacobi--Bellman equations or coupled Schr\"odinger systems.

Our work is complementary to this literature. Rather than fixing the reference dynamics in advance and solving a control problem, we start from a parametric family of diffusions with prescribed marginals, indexed by the diffusion coefficient $g_t$. Optimizing a path-space KL divergence over this family leads to a F\"ollmer process, and simultaneously identifies the underlying reference diffusion in terms of the interpolation schedules. This leads to an alternative variational characterization of F\"ollmer processes, tailored to the statistical setting in which the drift must be estimated from data.

F\"ollmer processes and Schr\"odinger bridges (and optimal transport) have been studied and applied in a wide range of contexts, including sampling from distributions with intractable normalization constants \cite{zhang2021path,huang2021schrodinger,jiao2021convergence,vargas2023bayesian}, generative modeling and Bayesian inference \cite{tzen2019theoretical,wang2021deep,debortoli2021diffusion,liu20232,peluchetti2023non,shi2024diffusion,pooladian2025plug}, as well as connections to stochastic analysis, functional inequalities, and measure concentration \cite{lehec2013representation,eldan2018regularization,eldan2020stability,klartag2023spectral}. In particular, efficient algorithms motivated by diffusion models have been proposed for estimating Schr\"odinger bridges from data \cite{debortoli2021diffusion, shi2024diffusion, pooladian2025plug}. Our work characterizes the variational optimality of this bridge in generative diffusions in the point-mass, or F\"ollmer, setting.

\subsection{Organization}

Section~\ref{sec:Stochastic Interpolants} introduces stochastic interpolants with a point source and derives the baseline generative diffusion, including regularity properties of the drift.
Section~\ref{sec:optimal-diffusion-Follmer} develops the \emph{a~posteriori} tuning of the diffusion coefficient, identifies the KL-optimal choice, and establishes the connection with F\"ollmer processes, including conditional-expectation formulas for the F\"ollmer drift.
Section~\ref{sec:invariance} proves schedule invariance of the minimized path-space KL divergence.
Section~\ref{sec-numerical illustration} provides numerical illustrations on a one-dimensional Ornstein--Uhlenbeck forecasting target and on a two-dimensional Kolmogorov-flow data-assimilation problem.
Section~\ref{sec:conclusions} concludes.
Proofs of the main theorems are given in the body; longer technical arguments and additional generalizations are deferred to the appendices.

% ============================================================
\section{Stochastic Interpolants with a Point Source}
\label{sec:Stochastic Interpolants}
% ============================================================

We construct generative diffusions that transport a Dirac mass at the origin to a target distribution $\mu_\star$ on $\bR^d$. Taking the point source at zero entails no loss of generality and simplifies notation; the results extend immediately to an arbitrary point source by translation.

\subsection{Definitions and assumptions}
\label{sec:Notations and assumptions}

For the regularity analysis of the drift, we work under the following assumption on the target distribution.

\begin{assumption}
\label{assum:bounded-support-smoothed-assumption}
    The target $\mu_\star$ is the convolution of a compactly supported distribution with $\sfN(0,\eta^2 \mathrm{I})$ for some $\eta > 0$.
\end{assumption}

This assumption is standard in the theoretical analysis of diffusion models \cite{chen2023sampling,gao2023gaussian}. It is obtained by adding small Gaussian noise to a compactly supported distribution, which is a natural idealization of finite datasets.

\begin{remark}
\label{rem:compact-support}
All results in this paper remain valid rigorously for compactly supported targets $\mu_\star$ (without Gaussian smoothing) on the time interval $[0, 1-\delta]$ for any $\delta > 0$. Target distributions with exponential tails can also be handled, where additional properties of the regularity of the drift with connection to the choice of the interpolation schedule can be analyzed; see Appendix~\ref{sec:singular behaviors of baseline SDEs at initial time}.
\end{remark}

\begin{definition}
\label{def:def1}
The \emph{stochastic interpolant} $I_t$ is defined by
\begin{equation}
    \label{eq:stochinterpolant}
    I_t = \beta_t \, x_\star + \sigma_t W_t,\quad t\in[0,1],
\end{equation}
where $x_\star \sim \mu_\star$ and $(W_t)_{t\in[0,1]}$ is a standard Wiener process independent of $x_\star$, and the interpolation schedules satisfy:
\begin{itemize}
\item  $\beta,\sigma\in C^1([0,1])$, \quad $\dot\beta_t >0$, \quad $\dot\sigma_t<0$ \quad for all $t\in[0,1]$;
\item $\beta_0=0$, \quad $\beta_1=1$, \quad $\sigma_0>0$, \quad $\sigma_1=0$.
\end{itemize}
\end{definition}

By construction, $I_0 = 0$ and $I_1 = x_\star \sim \mu_\star$, so the interpolant bridges the Dirac mass at the origin and the target distribution.

\subsection{Baseline generative diffusion}
\label{sec:Stochastic interpolation between targets and Brownian motions}

We now give the explicit Markov-diffusion construction underlying the informal statement in Theorem~\ref{thm:1:b_informal}, together with the drift formula and its regularity.

\begin{theorem}
\label{thm:1:b}
Let $I_t$ be the stochastic interpolant of Definition~\ref{def:def1}. Define
\begin{equation}
\label{eq:b:def:app}
\begin{aligned}
    &\forall t \in (0,1]: \quad x_t  = \beta_t \, x_\star + \sigma_t \sqrt{t}\,z , \\
    &\forall (t,x) \in (0,1] \times \R^d : \quad b_t(x) = \mathbb{E}\!\left[\dot\beta_t \, x_\star + \dot \sigma_t  \sqrt{t} \,z\;\middle|\;x_t  = x\right],
\end{aligned}
\end{equation}
where the expectation is over $x_\star \sim \mu_\star$ and $z\sim \sfN(0,\mathrm{I})$, $z\perp x_\star$, conditional on $x_t = x$. Set also $b_0(0) = \dot\beta_0 \bE[x_\star]$. Then, under Assumption~\ref{assum:bounded-support-smoothed-assumption}, the SDE
\begin{equation}
    \label{eq:sde}
    {\rm d}X_t = b_t(X_t) \,{\rm d}t + \sigma_t \, {\rm d}W_t, \quad X_0 = 0
\end{equation}
satisfies $\mathrm{Law}(X_t) = \mathrm{Law}(I_t)$ for all $t \in [0,1]$, and in particular $X_1 \sim \mu_\star$. In addition, the drift $b_t(x)$ is continuous in $(t,x)$ on $(0,1] \times \bR^d$ and globally Lipschitz in $x$ uniformly over $t \in (0, 1]$.
\end{theorem}

\begin{remark}
     Since $W_t \stackrel{d}{=} \sqrt{t}\,z$ at each $t\in[0,1]$, the process $x_t = \beta_t x_\star + \sigma_t\sqrt{t}\,z$ has the same law at each time $t$ as the interpolant $I_t$. Its introduction is useful for designing simulation-free estimation methods, since sampling $x_t$ requires only drawing $x_\star$ and $z$, without simulating the Brownian motion.
\end{remark}

\begin{proof}[Proof of Theorem~\ref{thm:1:b}] 
Let $\mu_t(\rd x)$ denote the law of $I_t$, and let $\phi \in C^2_b(\R^d)$ be a test function. By definition,
\begin{equation}
    \label{eq:chracteristic}
    \forall t\in [0,1]: \quad \int_{\R^d} \phi(x) \,\mu_t({\rm d}x)  = \E[\phi(I_t)].
\end{equation}
Applying It\^o's formula to $\phi(I_t)$ gives
\begin{equation}
\label{eq:dcharact}
{\rm d} \phi( I_t) = (\dot \beta_t \, x_\star + \dot \sigma_t \, W_t) \cdot \nabla \phi( I_t) \,{\rm d}t + \tfrac12\sigma_t^2 \Delta \phi(I_t) \,{\rm d}t + \sigma_t\nabla \phi(I_t) \cdot {\rm d}W_t.
\end{equation}
Taking expectations and using It\^o's isometry, we obtain
\begin{equation}
\label{eq:dcharact:E}
\E[\phi( I_t)]  =  \phi(0) + \int_0^t \left(\E\big[(\dot \beta_r\, x_\star + \dot \sigma_r W_r) \cdot \nabla \phi(I_r)\big]  + \tfrac12\sigma^2_r \E[\Delta \phi(I_r)] \right) {\rm d}r.
\end{equation}
Since $I_t\stackrel{d}{=}x_t$ and $W_t \stackrel{d}{=} \sqrt{t}\,z$ at each $t\in[0,1]$, we may replace $I_r$ by $x_r$ and $W_r$ by $\sqrt{r}\,z$:
\begin{equation}
\label{eq:t:derv}
    \int_{\R^d} \phi(x) \,\mu_t({\rm d}x)  =  \phi(0)+ \int_0^t \left(\E\big[(\dot \beta_r \,x_\star + \dot \sigma_r \sqrt{r} \, z ) \cdot \nabla \phi(x_r)\big]+ \tfrac12\sigma^2_r \E[\Delta \phi(x_r)]\right) {\rm d}r.
\end{equation}
Conditioning on $x_r = x$ via the tower property and using the definition of $b_r$ in~\eqref{eq:b:def:app}, we arrive at
\begin{equation}
\label{eq:t:derv:2}
    \int_{\R^d} \phi( x) \,\mu_t({\rm d}x) = \phi(0) + \int_0^t \int_{\R^d}  \left(  b_r(x) \cdot \nabla \phi(x)  + \tfrac12 \sigma^2_r  \Delta\phi(x) \right) \mu(r,{\rm d}x) \,{\rm d}r.
\end{equation}
Repeating the same calculation for $\E[\phi(X_t)]$ where $X_t$ solves the SDE~\eqref{eq:sde} yields the same weak equation~\eqref{eq:t:derv:2}, showing that $\mathrm{Law}(X_t) = \mu_t(\cdot)$ for all $t$.

The Lipschitz bound on $b_t$ follows from the analysis in Appendix~\ref{appendix-Lipschitz regularity of the drift}.
\end{proof}

\subsection{Statistical estimation from data}
\label{sec:Learning generative diffusions from data}

Since $b_t(x)$ is defined as a conditional expectation \eqref{eq:b:def:app}, it admits a variational characterization as the minimizer of a square loss. For any candidate drift $\hat b$, we have the bias-variance decomposition
\begin{equation}
\begin{aligned}
    \int_0^1 \E \big[|\hat b_t(x_t) - b_t(x_t)|^2\big]\,{\rm d}t
     & = \int_0^1 \E \big[|\hat b_t(x_t) - ( \dot \beta_t \, x_\star + \dot\sigma_t  \sqrt{t} \,z)|^2\big] \,{\rm d}t \\
     & \quad - \int_0^1 \mathrm{Var}[\dot\beta_t x_\star + \dot \sigma_t  \sqrt{t} \,z \mid x_t] \,{\rm d}t.
\end{aligned}
\end{equation}
Since the variance term is independent of $\hat{b}$, estimating the drift reduces to minimizing
\begin{equation}
\label{eqn-loss-for-interpolants}
    L_b[\hat b] = \int_0^1 \E \big[|\hat b_t(x_t) - ( \dot \beta_t \, x_\star + \dot\sigma_t  \sqrt{t} \,z)|^2\big]  \,{\rm d}t.
\end{equation}
The factors $\dot\beta_t$ and $\dot\sigma_t\sqrt{t}$ appearing in the regression target are bounded on $[0,1]$ (since $\beta, \sigma \in C^1([0,1])$), so the loss~\eqref{eqn-loss-for-interpolants} is bounded and non-singular.

Given an estimated drift $\hat{b}_t$, the approximate generative diffusion is
\begin{equation}
    {\rm d}\hat{X}_t = \hat{b}_t(\hat{X}_t) \,{\rm d}t + \sigma_t \, {\rm d}W_t, \quad \hat{X}_0 = 0,
\end{equation}
and $\hat{X}_1$ approximately samples from $\mu_\star$. Any time-discretization scheme applied to this SDE yields a practical sampling algorithm.

Throughout the main article, we assume that the estimator $\hat{b}_t$ is globally Lipschitz in $x$ uniformly over $t \in (0, 1]$, matching the regularity of $b_t$ established in Theorem~\ref{thm:1:b}. This technical assumption is needed for our application of Girsanov's theorem later in this article.

% ============================================================
\section{Optimal Diffusion Coefficients and F\"ollmer Processes}
\label{sec:optimal-diffusion-Follmer}
% ============================================================

We now formalize the freedom in choosing $g_t$ already foreshadowed in \S\ref{sec:setup-main-results}, identify the KL-optimal diffusion coefficient, and show that the resulting process is a F\"ollmer process with respect to a reference diffusion determined by the interpolation schedules alone.

\subsection{A-posteriori tuning of diffusion coefficients}
\label{sec:A-posteriori tuning of diffusion coefficients}

Since $\sigma_t > 0$ on $(0,1)$, the law of $I_t$ is a convolution with a non-degenerate Gaussian and hence admits a density $\rho_t$ with respect to Lebesgue measure for each $t \in (0,1)$. This allows us to express the score $\nabla \log \rho_t$ and tune the diffusion coefficient as follows.

\begin{theorem}
\label{thm:tunning_g}
     Let $g\in C^0([0,1])$ satisfy $g_t > 0$ on $(0,1)$, and assume the limits $\lim_{t\to0^+} t^{-1}(g^2_t -\sigma^2_t)$ and $\lim_{t\to1^-} g^2_t \sigma^{-1}_t$ exist and are finite. Let $b_t$ be the drift from~\eqref{eq:b:def:app}, and define
\begin{equation}
    \label{eq:df:A:c}
    A_t = \bigl(t\sigma_t (\dot\beta_t \sigma_t -\beta_t  \dot\sigma_t )\bigr)^{-1} .
\end{equation} 
Then the score of the time-$t$ density $\rho_t$ of $X_t \stackrel{d}{=} I_t$ is given by
\begin{equation}
    \label{eq:score:app}
    \nabla \log \rho_t(x) = A_t \bigl(\beta_t b_t(x) - \dot \beta_t \, x \bigr), \qquad (t,x)\in (0,1)\times\R^d ,
\end{equation}
and the drift
    \begin{equation}
\label{eq:drift:bc-appendix:a}
        b^g_t(x) = b_t(x) + \tfrac12 (g^2_t-\sigma^2_t) A_t \bigl(\beta_t b_t(x) - \dot \beta_t \, x \bigr)
\end{equation}
is well-defined for all $(t,x)\in (0,1)\times\R^d$, with finite limits at $t=0$ and $t=1$. The SDE
\begin{equation}
    \label{eq:sde-appenfix}
    {\rm d}X^g_t= b_t^g(X^g_t) \,{\rm d}t + g_t \, {\rm d}W_t\, , \quad X^g_0 = 0 ,
\end{equation}
satisfies $\mathrm{Law}(X^g_t)= \mathrm{Law}(I_t)$ for all $t\in [0,1]$, and in particular $X^g_1\sim \mu_\star$.
\end{theorem}

Note that $A_t > 0$ for all $t \in (0,1)$: since $\beta_0 = 0$, $\beta_1 = 1$, $\sigma_1 = 0$, $\dot\beta_t > 0$, and $\dot\sigma_t < 0$, we have $\dot\beta_t \sigma_t - \beta_t \dot\sigma_t > 0$ on $(0,1)$, and $t\sigma_t > 0$ on the same interval.

\begin{proof}[Proof of Theorem \ref{thm:tunning_g}]
 Since $x_t = \beta_t x_\star + \sqrt{t}\,\sigma_t z$ with $x_t \stackrel{d}{=} I_t$, we have
\begin{equation}
\label{eq:stein}
\nabla \log \rho_t(x) = -\frac{1}{\sqrt{t}\,\sigma_t}\,\mathbb{E}[z\mid x_t=x], \qquad (t,x) \in (0,1) \times\R^d .
\end{equation}
This identity connecting the denoising conditional expectation to the score function is known as Tweedie's formula \cite{efron2011tweedie} and can be derived from Stein's identity.

From the definitions of $b_t$ and $x_t$,
\begin{equation}
\begin{aligned}
    b_t(x) &= \dot{\beta}_t\,\mathbb{E}[x_\star\mid x_t =x] + \sqrt{t}\, \dot{\sigma}_t\,\mathbb{E}[z\mid x_t =x] ,\\
    x &= \beta_t\, \mathbb{E}[x_\star\mid x_t =x] + \sqrt{t}\, \sigma_t \,\mathbb{E}[z\mid x_t =x] ,
\end{aligned}
\end{equation}
from which we solve for $\mathbb{E}[z\mid x_t = x]$ and insert into~\eqref{eq:stein} to obtain~\eqref{eq:score:app}.

To verify that~\eqref{eq:sde-appenfix} preserves the marginals, observe that the Fokker--Planck equations for $b_t^g$ and $b_t$ coincide:
\[
\tfrac12 g_t^2 \Delta \rho_t - \nabla\cdot(b_t^g \rho_t) = \tfrac12 \sigma_t^2 \Delta \rho_t - \nabla\cdot(b_t \rho_t) ,
\]
since $\nabla\cdot(\rho_t \nabla\log\rho_t) = \Delta\rho_t$.

It remains to verify that $b_t^g$ is well-defined at the boundaries. Writing~\eqref{eq:drift:bc-appendix:a} explicitly:
\begin{equation}
\label{eq:bgs:full}
    b^g_t(x) = b_t(x) + \frac{g_t^2-\sigma_t^2}{2t\sigma_t}\cdot\frac{\beta_t b_t(x) -\dot{\beta}_t x}{\sigma_t\dot{\beta}_t-\dot{\sigma}_t\beta_t} .
\end{equation}
The ratios $\beta_t(\sigma_t\dot{\beta}_t-\dot{\sigma}_t\beta_t)^{-1}$ and $\dot\beta_t(\sigma_t\dot{\beta}_t-\dot{\sigma}_t\beta_t)^{-1}$ have finite limits at $t=0$ and $t=1$ by the assumptions on $\beta$ and $\sigma$. The only potentially singular factor is $(g_t^2-\sigma_t^2)/(2t\sigma_t)$, which has a $t^{-1}$ singularity at $t=0$ and a $\sigma_t^{-1}$ singularity at $t=1$. Both are removed by the hypotheses on $g$. More general $g$ are discussed in Remark \ref{remark-assumption-g-beta}.
\end{proof}

Theorem~\ref{thm:tunning_g} yields a family of generative SDEs parametrized by $g$. With the baseline drift $b_t$ from Theorem~\ref{thm:1:b}, the exact process is
\begin{equation}
\label{eq:sde-tuning-g-exact}
    {\rm d}X^g_t= \bigl(1+\tfrac12 \beta_t A_t (g_t^2-\sigma_t^2) \bigr) b_t(X^g_t) \,{\rm d}t - \tfrac12(g_t^2-\sigma_t^2) A_t \dot\beta_t X_t^g \,{\rm d}t + g_t \, {\rm d}W_t, \quad X^g_0 = 0 ,
\end{equation}
and the corresponding estimated process, using $\hat b_t$ in place of $b_t$, is
\begin{equation}
    \label{eq:sde-tuning-g-approx}
    {\rm d}\hat X^g_t= \bigl(1+\tfrac12 \beta_t A_t (g_t^2-\sigma_t^2) \bigr) \hat b_t(\hat X^g_t) \,{\rm d}t - \tfrac12(g_t^2-\sigma_t^2)A_t \dot\beta_t \hat X_t^g  \,{\rm d}t + g_t  \,{\rm d}W_t, \quad \hat X^g_0 = 0 .
\end{equation}
The estimated process can be used for generation without re-estimating the drift: the diffusion coefficient is modified \textit{at inference time}.

\subsection{KL-optimal diffusion coefficient}
\label{sec:Optimizing the KL statistical estimation error}

Given the family~\eqref{eq:sde-tuning-g-approx}, we seek the diffusion coefficient that minimizes the path-space KL divergence $\mathrm{KL}(\mathbb{P}_{X^g}\|\mathbb{P}_{\hat X^g})$ between the exact and estimated generative processes.

\begin{theorem}
\label{prop-min-KL}
    For each $\delta > 0$, $\mathrm{KL}(\mathbb{P}_{X^g}|_{[0,1-\delta]}\|\mathbb{P}_{\hat X^g}|_{[0,1-\delta]})$ is minimized over $g$ by 
    \begin{equation}
    \label{eqn-follmer-gt}
        g_t^{\rm F} = \left|2t\sigma^2_t \frac{\rm d}{{\rm d}t}\log\frac{\beta_t }{ \sqrt{t}\,\sigma_t }\right|^{1/2}.
    \end{equation}
    Moreover, $g_t^{\rm F}$ extends continuously to $[0,1]$ with $g_1^{\rm F} = 0$, and satisfies the boundary limit assumptions in Theorem~\ref{thm:tunning_g} provided $\dot\beta_0 > 0$ and $\ddot\beta_0$ exists.
\end{theorem}
\begin{proof}[Proof of Theorem \ref{prop-min-KL}]
On $[0,1-\delta]$, the diffusion coefficient $g_t$ is bounded away from zero. We can apply Girsanov's theorem using a variant of Novikov's condition \cite[Corollary 5.14]{karatzas2014brownian}, which is satisfied because $X_t \stackrel{d}{=} I_t$ has Gaussian tails and $b_t, \hat{b}_t$ are spatially Lipschitz. This yields
\begin{equation}
\label{eq:kl:a2}
    \mathrm{KL} (\mathbb{P}_{X^g}|_{[0,1-\delta]}\|\mathbb{P}_{\hat X^g}|_{[0,1-\delta]})  = \frac12 \int_0^{1-\delta} g_t^{-2} \bigl|1+\tfrac12 \beta_t  A_t(g_t^2-\sigma_t^2)\bigr|^2 L_t \,{\rm d}t ,
\end{equation}
where $L_t = \E[|\hat b_t(I_t) - b_t(I_t)|^2]$ is independent of $g$, using $X_t^g \stackrel{d}{=} I_t$. The minimization reduces to minimizing pointwise in $t$:
\begin{equation}
    \label{eq:min:int:a}
    \psi(g_t^2) := g_t^{-2} \bigl|\gamma_t + \alpha_t g_t^2\bigr|^2, \quad \alpha_t = \tfrac12\beta_t A_t > 0, \quad \gamma_t = 1 - \alpha_t\sigma_t^2 .
\end{equation}
If $\gamma_t \le 0$, the minimum $\psi = 0$ is attained at $g_t^2 = |\gamma_t|/\alpha_t$. If $\gamma_t > 0$, the minimum $\psi = 4\alpha_t\gamma_t$ is attained at $g_t^2 = \gamma_t/\alpha_t$. In both cases the minimizer is $g_t^2 = |\gamma_t|/\alpha_t$. Substituting the expressions for $\alpha_t$ and $\gamma_t$ and using the identity
\begin{equation}
\label{eq:Ds:2:app}
2t\sigma_t(\beta_t^{-1}\dot\beta_t \sigma_t- \dot\sigma_t) -\sigma_t^2 = 2t\sigma^2_t \frac{\rm d}{{\rm d}t}\log\frac{\beta_t }{ \sqrt{t}\,\sigma_t },
\end{equation}
we obtain~\eqref{eqn-follmer-gt}. Since the pointwise minimizer is independent of $\delta$, it is the unique minimizer for every truncation. The continuity of $g_t^{\rm F}$ on $[0,1]$ and the limit $g_1^{\rm F} = 0$ follow from $\sigma_1 = 0$ and the smoothness of the interpolation schedules.

It remains to show that this $g_t$ satisfies the boundary limit assumptions in Theorem \ref{thm:tunning_g} so that it is a valid choice.
The two conditions are that the limits $\lim_{t\to0^+} t^{-1}(g^2_t -\sigma^2_t)$ and $\lim_{t\to1^-} g^2_t \sigma^{-1}_t$ exist. The second condition is always satisfied since
\begin{equation}
    \label{eq:gF:s1}
    \lim_{t\to1^-} |g^\Fo_t|^2\sigma_t^{-1} = \lim_{t\to1^-} |2t (\beta_t^{-1} \dot\beta_t \sigma_t -\dot\sigma_t) -\sigma_t| = |2\dot\sigma_1| < \infty,
\end{equation}
where we used the fact $\sigma_1=0$. For the first condition, we have
\begin{equation*}
    \lim_{t\to0^+} t^{-1} (|g^\Fo_t|^2-\sigma^2_t) = \lim_{t\to0^+} (2\sigma_t (\beta_t^{-1} \dot\beta_t \sigma_t -\dot\sigma_t) -2t^{-1}\sigma^2_t) = 2\lim_{t\to0^+} (\beta_t^{-1} \dot\beta_t -t^{-1})-2\dot\sigma_0\sigma_0 
\end{equation*}
since $\sigma_0>0$ and $\dot\sigma_0<0$. Since $\beta_0=0$, if $\dot\beta_0>0$ and $\ddot\beta_0$ exists, we have
\begin{equation}
    \label{eq:lim}
    \lim_{t\to0^+} (\beta_t^{-1} \dot\beta_t -t^{-1}) = \tfrac12\ddot\beta_0 \dot\beta_0^{-1}\, .
\end{equation}
This implies that
\begin{equation}
    \label{eq:gF:s0a}
    \lim_{t\to0^+} t^{-1} (|g^\Fo_t|^2-\sigma^2_t) =\sigma_0^2 \ddot\beta_0 \dot\beta_0^{-1}-2\dot\sigma_0\sigma_0, \quad \text{if $\dot\beta_0>0$ and $\ddot\beta_0$ exists.}
\end{equation}
Thus, the conditions are verified.
\end{proof}

\begin{remark}
\label{rem:full-interval-KL}
Since the integrand in~\eqref{eq:kl:a2} is non-negative, monotone convergence gives
\[
    \mathrm{KL}(\mathbb{P}_{X^g}\|\mathbb{P}_{\hat X^g}) = \lim_{\delta \to 0^+} \mathrm{KL}(\mathbb{P}_{X^g}|_{[0,1-\delta]}\|\mathbb{P}_{\hat X^g}|_{[0,1-\delta]}).
\]
At the optimizer $g_t = g_t^{\rm F}$, the integrand reduces to $4\alpha_t\gamma_t \, L_t$ when $\gamma_t > 0$ and vanishes when $\gamma_t \le 0$. The limit is finite provided $\int_0^{1} \alpha_t\gamma_t \, L_t \,{\rm d}t < \infty$, and $g_t^{\rm F}$ remains the minimizer of the full-interval divergence.
\end{remark}

\begin{remark}
    The path-space KL divergence serves as a surrogate for the terminal generation error: by the data-processing inequality, $\mathrm{KL}(\mathbb{P}_{X^g}\|\mathbb{P}_{\hat X^g}) \ge \mathrm{KL}(\mathrm{Law}(X_1^g)\|\mathrm{Law}(\hat X_1^g))$. Thus $g^{\rm F}$ minimizes an upper bound on the generation error at $t = 1$.
\end{remark}

\begin{remark}
\label{remark-assumption-g-beta}
    We note that Definition~\ref{def:def1} assumes $\dot\beta_t > 0$ for all $t\in [0,1]$. If $\dot\beta_0 = 0$, all preceding results remain valid, except that the optimal $g^\Fo$ in Theorem \ref{prop-min-KL} now yields $g_{0}^\Fo\neq \sigma_0$, and $\lim_{t\to0^+} t^{-1} (|g^\Fo_t|^2-\sigma^2_t)$ does not exist. Thus, the boundary limit assumptions in Theorem~\ref{thm:tunning_g} fail to hold. In this case, the drift $b^{g^\Fo}_t$ exhibits a singularity as $t\to 0^+$, which requires more careful treatment when defining the solution to the SDE~\eqref{eq:sde-tuning-g-approx}. The solution remains well-defined because the singularity corresponds to a negative restoring force; see Appendix~\ref{appendix-singular-drift-discussion} for technical details.
\end{remark}

\subsection{Variational characterization of F\"ollmer processes}
\label{sec:Follmer-characterization}

We now show that the KL-optimal diffusion coefficient identifies a F\"ollmer process. We begin by recalling the variational definition of a F\"ollmer process associated with a linear reference SDE: given a reference process $(Y_t)_{t\in[0,1]}$ solving
\begin{equation}
    \label{eq:ref:foll-gen}
    {\rm d}Y_t = a_t Y_t \,{\rm d}t + g_t \,{\rm d}W_t, \quad Y_0 = 0,
\end{equation}
the F\"ollmer process associated with $Y$ and the target $\mu_\star$ is the diffusion $X^{\rm F}$ with $X^{\rm F}_0 = 0$ and $X^{\rm F}_1 \sim \mu_\star$ that minimizes the path-space relative entropy $\mathrm{KL}(\bP_{X^{\rm F}} \| \bP_Y)$. This is a special case of the Schr\"odinger bridge problem \cite{schrodinger1932theorie,leonard2014survey,chen2021stochastic} with a degenerate initial marginal, and the solution can be  characterized through the disintegration identity
\begin{equation}
\label{eq:kl:foll:disintegration}
         \mathrm{KL}[\P_{\check X}\Vert \P_Y] = \int_{\R^d} \mathrm{KL}[\P^{\check X_1=x}_{\check X}\Vert \P^{Y_1=x}_Y] \mu_\star({\rm d}x) + \mathrm{KL}[\mu_\star \Vert \mathrm{Law}(Y_1)],
\end{equation}
valid for any process $\check X$ with $\check X_0 = 0$ and $\check X_1 \sim \mu_\star$. The minimum is achieved by the process whose bridges coincide with those of $Y$; see \cite{follmer1986time,leonard2014survey} and Appendix~\ref{appendix-Schrodinger's bridges} for further discussion.

The following theorem shows that the KL-optimal generative diffusion from Theorem~\ref{prop-min-KL} is the F\"ollmer process for a specific reference process determined by the interpolation schedules.

\begin{theorem}
    \label{th:foll:app}
    Assume that $\beta_t /(\sqrt{t}\,\sigma_t)$ is non-decreasing on $[0,1]$.
    Then the process $X^{\rm F} \equiv X^{g^{\rm F}}$ solving~\eqref{eq:sde-appenfix} with $g_t = g^{\rm F}_t$ is the F\"ollmer process associated with the reference process $Y$ solving~\eqref{eq:ref:foll-gen} with $g_t = g^{\rm F}_t$ and 
\begin{equation}
\label{eqn-a_t}
    a_t = \frac{\rm d}{{\rm d}t} \log \frac{\beta_t^2 + t\sigma_t^2}{\beta_t}.
\end{equation}
That is, among all diffusions $\check X$ with $\check X_0 = 0$ and $\check X_1 \sim \mu_\star$, the process $X^{\rm F}$ uniquely minimizes $\mathrm{KL}(\bP_{\check X} \| \bP_Y)$.
\end{theorem}

The monotonicity assumption on $\beta_t /(\sqrt{t}\,\sigma_t)$ ensures that the KL-optimal diffusion coefficient identified in Theorem~\ref{prop-min-KL} eliminates the score term under time reversal, allowing the optimal process to be identified as a F\"ollmer process.

The reference process $Y$ depends on the interpolation schedules $(\beta_t, \sigma_t)$ but not on $\mu_\star$, which enters only through the terminal constraint. Theorem~\ref{th:foll:app} thus provides a new variational characterization of F\"ollmer processes: rather than fixing a reference process \textit{a priori} (as in the classical formulation \cite{follmer1986time,lehec2013representation}) and deriving optimal controls, we start from a parametric family of generative diffusions sharing the same marginal laws and optimize the diffusion coefficient. The F\"ollmer process and its reference emerge simultaneously from this optimization.

\begin{proof}[Proof of Theorem \ref{th:foll:app}]
We follow F\"ollmer's original strategy of time reversal \cite{follmer1986time}; see also \cite{haussmann1986time,cattiaux2023time}.

\medskip
\noindent\textit{Step 1: Forward SDE in terms of the score.}
Solving~\eqref{eq:score:app} for $b_t$ and substituting into~\eqref{eq:drift:bc-appendix:a}, the SDE~\eqref{eq:sde-appenfix} takes the form
\begin{equation}
    \label{eq:sde-appenfix:2}
    {\rm d}X^g_t= \left(\beta_t^{-1} A^{-1}_t +\tfrac12 (g_t^2-\sigma_t^2)\right)  \nabla \log \rho_t(X^g_t) \,{\rm d}t + \beta_t^{-1} \dot\beta_t X_t^g  \,{\rm d}t  + g_t \, {\rm d}W_t,
\end{equation}
with $X^g_0 = 0$.

\medskip
\noindent\textit{Step 2: Time reversal and decoupling from the score.}
Time-reversing~\eqref{eq:sde-appenfix:2} yields the SDE for $X_{t}^\rev \overset{d}{=} X^g_{1-t}$:
\begin{equation}
    \label{eq:sde-appenfix:3}
    \begin{aligned}
        {\rm d}X^\rev_t& = -\left(\beta_{1-t}^{-1} A^{-1}_{1-t} -\tfrac12 (g_{1-t}^2+\sigma_{1-t}^2)\right)  \nabla \log \rho_{1-t}(X^\rev_t) \,{\rm d}t \\
        & \quad - \beta_{1-t}^{-1} \dot\beta_{1-t} X^\rev_t \,{\rm d}t  + g_{1-t}  \,{\rm d}W_t.
    \end{aligned}
\end{equation}
We note that
\begin{equation*}
g^\Fo_t = \left|2t\sigma^2_t \frac{\rm d}{{\rm d}t}\log\frac{\beta_t }{ \sqrt{t}\,\sigma_t }\right|^{1/2} = \left|\frac{1-\frac12 \beta_t  A_t\sigma^2_t}{\frac12 \beta_t  A_t}\right|^{1/2}\, .
\end{equation*}
Setting $g_t=g^{\rm F}_t$ with $1-\frac12 \beta_t  A_t\sigma^2_t\ge 0$ (guaranteed by $\beta_t /(\sqrt{t}\,\sigma_t)$ non-decreasing), the score term vanishes and~\eqref{eq:sde-appenfix:3} reduces to
\begin{equation}
    \label{eq:sde-appenfix-}
    {\rm d}X^\rev_t= -\beta_{1-t}^{-1} \dot\beta_{1-t} X^\rev_t \,{\rm d}t  + g^{\rm F}_{1-t}  \,{\rm d}W_t, \quad X^\rev_0  \sim \mu_\star.
\end{equation}
The drift is now independent of $\nabla \log \rho_t$: the target $\mu_\star$ enters only through the initial condition.

\medskip
\noindent\textit{Step 3: Construction of the reference process.}
Replacing the initial condition in~\eqref{eq:sde-appenfix-} by $Y^\rev_0  \sim \sfN(0,\mathrm{I})$ gives
\begin{equation}
    \label{eq:sde-appenfix:5}
    {\rm d}Y^\rev_t= -\beta_{1-t}^{-1}\dot\beta_{1-t} Y^\rev_t \,{\rm d}t  + g^{\rm F}_{1-t}  \,{\rm d}W_t, \quad Y^\rev_0 \sim \sfN(0,\mathrm{I}).
\end{equation} 
Since ${\rm d}\beta_{1-t} = - \dot\beta_{1-t} \,{\rm d}t$, we can write ${\rm d}(\beta^{-1}_{1-t} Y^\rev_t) = \beta_{1-t}^{-1} g^{\rm F}_{1-t}  \,{\rm d}W_t$, yielding the explicit solution
\begin{equation}
    \label{eq:sde-appenfix:7}
    Y^\rev_t = \beta_{1-t} Y^\rev_0 + \beta_{1-t} \int_0^t \beta^{-1}_{1-u}g^{\rm F}_{1-u}  \,{\rm d}W_u.
\end{equation} 
A direct calculation using the explicit form of $|g^{\rm F}_t|^2 = 2t\sigma^2_t \frac{\rm d}{{\rm d}t}\log\frac{\beta_t }{ \sqrt{t}\,\sigma_t }$ shows that
\begin{equation}
    \label{eq:sde-appenfix:8}
        \E \left[\left|\beta_{1-t} \int_0^t \beta^{-1}_{1-u}g^{\rm F}_{1-u}  \,{\rm d}W_u\right|^2\right] = d\beta^2_{1-t} \int_0^t \beta^{-2}_{1-u}|g^{\rm F}_{1-u}|^2 \,{\rm d}u = d(1-t)\sigma^2_{1-t},
\end{equation}
where the last equality follows from the change of variables $u \mapsto 1-u$ and the identity $\int_{1-t}^1 \frac{2u\sigma^2_{u}}{\beta^2_{u}} \frac{\rm d}{{\rm d}u} 
    \log \frac{\beta_{u}}{\sqrt{u} \sigma_{u}} \,{\rm d}u = -\int_{1-t}^1 {\rm d}\left(\frac{u\sigma^2_{u}}{\beta^2_{u}}\right)$, together with $\sigma_1 = 0$. Therefore,
\begin{equation}
\label{eq:Y:ref}
Y^\rev_{1-t}\stackrel{d}{=}  \beta_t z + \sigma_t W_t \sim \sfN(0, (\beta_t^2+ t\sigma_t^2)\mathrm{I}),
\end{equation}
where $z\sim \sfN(0,\mathrm{I})$ with $z\perp W$. This is the stochastic interpolant~\eqref{eq:stochinterpolant} with $x_\star$ replaced by a standard Gaussian.

\medskip
\noindent\textit{Step 4: Identification of the reference SDE.}
From~\eqref{eq:Y:ref}, the score of $Y_t \stackrel{d}{=} Y^\rev_{1-t}$ is $\nabla \log \rho_t^{Y}(y) = -y/(\beta_t^2+ t\sigma_t^2)$.
Time-reversing~\eqref{eq:sde-appenfix:5},  we get \begin{equation}
    \label{eq:Y:ref:sde}
    {\rm d}Y_t = \beta_{t}^{-1}\dot\beta_t Y_t {\rm d}t + |g_t^\Fo|^2 \nabla \log \rho_t^{Y}(Y_t) {\rm d}t+ g^\Fo_t  {\rm d}W_t, \quad Y_{0}  = 0,
\end{equation}
and by inserting the explicit forms of $g^{\rm F}_t$ and $\nabla \log \rho_t^Y$, we finally get
\begin{equation}
    {\rm d}Y_t = a_t Y_t\,{\rm d}t + g^{\rm F}_t \,{\rm d}W_t, \quad Y_0 = 0,
\end{equation}
with $a_t = \frac{\rm d}{{\rm d}t} \log \frac{\beta_t^2 + t\sigma_t^2}{\beta_t}$, establishing~\eqref{eqn-a_t}.

\medskip
\noindent\textit{Step 5: KL optimality via disintegration.}
Since KL divergence is invariant under time reversal, $\mathrm{KL}(\bP_{\check X} \| \bP_Y) = \mathrm{KL}(\bP_{\check X^\rev} \| \bP_{Y^\rev})$. Applying the disintegration~\eqref{eq:kl:foll:disintegration} to the reversed processes gives
\begin{equation}
\label{eq:kl:foll:gen:interpolants}
         \mathrm{KL} (\mathbb{P}_{\check X^{\rev}}\| \mathbb{P}_{Y^{\rev}}) =  \int_{\R^d} \mathrm{KL} (\mathbb{P}_{\check X^{\rev}}^{\check X^{\rev}_1=x}\| \mathbb{P}_{Y^{\rev}}^{Y^{\rev}_1 = x}) \mu_\star({\rm d}x) + \mathrm{KL}[\mu_\star \Vert \sfN(0,\mathrm{I})].
\end{equation}
The second term is fixed by the constraint $\check X_1 \sim \mu_\star$. The first term is non-negative and vanishes for $\check X^\rev = X^\rev$ (the time-reversal of $X^{g^{\rm F}}$), since by~\eqref{eq:sde-appenfix-} and~\eqref{eq:sde-appenfix:5} the processes $X^\rev$ and $Y^\rev$ share the same dynamics---they differ only in their initial conditions. Therefore $X^{g^{\rm F}}$ minimizes $\mathrm{KL}(\bP_{\check X} \| \bP_Y)$ over all $\check X$ with $\check X_0 = 0$ and $\check X_1 \sim \mu_\star$.
\end{proof}

\subsection{Explicit formulas for specific schedules}
\label{sec:explicit-schedules}

We specialize the general formulas to two choices of interpolation schedules.

\subsubsection{Linear-linear schedule: $\beta_t = t$, $\sigma_t = 1-t$}

With $\dot\beta_t = 1$ and $\dot\sigma_t = -1$, the key quantities are:
\begin{align*}
    \dot\beta_t \sigma_t - \beta_t \dot\sigma_t = 1, \quad
    A_t = \frac{1}{t(1-t)}.
\end{align*}
The KL-optimal diffusion coefficient is
\begin{equation}
    g_t^{\rm F} = \sqrt{1-t^2},
\end{equation}
and the F\"ollmer drift takes the simple form
\begin{equation}
\label{eq:follmer-drift-linear}
    b^{g^{\rm F}}_t(x) = (1+t)\, b_t(x) - x.
\end{equation}
Since $b_t(x) = \bE[x_\star - \sqrt{t}\, z \mid x_t = x]$, the F\"ollmer drift admits the conditional expectation representation
\begin{equation}
\label{eq:follmer-drift-linear-cond}
    b^{g^{\rm F}}_t(x) = \bE[x_\star - 2\sqrt{t}\, z \mid x_t = x].
\end{equation}
Both are well defined and non-singular.

\subsubsection{Linear-square root schedule: $\beta_t = t$, $\sigma_t = \sqrt{1-t}$}

This schedule serves as a sanity check: it is the unique schedule for which the reference process $Y$ in Theorem~\ref{th:foll:app} is standard Brownian motion, recovering the classical F\"ollmer construction. Note that $\sigma_t = \sqrt{1-t}$ is not differentiable at $t = 1$, violating our standing assumption $\sigma \in C^1([0,1])$. Nevertheless, the formulas remain valid on $[0,1)$ and extend continuously to $t = 1$.

With $\dot\beta_t = 1$ and $\dot\sigma_t = -\frac{1}{2\sqrt{1-t}}$ for $t \in [0,1)$, we have
\begin{align*}
    \dot\beta_t \sigma_t - \beta_t \dot\sigma_t = \frac{2-t}{2\sqrt{1-t}}, \quad A_t = \frac{2}{t(2-t)}.
\end{align*}
The KL-optimal diffusion coefficient is $g_t^{\rm F} = 1$,
recovering the classical F\"ollmer process with Brownian diffusion as reference dynamics. The F\"ollmer drift satisfies
\begin{equation}
\label{eq:follmer-drift-sqrt}
    b^{g^{\rm F}}_t(x) = \frac{2\, b_t(x) - x}{2-t}.
\end{equation}

We have the $b_t(x) = \bE[x_\star - \frac{\sqrt{t}}{2\sqrt{1-t}}\, z \mid x_t = x]$ and the F\"ollmer drift admits the conditional expectation representation
\begin{equation}
\label{eq:drift:f}
    b^{g^{\rm F}}_t(x) = \bE\Big[x_\star - \frac{\sqrt{t}}{\sqrt{1-t}}\, z \;\Big|\; x_t = x\Big].
\end{equation}
It is well-defined and bounded as $t\to 0^+$.

The coefficient $\sqrt{t/(1-t)}$ blows up as $t \to 1^-$, but the drift remains bounded. To see this, using $x_t = t x_\star + \sqrt{t(1-t)}\, z$, we can write
\[
\bE[x_\star \mid x_t = x] = \frac{x - \sqrt{t(1-t)}\,\bE[z \mid x_t = x]}{t},
\]
so that
\[
b^{g^{\rm F}}_t(x) = \frac{x}{t} - \left(\frac{\sqrt{1-t}}{\sqrt{t}} + \frac{\sqrt{t}}{\sqrt{1-t}}\right)\bE[z \mid x_t = x] = \frac{x}{t} - \frac{1}{\sqrt{t(1-t)}}\,\bE[z \mid x_t = x].
\]
Using the fact $\nabla \log \rho_t(x) = -\bE[z \mid x_t = x]/\sqrt{t(1-t)}$, we have
\begin{equation}
\label{eq:drift:f:simplified}
    b^{g^{\rm F}}_t(x) = \frac{x}{t} + \nabla \log \rho_t(x).
\end{equation}
Under Assumption~\ref{assum:bounded-support-smoothed-assumption}, the score $\nabla \log \rho_t(x)$ remains non-singular as $t \to 1^-$, and we obtain
\begin{equation}
    \lim_{t \to 1^-} b^{g^{\rm F}}_t(x) = x + \nabla \log \rho^\star(x),
\end{equation}
where $\rho^\star$ is the density of $\mu_\star$. The F\"ollmer drift is non-singular on $[0,1]$.

\subsection{F\"ollmer drifts as conditional expectations}
\label{sec:Follmer-drifts-conditional}
The baseline drift~\eqref{eq:b:def:app} is a conditional expectation, which lets it be regressed from samples via the simulation-free loss~\eqref{eqn-loss-for-interpolants}. The two examples above show that the F\"ollmer drift admits an analogous representation. We now show that this is general: the F\"ollmer drift associated with any reference process of the form~\eqref{eq:ref:foll-gen} can be written as a conditional expectation, so that it too can be estimated by square-loss regression without simulating any SDE.

Consider the reference process~\eqref{eq:ref:foll-gen} and define
\begin{equation}
\label{eq:rth-def}
    r_t = \exp\left(\int_0^t a_u \,{\rm d}u\right), \quad h_t = \frac{1}{\eps}\int_0^t \frac{r_t^2}{r_u^2} g_u^2 \,{\rm d}u, \quad \eps = \int_0^1 \frac{r_1^2}{r_u^2} g_u^2 \,{\rm d}u > 0.
\end{equation}
We have the following result. Here we use $X^{\rm F}$ to denote a F\"ollmer process.
\begin{proposition}
\label{prop-Follmer-SDE}
    The F\"ollmer process $X^{\rm F}$ associated with the reference process $Y$ solving \eqref{eq:ref:foll-gen} and the target measure $\mu_\star$ satisfies the SDE
    \begin{equation}
    \label{eq:Follmer-SDE-cond-exp}
            {\rm d}X_t^{\rm F} = a_t X^{\rm F}_t \,{\rm d}t + g_t^2 \nabla \log P^Y_{1-t}f(X_t^{\rm F}) \,{\rm d}t +  g_t\,{\rm d}W_t, \quad X_0^{\rm F} = 0,
    \end{equation}
    where 
    $P_t^Y f(x) = \mathbb{E}[f(Y_1)|Y_{1-t}=x]$ and $f = {\rm d}\mu_\star/{\rm d}\mathrm{Law}(Y_1)$. Moreover, $X_t^{\rm F}$ has the same law at each $t$ as the stochastic interpolant
    \begin{equation}
    \label{eq:Follmer-interpolant}
    x_t = h_t x_\star + \sqrt{\eps h_t(1-h_t)}\,z, \quad x_\star \sim \mu_\star, \quad z \sim \sfN(0,\mathrm{I}), \quad z \perp x_\star,
    \end{equation}
    and the above F\"ollmer drift admits the conditional-expectation representation
    \begin{equation}
    \label{eq:Follmer-drift-cond}
    b^{\rm F}_t(x) = a_t x + g_t^2\,\E\!\left[\frac{1}{\eps}x_\star-\sqrt{\frac{h_t}{\eps(1-h_t)}}\,z \;\Big|\; x_t = x\right].
    \end{equation}
\end{proposition}

The proof is given in Appendix~\ref{appendix-sec:Follmer drifts and conditional expectations}. By Proposition~\ref{prop-Follmer-SDE}, the F\"ollmer drift can be estimated directly from data via square-loss regression. Analogously to~\eqref{eqn-loss-for-interpolants}, one minimizes
\begin{equation}
\label{eqn-loss-for-follmer}
    L_u[\hat u] = \int_0^1 \E \left[\left|\hat u_t(x_t) - \left( \frac{g_t^2}{\eps}x_\star-g_t^2\sqrt{\frac{h_t}{\eps(1-h_t)}}\,z\right)\right|^2\right]  {\rm d}t,
\end{equation}
and the estimated F\"ollmer process is ${\rm d}\hat{X}^{\rm F}_t = a_t \hat{X}^{\rm F}_t \,{\rm d}t + \hat{u}_t(\hat{X}^{\rm F}_t)\,{\rm d}t +  g_t\,{\rm d}W_t$, $\hat{X}^{\rm F}_0 = 0$.

Unlike the baseline loss~\eqref{eqn-loss-for-interpolants}, the regression target in~\eqref{eqn-loss-for-follmer} may have unbounded variance as $t \to 1$, since the factor $g_t^2/\sqrt{1-h_t}$ diverges unless $g_t/\sqrt{1-t}$ remains bounded. This is a manifestation of the terminal singularity common to all F\"ollmer-type constructions.

% ============================================================
\section{Schedule Invariance under Path-Space KL Divergence}
\label{sec:invariance}
% ============================================================

Since all schedules $(\beta_t, \sigma_t)$ lead to generative diffusions involving the score of $\mu_\star$ at varying noise levels, a natural question is whether certain schedules yield better statistical efficiency. The following result shows that, after KL-optimal tuning of the diffusion coefficient, the answer is no: the minimized path-space KL divergence depends only on the quality of the score estimates, not on the schedule.

\begin{theorem}
\label{thm:schedule-invariance}
    After optimizing the diffusion coefficient as in Theorem~\ref{prop-min-KL}, the minimized path-space KL divergence is
    \begin{equation}
    \label{eqn-KL-invariant}
    \mathrm{KL}^\star = 2 \int_0^\infty r \, \bE\!\left[|\nabla \log q_{r} (x_\star + rz) - \hat{s}_{r}(x_\star+rz)|^2\right] {\rm d}r,
    \end{equation}
    where $q_r = \mu_\star * \sfN(0, r^2\mathrm{I})$ is the density of $x_\star + rz$, and $\hat{s}_r$ is the estimated score at noise level~$r$ defined through the estimated drift $\hat{b}_t$ via~\eqref{eqn-b-t-as-scores-approx} below. In particular, assuming the same accuracy of the estimated score $\hat s_r$, $\mathrm{KL}^\star$ is independent of $(\beta_t, \sigma_t)$.
\end{theorem}

\begin{proof}
The proof proceeds in two steps: we first express the drift and its estimator in terms of a single family of scores parametrized by noise level, then evaluate the optimized KL divergence and perform a change of variables.

\medskip
\noindent\textit{Step 1: Reduction to scores.}
From the identities $b_t(x) = \dot{\beta}_t\,\bE[x_\star\mid x_t =x] + \sqrt{t}\, \dot{\sigma}_t\,\bE[z\mid x_t =x]$ and $x = \beta_t\, \bE[x_\star\mid x_t =x] + \sqrt{t}\, \sigma_t \,\bE[z\mid x_t =x]$, we can express the drift as 
\begin{equation}
\label{eqn-b-t-as-scores}
    b_t(x) = \frac{\dot\beta_t}{\beta_t}x + \frac{t\sigma_t^2}{\beta_t}\left(\frac{\dot\beta_t}{\beta_t}-\frac{\dot\sigma_t}{\sigma_t}\right)\nabla \log q_{\frac{\sqrt{t}\,\sigma_t}{\beta_t}}\!\left(\frac{x}{\beta_t}\right),
\end{equation}
where $q_r$ denotes the density of $x_\star + rz$. Correspondingly, the estimated drift defines an estimated score $\hat{s}_r$ via
\begin{equation}
\label{eqn-b-t-as-scores-approx}
    \hat{b}_t(x) = \frac{\dot\beta_t}{\beta_t}x + \frac{t\sigma_t^2}{\beta_t}\left(\frac{\dot\beta_t}{\beta_t}-\frac{\dot\sigma_t}{\sigma_t}\right)\hat{s}_{\frac{\sqrt{t}\,\sigma_t}{\beta_t}}\!\left(\frac{x}{\beta_t}\right).
\end{equation}
These representations make explicit that the drift at time $t$ depends on the score at noise level $r(t) = \sqrt{t}\,\sigma_t/\beta_t$, with a schedule-dependent reparametrization of the noise scale.

\medskip
\noindent\textit{Step 2: Evaluation of $\mathrm{KL}^\star$.}
Recall from~\eqref{eq:kl:a2} that the path-space KL divergence with diffusion coefficient $g_t$ is
\begin{equation}
\label{eqn-KL-path}
        \mathrm{KL} (\mathbb{P}_{X^g}\|\mathbb{P}_{\hat X^g})  = \frac12 \int_0^1 g^{-2}_t \left|1+\tfrac12 \beta_t  A_t(g_t^2-\sigma_t^2)\right|^2 \bE[ |b_t(X^g_t)-\hat b_t(X^g_t)|^2] \,{\rm d}t.
\end{equation}
Inserting $g_t = g_t^{\rm F}$ from~\eqref{eqn-follmer-gt}, together with the score representations~\eqref{eqn-b-t-as-scores} and~\eqref{eqn-b-t-as-scores-approx}, and using $X_t^g \stackrel{d}{=} I_t$, a direct calculation yields 
\begin{equation}
\begin{aligned}
    \mathrm{KL}^\star &= 2 \int_0^1 \max\!\left\{0,\; \frac{\sigma_t\sqrt{t}}{\beta_t}\frac{\rm d}{{\rm d}t}\!\left(-\frac{\sigma_t\sqrt{t}}{\beta_t}\right)\right\} \\
    & \quad \times\bE\!\left[\left|\nabla \log q_{\frac{\sigma_t\sqrt{t}}{\beta_t}}\!\left(x_\star+\frac{\sigma_t\sqrt{t}}{\beta_t}z\right) - \hat{s}_{\frac{\sigma_t\sqrt{t}}{\beta_t}}\!\left(x_\star+\frac{\sigma_t\sqrt{t}}{\beta_t}z\right)\right|^2\right] {\rm d}t.
\end{aligned}
\end{equation}
Setting $r = \sigma_t\sqrt{t}/\beta_t$ and performing the change of variables $t \mapsto r$ gives~\eqref{eqn-KL-invariant}. Since the right-hand side depends only on the score estimation error across noise scales, it is independent of $(\beta_t, \sigma_t)$.
\end{proof}

Theorem~\ref{thm:schedule-invariance} shows that different schedules are statistically equivalent in path-space relative entropy after optimal tuning: the minimized KL divergence coincides with that of score-based diffusion models initialized from a Gaussian source~\cite{kingma2021variational,chen2025lipschitz}. Different schedules nevertheless induce different optimization landscapes and numerical behavior, and the choice of schedule remains consequential in practical implementation~\cite{lim2024elucidating}.

\section{Numerical Illustrations of Path-Space Variational Optimality}
\label{sec-numerical illustration}

We illustrate the practical impact of the path-space variational principle of Theorem~\ref{prop-min-KL} through two experiments, both motivated by probabilistic forecasting and data assimilation. In each case the target $x_\star$ is a \emph{conditional} law of a future state given a past observation rather than an unconditional distribution. The framework of Sections~\ref{sec:Stochastic Interpolants}--\ref{sec:invariance} extends to this conditional setting verbatim by interpreting every density, drift, and KL divergence pointwise in the conditioning variable $c$: one trains a single drift $\hat b_t(x\mid c)$ on the loss~\eqref{eqn-loss-for-interpolants} with $x_\star$ replaced by $x_\star\mid c$, and Theorem~\ref{prop-min-KL} continues to characterize the diffusion coefficient $g_t$ that minimizes the conditional path-space KL divergence between the exact and learned generative SDEs (in expectation over $c$).

\smallskip
\paragraph{\bf Two experiments.}
Section~\ref{sec:numerics-ou} considers a one-dimensional Ornstein--Uhlenbeck forecasting problem for which the conditional drift $b_t(x\mid c)$ is available in closed form. This affords a direct Monte Carlo estimate of the path-space KL divergence~\eqref{eq:kl:a2} between the exact and learned generative dynamics, and isolates the schedule-induced contribution. We show that the F\"ollmer choice achieves an order-of-magnitude smaller path-space KL than the alternatives, exactly as Theorem~\ref{prop-min-KL} predicts.

Section~\ref{sec:numerics-ns} then considers the F\"ollmer schedule on a high-dimensional task in which the \emph{entire} generative trajectory enters the computation: a data-assimilation problem on a two-dimensional Kolmogorov flow. Path-space variational optimality is a natural notion of optimality for such tasks, and we report that the F\"ollmer choice yields the smallest posterior-mean error among the three diffusion coefficients tested.

Throughout this section we fix the linear-linear schedule $\beta_t = t$, $\sigma_t = 1-t$ and compare three choices of diffusion coefficient: the \emph{baseline} $g_t = \sigma_t = 1-t$ of Theorem~\ref{thm:1:b}; the \emph{constant} $g_t = 1$; and the \emph{F\"ollmer} choice $g_t^{\rm F} = \sqrt{1-t^2}$ of Theorem~\ref{prop-min-KL}. All three are run from the same estimated drift $\hat b_t$, composed at inference time through~\eqref{eq:sde-tuning-g-approx}; only the diffusion coefficient differs.

\subsection{Ornstein--Uhlenbeck process forecasting}
\label{sec:numerics-ou}
Let $(Y_s)_{s\geq 0}$ be the one-dimensional Ornstein--Uhlenbeck process ${\rm d}Y_s = -\lambda Y_s\,{\rm d}s + \sigma\,{\rm d}W_s$ with $\lambda = \sigma = 1$, sampled in its stationary regime $Y_s \sim \sfN(0, \sigma^2/(2\lambda))$. For a fixed lead time $\tau = 0.5$, we treat the present state $Y_s$ as the conditioning variable $c$ and target the conditional forecast increment
\begin{equation}
\label{eq:ou-target}
    x_\star \mid c = (Y_{s+\tau} - Y_s)\mid Y_s = c .
\end{equation}
The conditional law $x_\star\mid c$ is Gaussian, so the conditional baseline drift
\[
    b_t(x\mid c) = \E\bigl[\dot\beta_t x_\star + \dot\sigma_t\sqrt{t}\,z \,\big|\, x_t = x,\, c\bigr]
\]
is available in closed form, and we can compare it to its estimator $\hat b_t(x\mid c)$ exactly. We train $\hat b_t$ by minimizing the square loss~\eqref{eqn-loss-for-interpolants} (with $x_\star$ replaced by $x_\star\mid c$ and $c\sim \sfN(0,\sigma^2/(2\lambda))$ resampled per minibatch), using a three-layer MLP of hidden width $128$ with a Gaussian-Fourier time embedding, AdamW for $2\times 10^4$ steps at learning rate $2\times 10^{-4}$, and batch size $512$. Results below are averaged over five independent training seeds.

For each diffusion coefficient $g$ we compare the exact and learned generative SDEs through two divergences:
\begin{itemize}
\item \emph{Path-space KL} between $\mathbb{P}_{X^g}$ and $\mathbb{P}_{\hat X^g}$, in the Girsanov form~\eqref{eq:kl:a2},
\begin{equation}
\label{eq:ou-path-kl}
    \mathrm{KL}(\mathbb{P}_{X^g}\|\mathbb{P}_{\hat X^g}) = \tfrac{1}{2}\int_{0}^{1} g_t^{-2}\,\bigl|1 + \tfrac{1}{2}\beta_t A_t(g_t^2 - \sigma_t^2)\bigr|^2\, \E\bigl[|b_t(I_t\mid c) - \hat b_t(I_t\mid c)|^2\bigr]\,{\rm d}t .
\end{equation}
Using the identity $X_t^g \stackrel{d}{=} I_t$, the expectation is evaluated by drawing $t\sim \mathcal{U}(0,1)$ and $X_t$ directly from the interpolant law; no SDE trajectory needs to be simulated. We use $10^5$ Monte Carlo samples (jointly over $t$, $c$, and $X_t$).
\item \emph{Terminal marginal KL} at $t = 1-t_{\min}$ between $\hat X_t^g$ and the target $x_\star\mid c$, computed between $10^4$ Euler--Maruyama samples of $\hat X^g$ on $200$ uniform steps and the analytic target density via Gaussian kernel density estimators.
\end{itemize}
For numerical stability we integrate on $[t_{\min}, 1-t_{\min}]$ with $t_{\min} = 10^{-3}$ rather than $[0,1]$ throughout. The truncation changes the reported path-KL values by less than the Monte Carlo standard deviation.

\begin{table}[h]
\centering
\small
\begin{tabular}{lcc}
\toprule
diffusion coefficient & path-space KL & terminal marginal KL \\
\midrule
Baseline, $g_t = 1-t$            & $1.56 \pm 0.76$                          & $(8.2 \pm 2.9)\cdot 10^{-4}$ \\
Constant, $g_t = 1$              & $(4.5 \pm 2.2)\cdot 10^{-1}$             & $(8.9 \pm 5.6)\cdot 10^{-3}$ \\
F\"ollmer, $g_t = \sqrt{1-t^2}$  & ${\bf (1.5 \pm 0.5)\cdot 10^{-2}}$             & ${\bf (6.0 \pm 2.4)\cdot 10^{-4}}$ \\
\bottomrule
\end{tabular}
\caption{Path-space and terminal-marginal KL divergences for the OU forecasting target~\eqref{eq:ou-target} under the linear-linear schedule. Mean $\pm$ standard deviation over five training seeds and Monte Carlo average.}
\label{tab:ou}
\end{table}
\paragraph{\bf Results.} Table~\ref{tab:ou} reports both divergences. On the path-space KL, the F\"ollmer choice dominates, improving on the constant diffusion by roughly one and a half orders of magnitude and on the baseline by two. This ordering is exactly the one predicted by Theorem~\ref{prop-min-KL}: the integrand of~\eqref{eq:ou-path-kl} is pointwise minimized at $g = g_t^{\rm F}$, while the alternatives pay a penalty through the prefactor $g_t^{-2}\bigl|1+\tfrac12 \beta_t A_t(g_t^2-\sigma_t^2)\bigr|^2$.

For comparison, on the terminal marginal KL the same ordering holds: F\"ollmer is about $1.4\times$ better than the baseline and $\sim\!15\times$ better than the constant choice. The two divergences are related by the data-processing inequality, $\mathrm{KL}(\mathrm{Law}(X_1)\|\mathrm{Law}(\hat X_1)) \le \mathrm{KL}(\mathbb{P}_{X^g}\|\mathbb{P}_{\hat X^g})$.

\subsection{Posterior sampling for a 2D Kolmogorov flow}
\label{sec:numerics-ns}

We now showcase the F\"ollmer schedule on a high-dimensional posterior-sampling (data-assimilation) task for a two-dimensional Kolmogorov flow. 

\vspace{0.2em}
\paragraph{\bf Prior: conditional forecasting model.}
We work with the two-dimensional incompressible Navier--Stokes equations in vorticity form on the periodic torus $\mathbb{T}^2 = [0,2\pi)^2$,
\begin{equation}
\label{eq:kolmogorov-pde}
    \partial_s \omega + (u\cdot\nabla)\omega = \nu\,\Delta\omega + f - \alpha\,\omega,
    \qquad u = \nabla^\perp \Delta^{-1}\omega,
\end{equation}
with viscosity $\nu = 10^{-3}$, linear drag $\alpha = 0.1$, and Kolmogorov forcing $f(x_1,x_2) = \sin(4 x_2)$, giving a Reynolds number of order $10^3$. The flow is in a statistically stationary, weakly turbulent regime. We discretise the state $\omega\in\R^{128\times 128}$ on a uniform grid and integrate~\eqref{eq:kolmogorov-pde} with a pseudo-spectral scheme, recording snapshots every $\Delta t = 0.05$ time units. 

We define the conditioning variable as a $4\times$ average-pooled version of the current state, $\tilde\omega_t = A_4 \omega_t \in \R^{32\times 32}$, where $A_k$ denotes the linear $k\times k$ block-averaging operator on a $128\times 128$ grid. Many fine states $\omega_{t+\tau}$ are consistent with the same coarse $\tilde\omega_t$, which makes the forecast genuinely probabilistic. The prior we learn is therefore the conditional density
\[
    p_\tau(\omega_{t+\tau}\mid \tilde\omega_t),
\]
trained via the conditional stochastic interpolant framework of Section~\ref{sec:Stochastic Interpolants}: a single drift network $\hat b_t(x\mid \tilde\omega_t)$, a $\sim\!2$M-parameter U-Net taking $\tilde\omega_t$ as an additional input channel, is fit to the conditional version of the loss~\eqref{eqn-loss-for-interpolants}. The lag is set to $\tau = 0.5$, approximately half a decorrelation time of the flow: at shorter lags the forecast becomes nearly deterministic, while at much longer lags the conditional density approaches the climatology, so we choose an intermediate lag corresponding to the genuinely uncertain regime.

\vspace{0.3em}
\paragraph{\bf Observation model and posterior target.}
We then assume a noisy linear observation of the future state,
\begin{equation}
\label{eq:ns-obs}
    y = A_8 \omega_{t+\tau} + \eta, \qquad \eta\sim\sfN(0,\sigma_y^2 I_{16\times 16}), \quad \sigma_y = 0.3,
\end{equation}
where $A_8$ is an $8\times$ average-pooling operator. The observation $y\in \R^{16\times 16}$ has roughly $1.5\%$ of the dimension of the fine state and is independent of the conditioning $\tilde\omega_t$. The posterior of interest is
\[
    p(\omega_{t+\tau}\mid \tilde\omega_t,\, y)\propto p_\tau(\omega_{t+\tau}\mid\tilde\omega_t)\,p(y\mid\omega_{t+\tau}) ,
\]
a typical \emph{single-step data-assimilation} target.

\vspace{0.3em}
\paragraph{\bf Sampler.}
We derive the posterior sampler by writing the prior drift in score form and substituting the posterior score by Bayes' rule. For the linear-linear schedule, a direct computation rewrites~\eqref{eq:sde-tuning-g-approx} as
\begin{equation}
\label{eq:bg-score-decomp}
    \hat b_t^g(x\mid\tilde\omega_t) = \frac{x}{t} + C_g(t)\,\nabla_x\log\rho_t(x\mid\tilde\omega_t), \qquad C_g(t) := \frac{g_t^2 + (1-t^2)}{2},
\end{equation}
where $\rho_t(\cdot\mid\tilde\omega_t)$ is the conditional density of $X_t^g$ given $\tilde\omega_t$. To sample from the posterior $p(\omega_{t+\tau}\mid\tilde\omega_t,y)$, we apply Bayes' rule to the time-$t$ score: writing $\rho_t(x\mid\tilde\omega_t,y)\propto \rho_t(x\mid\tilde\omega_t)\,\phi_t(x)$ with $\phi_t(x) := p(y\mid x_t = x,\tilde\omega_t)$, we have
\[
    \nabla_x\log\rho_t(x\mid\tilde\omega_t,y) = \nabla_x\log\rho_t(x\mid\tilde\omega_t) + \nabla_x\log\phi_t(x).
\]
Substituting this conditional score in place of the unconditional one in~\eqref{eq:bg-score-decomp} yields the posterior generative SDE
\begin{equation}
\label{eq:ns-guided-sde}
    {\rm d}\hat X_t^g = \bigl[\hat b_t^g(\hat X_t^g\mid\tilde\omega_t) + C_g(t)\,\nabla_x\log\phi_t(\hat X_t^g)\bigr]\,{\rm d}t + g_t\,{\rm d}W_t ,
\end{equation}
which differs from the prior SDE by the addition of $C_g(t)\,\nabla\log\phi_t$ in the drift. The likelihood at intermediate time $t$, $\phi_t(x) := p(y\mid x_t = x,\tilde\omega_t)$, is intractable since it requires marginalising the network drift over future intermediate states. Following classifier guidance \cite{dhariwal2021diffusion} and diffusion posterior sampling (DPS) \cite{chungdiffusion}, we approximate it by $\phi_t(x) \approx p\bigl(y\mid \hat x_1(x,t)\bigr)$, where the \emph{denoiser}
\[
    \hat x_1(x,t) := \E[x_\star\mid x_t = x,\,\tilde\omega_t]
\]
is a one-step prediction of the terminal state from the current intermediate state. For the linear-linear schedule the denoiser is given in closed form by the trained drift: starting from the interpolant $x_t = t\,x_\star + (1-t)\sqrt{t}\,z$, with $z\sim\sfN(0,I)$ independent of $x_\star$, the baseline drift is
\[
    b_t(x\mid\tilde\omega_t) = \E\bigl[\dot\beta_t x_\star + \dot\sigma_t\sqrt{t}\,z \,\big|\, x_t = x,\,\tilde\omega_t\bigr] = \E[x_\star\mid x_t,\tilde\omega_t] - \sqrt{t}\,\E[z\mid x_t,\tilde\omega_t].
\]
Combining this with the linear constraint $x_t = t\,\E[x_\star\mid x_t,\tilde\omega_t] + (1-t)\sqrt{t}\,\E[z\mid x_t,\tilde\omega_t]$ (taking the conditional expectation of the interpolant) and solving the $2\times 2$ linear system for $\E[x_\star\mid x_t,\tilde\omega_t]$ yields
\begin{equation}
\label{eq:tweedie-zps}
    \hat x_1(x,t) = x + (1-t)\,\hat b_t(x\mid\tilde\omega_t),
\end{equation}
which we use as a plug-in estimator with $\hat b_t$ in place of $b_t$.

\vspace{0.3em}
\paragraph{\bf Results.}
Table~\ref{tab:ns-posterior} reports the relative posterior-mean RMSE, defined as \[\|\bar\omega_{t+\tau} - \omega_{t+\tau}^{\rm true}\|_2/\|\omega_{t+\tau}^{\rm true}\|_2,\] where $\bar\omega_{t+\tau}$ is the ensemble mean of $16$ posterior samples drawn by integrating~\eqref{eq:ns-guided-sde} from independent initial draws. The reported number is averaged over six held-out test pairs $(\tilde\omega_t, \omega_{t+\tau})$, with the $\pm$ being the standard deviation across these pairs. The F\"ollmer schedule attains the smallest RMSE, with a $\sim\!12\%$ improvement over the baseline and $\sim\!7\%$ over the constant choice. The reported $\pm$ may look large relative to these gaps, but it is dominated by the variation in test-pair difficulty (per-pair RMSE values range from $0.21$ to $0.42$), which is shared across all schedules; the F\"ollmer improvement is in fact consistent across every individual test pair.

Posterior sampling is a task in which the \emph{entire} generative trajectory contributes to each output: the integrated effect of guidance throughout the SDE shapes the resulting sample. We believe this is a regime in which the path-space variational principle of Theorem~\ref{prop-min-KL} may be relevant, and the F\"ollmer schedule's advantage in Table~\ref{tab:ns-posterior} is consistent with this view. These findings complement the related observation in \cite{domingo2024adjoint} that memoryless noise schedules are advantageous for posterior sampling with Gaussian-base diffusion models.

Figure~\ref{fig:kolmogorov-flow-forecast} shows a representative sample: the ensemble mean captures the truth, and the ensemble standard deviation correlates with the pointwise error of the ensemble mean.

\begin{table}[h]
\centering
\small
\begin{tabular}{lc}
\toprule
diffusion coefficient & posterior-mean RMSE \\
\midrule
Baseline, $g_t = 1-t$          & $0.313 \pm 0.090$ \\
Constant, $g_t = 1$            & $0.297 \pm 0.084$ \\
F\"ollmer, $g_t = \sqrt{1-t^2}$ & $\mathbf{0.275 \pm 0.091}$ \\
\bottomrule
\end{tabular}
\caption{Relative posterior-mean RMSE on the Kolmogorov-flow data-assimilation task~\eqref{eq:ns-obs} at forecast lag $\tau = 0.5$, sampling from the posterior SDE~\eqref{eq:ns-guided-sde}. Mean $\pm$ standard deviation across six held-out test pairs $(\tilde\omega_t, \omega_{t+\tau})$ drawn from disjoint test trajectories, with $16$ posterior samples per test pair and $100$ Euler--Maruyama steps on $[t_{\min}, 1-t_{\min}]$, $t_{\min} = 10^{-2}$. The reported $\pm$ is dominated by the variation in test-pair difficulty; for the schedule comparison, the relevant uncertainty is the standard error of the per-pair paired difference, which is roughly $5\times$ smaller (see main text).}
\label{tab:ns-posterior}
\end{table}

\begin{figure}[ht]
    \centering
    \includegraphics[width=0.9\linewidth]{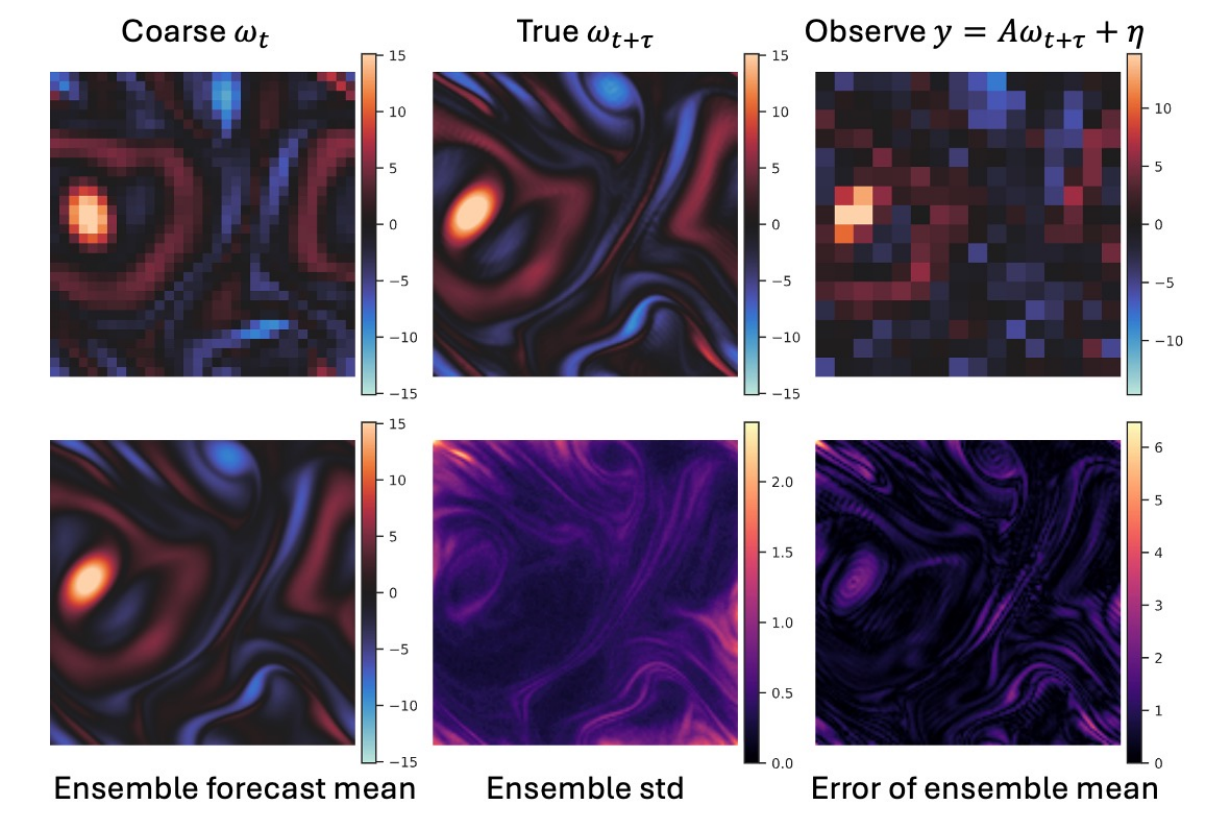}
    \caption{A representative sample from the Kolmogorov-flow posterior-sampling experiment, drawn under the F\"ollmer schedule with $N = 16$ posterior samples on one held-out test pair. Top row: coarse conditioning variable $\tilde\omega_t$, ground-truth future state $\omega_{t+\tau}$, and observation $y = A_8\,\omega_{t+\tau} + \eta$. Bottom row: ensemble-mean forecast, ensemble standard deviation $\hat\sigma$, and pointwise error of the ensemble mean relative to $\omega_{t+\tau}$. The ensemble spread correlates with the ensemble-mean error, and the pointwise error is within $\approx 3\hat\sigma$ throughout the domain.}
    \label{fig:kolmogorov-flow-forecast}
\end{figure}

\section{Conclusions}
\label{sec:conclusions}

This paper has developed a variational analysis of generative diffusions that transport a point source to a prescribed target distribution, organized around three results. First, within the stochastic interpolant framework the choice of diffusion coefficient is a free parameter that does not affect marginals; minimizing a path-space Kullback--Leibler criterion over this parameter selects, in closed form, a distinguished diffusion which we identify as a F\"ollmer process. This yields a new variational characterization of F\"ollmer processes in which the reference diffusion is not imposed \emph{a priori}, but emerges from optimization over admissible dynamics with prescribed marginals---in particular, the schedules determine the reference process while the target enters only through endpoint constraints. Second, the F\"ollmer drift admits a simulation-free conditional-expectation representation across linear reference processes (Proposition~\ref{prop-Follmer-SDE}), so it can be estimated from data by square-loss regression. Third, once the diffusion coefficient is tuned optimally, different interpolation schedules become statistically equivalent in path-space relative entropy.

These results clarify the role of diffusion coefficients and interpolation schedules in point-source generative models, and provide a principled foundation for their design in applications such as probabilistic forecasting and data assimilation, as demonstrated on a low-dimensional and a high-dimensional experiment in Section~\ref{sec-numerical illustration}. More broadly, they illustrate how variational principles can be used to disentangle marginal constraints from path-space optimality in stochastic dynamics. Natural directions for future work include extensions to nonlinear reference processes, infinite-dimensional settings (e.g., \cite{chen2025scale}), problems with additional structural constraints, and analysis of the effect of time discretizations.

\bibliography{ref_follmer}

@article{hairer2006ergodicity,
  title={Ergodicity of the 2D Navier-Stokes equations with degenerate stochastic forcing},
  author={Hairer, Martin and Mattingly, Jonathan C},
  journal={Annals of Mathematics},
  pages={993--1032},
  year={2006},
  publisher={JSTOR}
}

@article{albergo2023stochastic,
  title={Stochastic interpolants: A unifying framework for flows and diffusions},
  author={Albergo, Michael and Boffi, Nicholas M and Vanden-Eijnden, Eric},
  journal={Journal of Machine Learning Research},
  volume={26},
  number={209},
  pages={1--80},
  year={2025}
}

@inproceedings{albergo2022building,
  title={Building Normalizing Flows with Stochastic Interpolants},
  author={Albergo, Michael S and Vanden-Eijnden, Eric},
  booktitle={The Eleventh International Conference on Learning Representations},
  year={2022}
}

@inproceedings{liu2022flow,
  title={Flow Straight and Fast: Learning to Generate and Transfer Data with Rectified Flow},
  author={Liu, Xingchao and Gong, Chengyue and Liu, Qiang},
  booktitle={The Eleventh International Conference on Learning Representations},
  year={2022}
}

@inproceedings{lipman2022flow,
  title={Flow Matching for Generative Modeling},
  author={Lipman, Yaron and Chen, Ricky TQ and Ben-Hamu, Heli and Nickel, Maximilian and Le, Matthew},
  booktitle={The Eleventh International Conference on Learning Representations},
  year={2022}
}

@inproceedings{ho2020denoising,
  title={Denoising diffusion probabilistic models},
  author={Ho, Jonathan and Jain, Ajay and Abbeel, Pieter},
  booktitle={Advances in neural information processing systems},
  volume={33},
  pages={6840--6851},
  year={2020}
}

@article{song2020score,
  title={Score-based generative modeling through stochastic differential equations},
  author={Song, Yang and Sohl-Dickstein, Jascha and Kingma, Diederik P and Kumar, Abhishek and Ermon, Stefano and Poole, Ben},
  journal={arXiv preprint arXiv:2011.13456},
  year={2020}
}

@article{follmer1986time,
  title={Time reversal on Wiener space},
  author={F{\"o}llmer, H},
  journal={Stochastic Processes—Mathematics and Physics},
  pages={119--129},
  year={1986},
  publisher={Springer}
}

@inproceedings{schrodinger1932theorie,
  title={Sur la th{\'e}orie relativiste de l'{\'e}lectron et l'interpr{\'e}tation de la m{\'e}canique quantique},
  author={Schr{\"o}dinger, E},
  booktitle={Annales de l'institut Henri Poincar{\'e}},
  volume={3},
  pages={269--310},
  year={1932}
}

@article{leonard2014survey,
  title={A survey of the Schr{\"o}dinger problem and some of its connections with optimal transport},
  author={L{\'e}onard, Christian},
  journal={Discrete and Continuous Dynamical Systems-Series A},
  volume={34},
  number={4},
  pages={1533--1574},
  year={2014}
}

@article{huang2021schrodinger,
  title={{S}chr{\"o}dinger-{F\"o}llmer sampler: sampling without ergodicity},
  author={Huang, Jian and Jiao, Yuling and Kang, Lican and Liao, Xu and Liu, Jin and Liu, Yanyan},
  journal={arXiv preprint arXiv:2106.10880},
  year={2021}
}

@article{jiao2021convergence,
  title={Convergence Analysis of {S}chr{\"o}dinger-{F\"o}llmer Sampler without Convexity},
  author={Jiao, Yuling and Kang, Lican and Liu, Yanyan and Zhou, Youzhou},
  journal={arXiv preprint arXiv:2107.04766},
  year={2021}
}

@article{vargas2023bayesian,
  title={Bayesian learning via neural Schr{\"o}dinger--F{\"o}llmer flows},
  author={Vargas, Francisco and Ovsianas, Andrius and Fernandes, David and Girolami, Mark and Lawrence, Neil D and N{\"u}sken, Nikolas},
  journal={Statistics and Computing},
  volume={33},
  number={1},
  pages={3},
  year={2023},
  publisher={Springer}
}

@article{liu20232,
  title={{$I^{2}SB$}: Image-to-Image {S}chr{\"o}dinger Bridge},
  author={Liu, Guan-Horng and Vahdat, Arash and Huang, De-An and Theodorou, Evangelos A and Nie, Weili and Anandkumar, Anima},
  journal={arXiv preprint arXiv:2302.05872},
  year={2023}
}

@inproceedings{tzen2019theoretical,
  title={Theoretical guarantees for sampling and inference in generative models with latent diffusions},
  author={Tzen, Belinda and Raginsky, Maxim},
  booktitle={Conference on Learning Theory},
  pages={3084--3114},
  year={2019},
  organization={PMLR}
}

@inproceedings{debortoli2021diffusion,
  title={Diffusion schr{\"o}dinger bridge with applications to score-based generative modeling},
  author={De Bortoli, Valentin and Thornton, James and Heng, Jeremy and Doucet, Arnaud},
  booktitle={Advances in Neural Information Processing Systems},
  volume={34},
  pages={17695--17709},
  year={2021}
}

@inproceedings{wang2021deep,
  title={Deep generative learning via schr{\"o}dinger bridge},
  author={Wang, Gefei and Jiao, Yuling and Xu, Qian and Wang, Yang and Yang, Can},
  booktitle={International Conference on Machine Learning},
  pages={10794--10804},
  year={2021},
  organization={PMLR}
}

@inproceedings{zhang2021path,
  title={Path Integral Sampler: A Stochastic Control Approach For Sampling},
  author={Zhang, Qinsheng and Chen, Yongxin},
  booktitle={International Conference on Learning Representations},
  year={2021}
}

@article{chen2021stochastic,
  title={Stochastic control liaisons: Richard sinkhorn meets gaspard monge on a schrodinger bridge},
  author={Chen, Yongxin and Georgiou, Tryphon T and Pavon, Michele},
  journal={Siam Review},
  volume={63},
  number={2},
  pages={249--313},
  year={2021},
  publisher={SIAM}
}

@article{gyongy1986mimicking,
  title={Mimicking the one-dimensional marginal distributions of processes having an It{\^o} differential},
  author={Gy{\"o}ngy, Istv{\'a}n},
  journal={Probability theory and related fields},
  volume={71},
  number={4},
  pages={501--516},
  year={1986},
  publisher={Springer}
}

@article{peluchetti2023non,
  title={Non-denoising forward-time diffusions},
  author={Peluchetti, Stefano},
  journal={arXiv preprint arXiv:2312.14589},
  year={2023}
}

@inproceedings{eldan2020stability,
  title={Stability of the logarithmic Sobolev inequality via the F{\"o}llmer process},
  author={Eldan, Ronen and Lehec, Joseph and Shenfeld, Yair},
  booktitle={Annales de l’Institut Henri Poincar{\'e}-Probabilit{\'e}s et Statistiques},
  volume={56},
  pages={2253--2269},
  year={2020}
}

@inproceedings{lehec2013representation,
  title={Representation formula for the entropy and functional inequalities},
  author={Lehec, Joseph},
  booktitle={Annales de l'IHP Probabilit{\'e}s et statistiques},
  volume={49},
  pages={885--899},
  year={2013}
}

@article{eldan2018regularization,
  title={Regularization under diffusion and anticoncentration of the information content},
  author={Eldan, Ronen and Lee, James R},
  journal={Duke Mathematical Journal},
  volume={167},
  number={5},
  pages={969--993},
  year={2018},
  publisher={Duke University Press}
}

@inproceedings{shi2024diffusion,
  title={Diffusion Schr{\"o}dinger bridge matching},
  author={Shi, Yuyang and De Bortoli, Valentin and Campbell, Andrew and Doucet, Arnaud},
  booktitle={Advances in Neural Information Processing Systems},
  volume={36},
  year={2024}
}

@book{karatzas2014brownian,
  title={Brownian motion and stochastic calculus},
  author={Karatzas, Ioannis and Shreve, Steven},
  volume={113},
  year={2014},
  publisher={springer}
}

@inproceedings{chen2024probabilistic,
  title={Probabilistic Forecasting with Stochastic Interpolants and {F}{\"o}llmer Processes},
  author={Chen, Yifan and Goldstein, Mark and Hua, Mengjian and Albergo, Michael Samuel and Boffi, Nicholas Matthew and Vanden-Eijnden, Eric},
  booktitle={Forty-first International Conference on Machine Learning},
 year={2024}
}

@inproceedings{chen2023sampling,
  title={Sampling is as easy as learning the score: theory for diffusion models with minimal data assumptions},
  author={Chen, Sitan and Chewi, Sinho and Li, Jerry and Li, Yuanzhi and Salim, Adil and Zhang, Anru R},
  booktitle={International Conference on Learning Representations},
  year={2023}
}

@article{gao2023gaussian,
  title={Gaussian interpolation flows},
  author={Gao, Yuan and Huang, Jian and Jiao, Yuling},
  journal={arXiv preprint arXiv:2311.11475},
  year={2023}
}

@article{haussmann1986time,
  title={Time reversal of diffusions},
  author={Haussmann, Ulrich G and Pardoux, Etienne},
  journal={The Annals of Probability},
  pages={1188--1205},
  year={1986},
  publisher={JSTOR}
}

@article{cattiaux2023time,
  title={Time reversal of diffusion processes under a finite entropy condition},
  author={Cattiaux, Patrick and Conforti, Giovanni and Gentil, Ivan and L{\'e}onard, Christian},
  journal={Annales de l'Institut Henri Poincar{\'e} (B) Probabilit{\'e}s et Statistiques},
  volume={59},
  number={4},
  pages={1844--1881},
  year={2023},
  organization={Institut Henri Poincar{\'e}}
}

@article{lim2024elucidating,
  title={Elucidating the Design Choice of Probability Paths in Flow Matching for Forecasting},
  author={Lim, Soon Hoe and Wang, Yijin and Yu, Annan and Hart, Emma and Mahoney, Michael W and Li, Xiaoye S and Erichson, N Benjamin},
  journal={arXiv preprint arXiv:2410.03229},
  year={2024}
}

@article{mucke2025physics,
  title={Physics-aware generative models for turbulent fluid flows through energy-consistent stochastic interpolants},
  author={M{\"u}cke, Nikolaj T and Sanderse, Benjamin},
  journal={arXiv preprint arXiv:2504.05852},
  year={2025}
}

@inproceedings{sohl2015deep,
  title={Deep unsupervised learning using nonequilibrium thermodynamics},
  author={Sohl-Dickstein, Jascha and Weiss, Eric A and Maheswaranathan, Niru and Ganguli, Surya},
  booktitle={Proceedings of the 32nd International Conference on International Conference on Machine Learning-Volume 37},
  pages={2256--2265},
  year={2015}
}

@article{song2019generative,
  title={Generative modeling by estimating gradients of the data distribution},
  author={Song, Yang and Ermon, Stefano},
  journal={Advances in neural information processing systems},
  volume={32},
  year={2019}
}

@article{chen2025flowdas,
  title={{FlowDAS}: A Stochastic Interpolant-based Framework for Data Assimilation},
  author={Chen, Siyi and Jia, Yixuan and Qu, Qing and Sun, He and Fessler, Jeffrey A},
  journal={arXiv preprint arXiv:2501.16642},
  year={2025}
}

@article{sabti2025generative,
  title={A generative modeling approach to reconstructing 21 cm tomographic data},
  author={Sabti, Nashwan and Sudha, Ram Purandhar Reddy and Mu{\~n}oz, Julian B and Mishra-Sharma, Siddharth and Youn, Taewook},
  journal={Machine Learning: Science and Technology},
  volume={6},
  number={1},
  pages={015039},
  year={2025},
  publisher={IOP Publishing}
}

@inproceedings{cuesta-lazaro2024joint,
    title={Joint cosmological parameter inference and initial condition reconstruction with Stochastic Interpolants},
    author={Cuesta-Lazaro, Carolina and Bayer, Adrian E and Albergo, Michael S and Mishra-Sharma, Siddharth and Modi, Chirag and Eisenstein, Daniel J},
    booktitle={Machine Learning and the Physical Sciences Workshop},
    year={2024},
    organization={NeurIPS},
    address={Vancouver, Canada},
    month={December}
}

@article{anderson1982reverse,
  title={Reverse-time diffusion equation models},
  author={Anderson, Brian DO},
  journal={Stochastic Processes and their Applications},
  volume={12},
  number={3},
  pages={313--326},
  year={1982},
  publisher={Elsevier}
}

@article{schiodt2025generative,
  title={Generative Super-Resolution of Turbulent Flows via Stochastic Interpolants},
  author={Schi{\o}dt, Martin and M{\"u}cke, Nikolaj Takata and Velte, Clara Marika},
  journal={arXiv preprint arXiv:2508.13770},
  year={2025}
}

@article{chen2025scale,
  title={Scale-Adaptive Generative Flows for Multiscale Scientific Data},
  author={Chen, Yifan and Vanden-Eijnden, Eric},
  journal={arXiv preprint arXiv:2509.02971},
  year={2025}
}

@article{chen2025lipschitz,
  title={Lipschitz-guided design of interpolation schedules in generative models},
  author={Chen, Yifan and Vanden-Eijnden, Eric and Xu, Jiawei},
  journal={arXiv preprint arXiv:2509.01629},
  year={2025}
}

@article{yasuda2025probabilistic,
  title={Probabilistic Super-Resolution for Urban Micrometeorology via a Schr$\backslash$" odinger Bridge},
  author={Yasuda, Yuki and Onishi, Ryo},
  journal={arXiv preprint arXiv:2510.12148},
  year={2025}
}

@article{pooladian2025plug,
  title={Plug-in Estimation of Schr{\"o}dinger Bridges},
  author={Pooladian, Aram-Alexandre and Niles-Weed, Jonathan},
  journal={SIAM Journal on Mathematics of Data Science},
  volume={7},
  number={3},
  pages={1315--1336},
  year={2025},
  publisher={SIAM}
}

@article{horowitz2025baryonbridge,
  title={BaryonBridge: Stochastic Interpolant Model for Fast Hydrodynamical Simulations},
  author={Horowitz, Benjamin and Cuesta-Lazaro, Carolina and Yehia, Omar},
  journal={arXiv preprint arXiv:2510.19224},
  year={2025}
}

@article{kossaifi2026demystifying,
  title={Demystifying Data-Driven Probabilistic Medium-Range Weather Forecasting},
  author={Kossaifi, Jean and Kovachki, Nikola and Mardani, Morteza and Leibovici, Daniel and Ravuri, Suman and Shokar, Ira and Calvello, Edoardo and Abbas, Mohammad Shoaib and Harrington, Peter and Subramaniam, Ashay and others},
  journal={arXiv preprint arXiv:2601.18111},
  year={2026}
}

@article{efron2011tweedie,
  title={Tweedie’s formula and selection bias},
  author={Efron, Bradley},
  journal={Journal of the American Statistical Association},
  volume={106},
  number={496},
  pages={1602--1614},
  year={2011},
  publisher={Taylor \& Francis}
}

@article{kingma2021variational,
  title={Variational diffusion models},
  author={Kingma, Diederik and Salimans, Tim and Poole, Ben and Ho, Jonathan},
  journal={Advances in neural information processing systems},
  volume={34},
  pages={21696--21707},
  year={2021}
}

@article{domingo2024adjoint,
  title={Adjoint matching: Fine-tuning flow and diffusion generative models with memoryless stochastic optimal control},
  author={Domingo-Enrich, Carles and Drozdzal, Michal and Karrer, Brian and Chen, Ricky TQ},
  journal={arXiv preprint arXiv:2409.08861},
  year={2024}
}

@article{dhariwal2021diffusion,
  title={Diffusion models beat gans on image synthesis},
  author={Dhariwal, Prafulla and Nichol, Alexander},
  journal={Advances in neural information processing systems},
  volume={34},
  pages={8780--8794},
  year={2021}
}

@inproceedings{chungdiffusion,
  title={Diffusion Posterior Sampling for General Noisy Inverse Problems},
  author={Chung, Hyungjin and Kim, Jeongsol and Mccann, Michael Thompson and Klasky, Marc Louis and Ye, Jong Chul},
  booktitle={The Eleventh International Conference on Learning Representations},
  year={2023}
}

@inproceedings{klartag2023spectral,
  title={Spectral monotonicity under Gaussian convolution},
  author={Klartag, Bo’az and Putterman, Eli},
  booktitle={Annales de la Facult{\'e} des sciences de Toulouse: Math{\'e}matiques},
  volume={32},
  number={5},
  pages={939--967},
  year={2023}
}
\bibliographystyle{plain}

\newpage
\appendix
\section{Lipschitz Regularity of the Baseline Drift}
\label{appendix-Lipschitz regularity of the drift}
This section presents technical results regarding the Lipschitz regularity of the drift functions, in particular, the proof of the Lipschitz bounds stated in Theorem~\ref{thm:1:b}.

We follow the setting in \eqref{eq:b:def:app}. Consider $x_t  = \beta_t  x_\star + \sigma_t \sqrt{t}\,z$, $t\in[0,1]$, where $x_\star \sim \mu_\star$, $z\sim \sfN(0,\mathrm{I})$, $z\perp x_\star$. The drift $b_t(x)$ can be decomposed as
    \begin{equation}
    b_t(x) =  \dot\beta_t  \mathbb{E}[x_\star|x_t = x] + \dot \sigma_t   \mathbb{E}[\sqrt{t}\,z|x_t = x]\, .
    \end{equation}
Using the relation $x = \beta_t \mathbb{E}[x_\star|x_t = x] + \sigma_t   \mathbb{E}[\sqrt{t}\,z|x_t = x]$, we can equivalently write
\begin{equation}
    b_t(x) =  \frac{\dot{\sigma}_t}{\sigma_t}x + \Big(\dot{\beta}_t - \frac{\beta_t\dot{\sigma}_t}{\sigma_t}\Big)\mathbb{E}[x_\star|x_t = x]\, .
\end{equation}
We aim to characterize the Lipschitz regularity of $b$.  We denote $\rho^\star$ the density of $\mu_\star$ regarding the Lebesgue measure. First, an important observation is as follows: 
\begin{proposition}
\label{prop:cond-exp-gradient}
Suppose there exist constants $C_1,C_2 > 0$ such that
    \begin{equation}
    \label{assump-exp-tail}
        \rho^\star(x)\leq C_2\exp(-C_1|x|)
    \end{equation}
    for any $x \in \mathbb{R}^d$. Assume also $\beta_t, \sigma_t \in C^1([0,1])$ and $\sigma_t > 0$ for $t \in [0,1)$. Then for $t \in [0,1)$ and $x \in \mathbb{R}^d$,
    \begin{equation}
        \nabla_x \mathbb{E}[x_\star|x_t = x] = \frac{\beta_t }{t \sigma^2_t} \mathrm{Cov}(x_\star|x_t=x).
    \end{equation}
\end{proposition}
\begin{proof}[Proof of Proposition \ref{prop:cond-exp-gradient}]
By definition, $x_t  = \beta_t  x_\star + \sigma_t \sqrt{t}z$, and the conditional density  of $x_\star | x_t = x$ is proportional to \[\rho^\star(x_\star)\exp(-\frac{|x-\beta_t x_\star|^2}{2\sigma^2_t t}) ,\] 
from which we derive the following formula for the conditional expectation:
\begin{equation}
\label{eqn-cond-expectation-formula}
    \mathbb{E}[x_\star|x_t = x] = \frac{\int_{\R^d} x_\star \rho^\star(x_\star) \exp(-\frac12M_t |x_\star|^2 +m_t x_\star \cdot x ) {\rm d}x_\star}{\int_{\R^d} \rho^\star(x_\star) \exp(-\frac12M_t |x_\star|^2 +m_t x_\star \cdot x) {\rm d}x_\star}.
\end{equation}
Here we defined
    $M_t = \frac{\beta^2_t}{t \sigma^2_t}, m_t = \frac{\beta_t }{t \sigma^2_t}$.
  We differentiate \eqref{eqn-cond-expectation-formula} to derive
\begin{equation}
    \nabla_x \mathbb{E}[x_\star|x_t = x] = m_t\frac{P_t(x)Q_t(x) -R_t(x)R_t(x)^T}{|Q_t(x)|^2} \, ,
\end{equation}
where 
\begin{equation*}
    \begin{aligned}
        P_t(x) &= \int_{\R^d} x_\star x_\star^T \rho^\star(x_\star) \exp(-\frac12M_t |x_\star|^2 +m_t x_\star \cdot x ) {\rm d}x_\star\\
        Q_t(x) & = \int_{\R^d} \rho^\star(x_\star) \exp(-\frac12M_t |x_\star|^2 +m_t x_\star \cdot x) {\rm d}x_\star\\
        R_t(x) &= \int_{\R^d} x_\star \rho^\star(x_\star) \exp(-\frac12M_t |x_\star|^2 +m_t x_\star \cdot x) {\rm d} x_\star\, .
    \end{aligned}
\end{equation*}
To derive the above formula, we need to verify the interchange of limits and integrations. It is guaranteed by using \eqref{assump-exp-tail} and the Lebesgue dominated convergence theorem, for sufficiently small $t$.
Inspecting the above integrals, we can also equivalently write the result as
\[\nabla_x \mathbb{E}[x_\star|x_t = x] = m_t \text{Cov}(x_\star|x_t=x)\, .\]
Using the Lebesgue dominated convergence theorem, we get
\begin{equation}
    \lim_{t \to 0} \text{Cov}(x_\star|x_t=x) = \lim_{t\to 0} \frac{P_t(x)Q_t(x) -R_t(x)R_t(x)^T}{|Q_t(x)|^2} = \text{Cov}(x_\star)\, .
\end{equation}
The proof is complete.
\end{proof}
We analyze the Lipschitz regularity of the drift under two assumptions on $\mu_\star$.
\subsection{Local-in-time Lipschitz bound}
\begin{proposition}
   Assume $\mu_\star$ has bounded support, such that for $x_\star \sim \mu_\star$, it holds that $|x_\star|^2 \leq R^2$ almost surely. We have
    \[|\nabla_x \mathbb{E}[x_\star|x_t = x]|\leq 4R^2 \frac{\beta_t}{t\sigma_t^2}, \quad t \in (0,1), \]
    and 
    \begin{equation}
    \label{eqn-gradient-bound-bounded-support}
        |\nabla_x b_t(x)| \leq \left|\frac{\dot{\sigma}_t}{\sigma_t}\right| + 4R^2 \frac{\beta_t}{t\sigma_t^2}\left|\dot{\beta}_t - \frac{\beta_t\dot{\sigma}_t}{\sigma_t}\right|, \quad t \in (0,1)\, .
    \end{equation}
\end{proposition}
\begin{proof}
    The result is obtained by using $\nabla_x \mathbb{E}[x_\star|x_t = x] = \frac{\beta_t }{t \sigma^2_t} \mathrm{Cov}(x_\star|x_t=x)$ from Proposition~\ref{prop:cond-exp-gradient} and $\mathrm{Cov}(x_\star|x_t=x) \preceq 4R^2 \cdot \mathrm{I}$, which follows from the bounded support assumption.
\end{proof}
Since $\beta_t, \sigma_t \in C^1([0,1])$ with $\beta_0 = 0$ and $\sigma_t > 0$ for $t \in [0,1)$, the bound \eqref{eqn-gradient-bound-bounded-support} is uniform in time for $t \in [0, 1-\delta]$ for any $0 < \delta < 1$. Thus, the Lipschitz constant of $b_t$ is uniform in time on $[0, 1-\delta]$.

\subsection{Uniform-in-time Lipschitz bound} The local-in-time Lipschitz bound can be improved to a global-in-time bound if $\mu_\star$ is a Gaussian smoothing of a bounded support distribution.
\begin{proposition}
    Under Assumption~\ref{assum:bounded-support-smoothed-assumption}, we have
    \begin{equation}
    \label{eqn-gradient-bound-smoothed-bounded-support}
        |\nabla_x b_t(x)| \leq \left|\frac{\sigma_t\dot{\sigma}_t t + \beta_t \dot{\beta}_t \eta^2}{\sigma_t^2 t +\beta_t^2 \eta^2}\right| + 4R^2 \left|\frac{\beta_t\sigma_t t(\dot{\beta}_t\sigma_t - \beta_t \dot{\sigma}_t)}{(\sigma_t^2 t +\beta_t^2 \eta^2)^2}\right|, \quad t \in (0,1]\, .
    \end{equation}
\end{proposition}
\begin{proof}
    Recall $x_t  = \beta_t  x_\star + \sigma_t \sqrt{t}\,z$, $t\in[0,1]$, where $x_\star \sim \mu_\star$, $z\sim \sfN(0,\mathrm{I})$, $z\perp x_\star$. Since $\mu_\star=p*\sfN(0,\eta^2 \mathrm{I})$, we can write $x_\star = y_\star + \eta z_1$, where $y_\star \sim p$ has bounded support and $z_1 \sim \sfN(0,\mathrm{I})$ is independent, so that 
    \[x_t  = \beta_t  y_\star + (\sigma_t \sqrt{t}\,z + \beta_t \eta z_1) \overset{d}{=} \beta_t y_\star + \sqrt{\sigma_t^2 t +\beta_t^2 \eta^2}\, z := \tilde{x}_t\, . \]
    Let us now calculate $\mathbb{E}[x_\star|x_t = x]$. First, by definition, we have
\[\mathbb{E}[x_\star|x_t = x] = \mathbb{E}[y_\star|x_t = x] + \mathbb{E}[\eta z_1|x_t = x]\, .\]
Using the relation between score and conditional expectation, we have
\begin{equation}
    \nabla \log \rho_t(x) = -\frac{1}{\beta_t \eta }\mathbb{E}[z_1|x_t = x] = - \frac{1}{\sqrt{\sigma_t^2 t +\beta_t^2 \eta^2}}\mathbb{E}[z|\tilde{x}_t = x]\, ,
\end{equation}
where $\rho_t$ is the density of $x_t$ (or $\tilde{x}_t$). This shows that
\begin{equation}
    \mathbb{E}[\eta z_1|x_t = x] = \frac{\beta_t \eta^2}{\sqrt{\sigma_t^2 t +\beta_t^2 \eta^2}}\mathbb{E}\Big[\frac{\tilde{x}_t - \beta_t y_\star}{\sqrt{\sigma_t^2 t +\beta_t^2 \eta^2}}\;\Big|\;\tilde{x}_t = x\Big]\, .
\end{equation}
Using this formula, we get 
\begin{equation}
    \mathbb{E}[x_\star|x_t = x] = \frac{\sigma_t^2 t}{\sigma_t^2 t +\beta_t^2 \eta^2}\mathbb{E}[y_\star|\tilde{x}_t = x] + \frac{\beta_t\eta^2}{\sigma_t^2 t +\beta_t^2 \eta^2}x\, .
\end{equation}
Therefore, we get
\begin{equation}
\begin{aligned}
     b_t(x) &=  \frac{\dot{\sigma}_t}{\sigma_t}x + \Big(\dot{\beta}_t - \frac{\beta_t\dot{\sigma}_t}{\sigma_t}\Big)\mathbb{E}[x_\star|x_t = x]\\
     & = \frac{\dot{\sigma}_t}{\sigma_t}x + \Big(\dot{\beta}_t - \frac{\beta_t\dot{\sigma}_t}{\sigma_t}\Big)\left(\frac{\sigma_t^2 t}{\sigma_t^2 t +\beta_t^2 \eta^2}\mathbb{E}[y_\star|\tilde{x}_t = x] + \frac{\beta_t\eta^2}{\sigma_t^2 t +\beta_t^2 \eta^2}x\right)\\
     & = \frac{\sigma_t\dot{\sigma}_t t + \beta_t \dot{\beta}_t \eta^2}{\sigma_t^2 t +\beta_t^2 \eta^2} x + \frac{\sigma_t t(\dot{\beta}_t\sigma_t - \beta_t \dot{\sigma}_t)}{\sigma_t^2 t +\beta_t^2 \eta^2}\mathbb{E}[y_\star|\tilde{x}_t = x]\, .
\end{aligned}
\end{equation}
Now, we can use a similar argument as in the bounded support case to deduce that
\begin{equation}
    |\nabla_x\mathbb{E}[y_\star|\tilde{x}_t = x]|\leq 4R^2\frac{\beta_t}{\sigma_t^2t + \beta_t^2\eta^2}\, .
\end{equation}
Using this bound, we derive
\begin{equation}
    |\nabla_x b_t(x)| \leq \left|\frac{\sigma_t\dot{\sigma}_t t + \beta_t \dot{\beta}_t \eta^2}{\sigma_t^2 t +\beta_t^2 \eta^2}\right| + 4R^2 \left|\frac{\beta_t\sigma_t t(\dot{\beta}_t\sigma_t - \beta_t \dot{\sigma}_t)}{(\sigma_t^2 t +\beta_t^2 \eta^2)^2}\right|.
\end{equation}
The bound is finite for any $t \in (0,1]$, so we get a global-in-time Lipschitz bound.
\end{proof}
\section{Singularity Discussions on The F\"ollmer Drift at Initial Time}
\label{appendix-singular-drift-discussion}
This section discusses the singular behavior of the F\"ollmer drift at the initial time due to the choice $\dot\beta_0=0$. This supplements Remark \ref{remark-assumption-g-beta}.
\subsection{Discussions on the case $\dot \beta_0=0$} We first consider a specific example where $\sigma_t= 1-t$ and $\beta_t = t^2$. We set $g_t = g^\Fo_t=\sqrt{(1-t)(3-t)}$, the corresponding optimal diffusion coefficient, in~\eqref{eq:sde-tuning-g-approx}. The SDE admits the explicit formula
\begin{equation*}
    {\rm d} X^{g^\Fo}_t = \Big(1+\frac{1}{2-t}\Big)  b_t(X^{g^\Fo}_t){\rm d}t  - \frac{1}{t(2-t)}(2X^{g^\Fo}_t) {\rm d}t + \sqrt{(1-t)(3-t)} {\rm d}W_t.
\end{equation*}
The drift in this equation is singular at $t=0$ because of the term $-2X^{g^\Fo}_t/(t(2-t))$. Nevertheless, this term can be seen as an infinite restoring force, and the solution to this SDE is still well-defined for the initial condition $X^{g^\Fo}_{t=0} = 0$. In fact, by stochastic calculus, we can show that the solution satisfies the integral equation
\begin{equation*}
    X^{g^\Fo}_t = \frac{2-t}{t} \int_0^t \frac{u}{2-u} \left(\Big(1+\frac{1}{2-u}\Big)  b_u(X^{g^\Fo}_u)\right)  {\rm d}u +  \frac{2-t}{t} \int_0^t \frac{u\sqrt{(1-u)(3-u)}}{2-u}  {\rm d}W_u.
\end{equation*}
We present the existence and uniqueness of solutions to the above type of singular SDEs in Appendix \ref{appendix-singular-SDE}.

Beyond this specific example, for general $\dot{\beta}_0 = 0$, we have 
\begin{equation}
\begin{aligned}
    1+\tfrac12 \beta_t A_t (|g_t^{\rm F}|^2-\sigma_t^2) &= 1 + \frac{\beta_t}{2t\sigma_t(\dot\beta_t\sigma_t - \beta_t\dot\sigma_t)}\left(2t\sigma_t(\beta_t^{-1}\dot\beta_t \sigma_t- \dot\sigma_t) -\sigma_t^2\right)\\
    & = 2 - \frac{\sigma_t\beta_t}{t(\dot\beta_t\sigma_t - \beta_t\dot\sigma_t)}.
\end{aligned}
\end{equation}
Thus, the corresponding drift of the SDE has the form
\begin{equation}
\label{eqn-formula-btg}
    b_t^{g^\Fo}(x) = \left(2 - \frac{\sigma_t\beta_t}{t(\dot\beta_t\sigma_t - \beta_t\dot\sigma_t)}\right)b_t(x) - \frac{\dot\beta_t x}{t\sigma_t(\dot\beta_t\sigma_t - \beta_t\dot\sigma_t)}.
\end{equation}
Suppose $\beta_t = m t^k + o(t^k)$ for some $m > 0$ and $k > 1$ since $\dot\beta_0 = 0$. We get
\begin{equation}
    \lim_{t \to 0^+} 2 - \frac{\sigma_t\beta_t}{t(\dot\beta_t\sigma_t - \beta_t\dot\sigma_t)} = 2 - \lim_{t \to 0^+}\frac{\sigma_t t^k}{t(kt^{k-1}\sigma_t-t^k\dot\sigma_t)} = 2 - \frac{1}{k},
\end{equation}
which is finite. Moreover, if $\dot\sigma_1 \neq 0$, then
\begin{equation}
    \lim_{t \to 1^-} 2 - \frac{\sigma_t\beta_t}{t(\dot\beta_t\sigma_t - \beta_t\dot\sigma_t)} = 2 - \lim_{t \to 1^-}\frac{\sigma_t\beta_t}{(\dot\beta_t\sigma_t - \beta_t\dot\sigma_t)} = 2 - \frac{0}{0-\dot{\sigma}_1} = 2,
\end{equation}
which is finite as well. This means that in formula \eqref{eqn-formula-btg}, the coefficient of $b_t$ is uniformly bounded in time. For the linear term, we have 
\begin{equation}
\lim_{t \to 0^+}  - \frac{\dot\beta_t x}{t\sigma_t(\dot\beta_t\sigma_t - \beta_t\dot\sigma_t)} = - \infty,
\end{equation}
which also corresponds to a linear restoring force. The solution to the singular SDE with drift \eqref{eqn-formula-btg} is well-defined; see Appendix \ref{appendix-singular-SDE}. Therefore, the SDE \eqref{eq:sde-appenfix} with $g = g^{\Fo}$ is well-defined and can thus be used as a generative process.

\subsection{Solutions to SDEs with singular drifts}
\label{appendix-singular-SDE}
This section presents technical results regarding the existence and uniqueness of solutions to SDEs with singular drifts that appear when we optimize diffusion coefficients.

Consider the SDE
\begin{equation}
    {\rm d}X_s = -\frac{p(X_s-x_0)}{s} {\rm d}s + b_s(X_s) {\rm d}s + g_s{\rm d}W_s, \quad X_0 = x_0 \in \mathbb{R}^d,
\end{equation}
where $s \in [0,T]$ for some $T > 0$.
Here $p > 0$, $g: \mathbb{R} \to \mathbb{R}_+$. There is a singularity in the drift at $s=0$ so standard theories of SDEs do not apply directly. To deal with this, we apply It\^{o}'s calculus to obtain
\[ {\rm d}\left(s^p(X_s-x_0)\right) = s^p b(s,X_s){\rm d}s + s^p g_s{\rm d}W_s,  \]
from which we get the following stochastic integral equation
\begin{equation}
\label{eqn-integral-singular-init-time}
 X_s = x_0 + \frac{1}{s^p}\left( \int_0^s u^p b_u(X_u){\rm d}u + \int_0^s u^p g_u{\rm d}W_u  \right), \quad X_0 = x_0.
\end{equation}
We use the notation that $|\cdot|$ represents the Euclidean norm of vectors in any dimension. 
\begin{theorem}[Existence, Uniqueness, Bounds]
    Suppose $g\in C^1([0,T])$, and $b$ satisfies the Lipschitz and linear growth conditions
    \begin{itemize}
        \item $|b_s(x)-b_s(y)|\leq K|x-y|$
        \item $|b_s(x)|^2 \leq K^2(1+|x|^2)$
    \end{itemize}
    for any $s \in [0,T], x \in \mathbb{R}^d, y \in \mathbb{R}^d$ where $K$ is a positive constant. Let $(\Omega, \mathcal{F}, \mathbb{P})$ be the probability space and $\{\mathcal{F}_u, 0\leq u \leq s\}$ is the canonical filtration of the Wiener process $\{W_u, 0\leq u \leq s\}$.
    Then, there exists a continuous, adapted process $\{X_u, \mathcal{F}_u, 0\leq u \leq s\}$ which is a strong solution to \eqref{eqn-integral-singular-init-time}. That is, \eqref{eqn-integral-singular-init-time} holds almost surely for $X$ and $\mathbb{E}|X_u|^2 < \infty, 0\leq u \leq s$. Moreover, the strong solution is unique and there exists a constant $C$ that depends only on $K, T$ and $\|g\|_{L^{\infty}([0,T])}$ such that
    $\mathbb{E}[|X_s|^2] \leq C(1+|x_0|^2)\exp(Cs), s \in [0,T]$.
\end{theorem}
\begin{proof}
The idea of the proof is similar to the proof of \cite[Theorem 2.9]{karatzas2014brownian}, which handles SDEs with non-singular drift. We adapt the logic to the singular SDE here.

    We first show existence. Consider the Picard-Lindel{\" o}f iterations ($s \in [0,T]$):
    \begin{equation}
    \label{eqn-integral-singular-init-time-picard}
        X_s^{(0)} \equiv x_0; \quad X_s^{(n+1)} = x_0 + \frac{1}{s^p}\left( \int_0^s u^p b_u(X_u^{(n)}){\rm d}u + \int_0^s u^p g_u{\rm d}W_u  \right), \quad n\geq 0.
    \end{equation}
To begin, we have 
\begin{equation}
\begin{aligned}
     \mathbb{E}[|X_s^{(n+1)}|^2]&\leq 3|x_0|^2 + 3\left|\frac{1}{s^p} \int_0^s u^p b_u(X_u^{(n)}){\rm d}u \right|^2 + 3 \left|\frac{1}{s^p}\int_0^s u^p g_u{\rm d}W_u\right|^2
    \\
    & \leq 3|x_0|^2 + 3TK^2 \int_0^s (1+\mathbb{E}[|b_u(X_u^{(n)})|^2]) + 3T\|g\|^2_{L^{\infty}([0,T])}.
\end{aligned}
\end{equation}
We can write the above inequality in the following form
\begin{equation}
    \mathbb{E}[|X_s^{(n+1)}|^2] \leq C(1+|x_0|^2) + C \int_0^s \mathbb{E}[|b_u(X_u^{(n)})|^2] {\rm d}u, \quad s \in [0,T],
\end{equation}
where $C$ is a constant that depends only on $K, T$ and $\|g\|_{L^{\infty}([0,T])}$. Iterating the above inequality leads to 
\begin{equation}
\label{eqn-picard-iteration-L2-bound}
\mathbb{E}[|X_s^{(n+1)}|^2] \leq C(1+|x_0|^2)\left(1+Cs +\frac{(Cs)^2}{2} + \cdots + \frac{(Cs)^{n+1}}{(n+1)!} \right) \leq C(1+|x_0|^2)\exp(Cs).
\end{equation}
This shows that $X^{(n)}_s$ is square integrable for any $n \geq 0$ and $s \in [0,T]$. Moreover, if $X^{(n)}$ is a continuous, adapted process, we know that the process $X^{(n+1)}$ is continuous almost surely for $s \in (0,T]$ by the definitions of Lebesgue and stochastic integrations. It remains to examine the behavior of $X^{(n+1)}$ when $s \to 0$. Firstly, $\lim_{s\to 0} \frac{1}{s^p}\int_0^s u^p b_u(X_u^{(n)}){\rm d}u = 0$ almost surely since $X^{(n)}$ is a continuous stochastic process and $b$ is Lipschitz. Secondly, let $M_s :=\frac{1}{s^p} \int_0^s u^p g_u{\rm d}W_u$, $s \in (0,T]$, then using the formula of integration by parts in stochastic integration leads to
\[M_s = \frac{1}{s^p} \left(s^pg_sW_s - \int_0^s (\dot{g}_u u^p + pu^{p-1}g_u)W_u {\rm d}u\right), \quad \text{a.s.} \]
From the above, we get that $\lim_{s\to 0}M_s = 0$ almost surely. Therefore, we obtain $\lim_{s \to 0} X^{(n+1)}_s = x_0$, which means that $X^{(n+1)}$ is also a continuous, adapted process, assuming $X^{(n)}$ is the case. Using mathematical induction, we establish that $X^{(n)}$, for any $n \geq 0$, is a continuous, adapted process.

We consider the difference $X^{(n+1)}_s-X^{(n)}_s$, which satisfies
 \begin{equation}
        X_s^{(n+1)} - X_s^{(n)} =  \frac{1}{s^p}\left( \int_0^s u^p \left(b_u(X_u^{(n)})-b_u(X_u^{(n-1)})\right){\rm d}u \right).
    \end{equation}
Using the Lipschitz condition on $b$ and the fact $s \leq T$, we get
\begin{equation}
    \mathbb{E}[\max_{0\leq u \leq s}|X_s^{(n+1)} - X_s^{(n)}|^2] \leq TK^2 \int_0^T \mathbb{E}[|X_s^{(n)} - X_s^{(n-1)}|^2].
\end{equation}
Iterating the above inequality leads to
\begin{equation}
    \mathbb{E}[\max_{0\leq u \leq s}|X_s^{(n+1)} - X_s^{(n)}|^2] \leq C^\star \frac{(TK^2 s)^n}{n!}, \quad 0\leq s \leq T, n \geq 0,
\end{equation}
where $C^\star := \max_{0 \leq s \leq T} \mathbb{E}|X^{(1)}_s-x_0|^2$, which is finite. Then, using Markov's inequality gives
\begin{equation}
    \mathbb{P}(\max_{0\leq u \leq s}|X_s^{(n+1)} - X_s^{(n)}| > \frac{1}{2^{n+1}}) \leq 4C^\star \frac{(4TK^2 s)^n}{n!}, \quad n \geq 0.
\end{equation}
The above upper bound leads to a convergent series in $n$. Thus, from the Borel--Cantelli lemma, we conclude that there exists an event $\Omega^\star$ with $\mathbb{P}(\Omega^\star)=1$ and an integer-valued random variable $N(\omega)$, such that for all $\omega \in \Omega^\star$, it holds that
\begin{equation}
    \max_{0\leq u \leq s}|X_s^{(n+1)} - X_s^{(n)}| \leq \frac{1}{2^{n+1}}, \quad \forall n \geq N(\omega).
\end{equation}
Therefore,
\begin{equation}
    \max_{0\leq u \leq s}|X_s^{(n+m)} - X_s^{(n)}| \leq \frac{1}{2^{n}}, \quad \forall n \geq N(\omega), m \geq 1.
\end{equation}
The above fact implies that the continuous sample paths $X_s^{(n)}(\omega)$, $0\leq s \leq T$, converge, as $n\to \infty$ and in the supremum norm on continuous functions, to a limit $X_s(\omega)$, $0\leq s \leq T$, for all $\omega \in \Omega^\star$. Taking $n \to \infty$ in \eqref{eqn-integral-singular-init-time-picard}, we obtain that $X_s$ satisfies the stochastic integral equation almost surely. The bound $\mathbb{E}[|X_s|^2] \leq C(1+|x_0|^2)\exp(Cs)$, $s \in [0,T]$, is obtained by applying Fatou's lemma to \eqref{eqn-picard-iteration-L2-bound}. The proof of existence is complete.

Now we show uniqueness. Suppose both $X_s$ and $\tilde{X}_s$ satisfy the stochastic integral equation. We have
\begin{equation}
    X_s - \tilde{X}_s = \frac{1}{s^p}\left( \int_0^s u^p \left(b_u(X_u)-b_u(\tilde{X}_u)\right){\rm d}u\right).
\end{equation}
Squaring both sides and taking expectations leads to
\begin{equation}
    \mathbb{E}[|X_s - \tilde{X}_s|^2] \leq sK^2 \int_0^s \mathbb{E}[|X_u-\tilde{X}_u|^2] {\rm d}u.
\end{equation}
Applying Gr{\"o}nwall's inequality gives $X_s = \tilde{X}_s$ almost surely, which shows uniqueness.
\end{proof}
\section{Characterization of F\"ollmer Processes}
This section presents technical results about F\"ollmer processes, in particular the proof of Proposition~\ref{prop-Follmer-SDE} regarding the expression of F\"ollmer's drift as a conditional expectation, and a Schr\"odinger bridge and stochastic control perspective on the formula of F\"ollmer's drift. 
\subsection{F\"ollmer drifts as conditional expectations}
\label{appendix-sec:Follmer drifts and conditional expectations}
\begin{proof}[Proof of Proposition~\ref{prop-Follmer-SDE}]
We show the results formally following the steps below.

\textbf{Step 1.} The explicit solution to \eqref{eq:ref:foll-gen} is 
\[Y_t = \int_0^t \frac{r_t}{r_u}g_u \,{\rm d}W_u\, .\]
Thus $\E[Y_t] = 0$ and $\mathrm{Cov}(Y_t, Y_u) = h_{\min(t,u)} \eps\, \mathrm{I}$. Also, the density of $Y_t$, denoted by $\rho^{Y}_t$, is equal to the density of $\sfN(0, h_t \eps \,\mathrm{I})$.

\textbf{Step 2.} Since $Y_t$ is a Gaussian process, by Gaussian conditioning on endpoints and calculating the conditional mean and variance, we obtain that 
the reference process $Y_t$ has the same law at each $t$ as the interpolant process
\[y_t = h_t Y_1 + \sqrt{\eps h_t(1-h_t)}\,z, \]
with $z \sim \sfN(0,\mathrm{I})$ and $z \perp Y_1$. 

\textbf{Step 3.} We time-reverse the reference process and obtain $Y_t^{\rm R}\overset{d}{=}Y_{1-t}$, which satisfies the SDE
\begin{equation}
\begin{aligned}
    {\rm d}Y^{\rm R}_t &= -a_{1-t} Y^{\rm R}_t \,{\rm d}t + g_{1-t}^2\nabla \log \rho^{Y}_{1-t}(Y^{\rm R}_t)\,{\rm d}t  + g_{1-t}\,{\rm d}W_t\\
    & = -a_{1-t} Y^{\rm R}_t \,{\rm d}t - \frac{g_{1-t}^2}{\eps h_{1-t}}Y^{\rm R}_t\,{\rm d}t  + g_{1-t}\,{\rm d}W_t,\quad Y_0^{\rm R} \sim \sfN(0, h_1 \eps \,\mathrm{I}),
\end{aligned}
\end{equation}
where we used the fact that $\nabla \log \rho^{Y}_{1-t}(y) = - \frac{1}{\eps h_{1-t}}y$, since $\rho_{1-t}^Y$ is the density of $\sfN(0, h_{1-t} \eps\, \mathrm{I})$.

\textbf{Step 4.} To construct the F\"ollmer process, we change the initial condition $Y^{\rm R}_0$ to make it distributed according to the target $\mu_\star$. More precisely, denote the resulting process by $X^{\rm R}$, which satisfies
\begin{equation}
\begin{aligned}
    {\rm d}X^{\rm R}_t  = -a_{1-t} X^{\rm R}_t \,{\rm d}t - \frac{g_{1-t}^2}{\eps h_{1-t}}X^{\rm R}_t\,{\rm d}t  + g_{1-t}\,{\rm d}W_t,\quad X_0^{\rm R} \sim \mu_\star.
\end{aligned}
\end{equation}
The corresponding F\"ollmer process can be obtained by reversing the above SDE, namely $X^\Fo_t\overset{d}{=} X^{\rm R}_{1-t}$, which satisfies
\begin{equation}
    {\rm d}X_t^{\rm F} = a_t X^{\rm F}_t \,{\rm d}t  +\frac{g_t^2}{\eps h_t}X_t^{\rm F}\, {\rm d}t + g_t^2 \nabla \log \rho_t(X_t^{\rm F})\, {\rm d}t +  g_t\,{\rm d}W_t,\quad X_0^{\rm F} = 0,
\end{equation}
where $\rho_t$ is the density of $X^\Fo_{t}\overset{d}{=} X^{\rm R}_{1-t}$. Let us also write the above SDE as ${\rm d}X_t^{\rm F} = b^{\rm F}_t(X_t^{\rm F}) \,{\rm d}t + g_t \,{\rm d}W_t$.

Note that, by construction, $X_t^{\rm F}$ has the same law at each $t$ as the interpolant
\[x_t = h_t x_\star + \sqrt{\eps h_t(1-h_t)}\,z,\]
where $x_\star \sim \mu_\star$, $z \sim \sfN(0,\mathrm{I})$, and $z \perp x_\star$, since $X^{\Fo}$ and $Y$ share the same bridge process.

\textbf{Step 5.} The F\"ollmer drift above can be written as a conditional expectation, based on the fact that the score function can be written as a conditional expectation. More precisely, the score function of $X_t^{\rm F}$ is the same as that of $x_t$. By Tweedie's formula \cite{efron2011tweedie}, we have
\[\nabla \log \rho_t(x) = \E\left[-\frac{1}{\sqrt{\eps h_t(1-h_t)}}\,z\;\Big|\;x_t=x\right]. \]
Plugging this identity into the formula of the F\"ollmer drift, we get
\[b^{\rm F}_t(x) = a_t x + g_t^2\,\E\left[\frac{1}{\eps}x_\star-\sqrt{\frac{h_t}{\eps(1-h_t)}}\,z\;\Big|\;x_t = x\right].\]

\textbf{Step 6.} Since $\frac{g_t^2}{\eps h_t}X_t^{\rm F} = -g_t^2\nabla \log \rho^Y_t(X^{\rm F}_t)$, we can also write the F\"ollmer process in the following way:
\[{\rm d}X_t^{\rm F} = a_t X^{\rm F}_t \,{\rm d}t + g_t^2 \nabla \log \frac{\rho_t(X_t^{\rm F})}{\rho^Y_t(X^{\rm F}_t)} \,{\rm d}t +  g_t\,{\rm d}W_t, \quad X_0^{\rm F} = 0, \]
where $\rho_t$ is the density of $X^\Fo_t$ and $\rho^Y_t$ is the density of the reference process $Y_t$.

\textbf{Step 7.} We can show that 
\[\log \frac{\rho_t(x)}{\rho^Y_t(x)} = \log P^Y_{1-t}f(x)\, , \]
where $P_t^Y f(x) = \mathbb{E}[f(Y_1)|Y_{1-t}=x]$ and $f(x) = \frac{{\rm d}\mu_\star}{{\rm d}\mathrm{Law}(Y_1)}(x)$. 
One way to show this is to use the fact that $\log \frac{\rho_t(x)}{\rho^Y_t(x)}$ solves a Hamilton--Jacobi--Bellman equation. More precisely, we can write down the following two Fokker--Planck equations:
\begin{equation}
    \begin{aligned}
        &\partial_t \rho^Y_t(x) + \nabla \cdot \left(\rho^Y_t(x) v_t(x)\right) = \frac{1}{2}g_t^2 \Delta \rho^Y_t(x)\\
        &\partial_t \rho_t(x) + \nabla \cdot \left(\rho_t(x) \left(v_t(x)+g_t^2 \nabla \log \frac{\rho_t(x)}{\rho^Y_t(x)}\right)\right) = \frac{1}{2}g_t^2 \Delta \rho_t(x),
    \end{aligned}
\end{equation}
where $v_t(x) = a_t x$. From these two equations, we obtain
\[ \partial_t \log \frac{\rho_t(x)}{\rho^Y_t(x)} + \nabla \log\frac{\rho_t(x)}{\rho^Y_t(x)} \cdot v_t(x)  +\frac{1}{2}g_t^2\left|\nabla \log\frac{\rho_t(x)}{\rho^Y_t(x)}\right|^2 + \frac{1}{2}g_t^2\Delta \log \frac{\rho_t(x)}{\rho^Y_t(x)} = 0, \]
which is a Hamilton--Jacobi--Bellman equation. We have an explicit solution for this equation as follows:
\[\log \frac{\rho_t}{\rho^Y_t} = \log \mathbb{E}\left[\frac{{\rm d}\mu_\star}{{\rm d}\mathrm{Law}(Y_1)}(Y_1)\;\Big|\;Y_t = x\right] =  \log P^Y_{1-t}f. \]
Therefore, the F\"ollmer process can be written as 
\[{\rm d}X_t^{\rm F} = a_t X^{\rm F}_t \,{\rm d}t + g_t^2 \nabla \log P^Y_{1-t}f(X_t^{\rm F})\, {\rm d}t +  g_t\,{\rm d}W_t, \quad X_0^{\rm F} = 0\, . \]
This formula, involving $\nabla \log P^Y_{1-t}f(X_t^{\rm F})$, is commonly used in the literature as the F\"ollmer drift when the reference process is the Wiener process. We note this formula can also be directly derived using the perspective of Schr\"odinger's bridge and PDEs. For completeness, we present the derivation in Appendix~\ref{appendix-Schrodinger's bridges}.
\end{proof}

\subsection{A perspective from Schr\"odinger bridges}
\label{appendix-Schrodinger's bridges}
We provide a formal derivation of the F\"ollmer drift through the perspective of Schr\"odinger bridges.

The goal in the Schr\"odinger bridge problem is to find a drift $u$ such that the KL divergence between the path measures of the process $X$ and a given reference $Y$ is minimized among all processes with marginal distributions $\mu$ and $\nu$ at $t=0$ and $t=1$, respectively~\cite{leonard2014survey,chen2021stochastic}. 

Let us denote the path measures of the processes $X=(X_t)_{t\in [0,1]}$ and $Y=(Y_t)_{t\in [0,1]}$, which are both in $C([0,1], \mathbb{R}^d)$, as $\P_X$ and $\P_Y$, respectively. Consider $Y$ to be the solution to the following linear SDE:
\begin{equation}
    \label{eq:ref:foll-gen-append}
    {\rm d}Y_t = a_t Y_t \,{\rm d}t + g_t\,{\rm d}W_t, \quad Y_{t=0} = 0,
\end{equation}
and ${\rm d}X_t = a_t X_t \,{\rm d}t + u_t(X_t)\,{\rm d}t + g_t\,{\rm d}W_t$ where $u_t$ is to be determined. By Girsanov's theorem, the KL divergence of $\mathbb{P}_X$ from $\mathbb{P}_Y$ is given by
\begin{equation}
    \label{eq:KL:path:1}
    \mathrm{KL}[\P_X\Vert \P_Y] = \frac12\int_0^1 |g_t|^{-2} \mathbb{E}[|u_t(X_t)|^2]\, {\rm d}t.
\end{equation}
Therefore, minimizing this KL divergence subject to the constraints can be phrased as the following stochastic control problem:
\begin{equation}
\label{eq:KL:min:sb}
    \begin{aligned}
        \min_{u} \quad &\frac12\int_0^1 |g_t|^{-2} \mathbb{E}[|u_t(X_t)|^2]\, {\rm d}t, \\
        \text{s.t.} \quad  &{\rm d}X_t = a_t X_t \,{\rm d}t + u_t(X_t)\,{\rm d}t + g_t\,{\rm d}W_t, \quad X_0 \sim \mu,\quad X_1 \sim \nu.
    \end{aligned}
\end{equation}
The F\"ollmer process we consider is a particular instantiation of~\eqref{eq:KL:min:sb} when $\mu = \delta_{0}$ and $\nu = \mu_\star$.

If we denote by $\rho_t(x)$ the density of $X_t$, and write $v_t(x) = a_t x$, then the minimization problem in~\eqref{eq:KL:min:sb} can be formulated as
\begin{equation}
\label{eq:KL:min:sb:2}
    \begin{aligned}
        \min_{\rho}\quad &\frac12\int_0^1 |g_t|^{-2} \int_{\R^d}|u_t(x)|^2 \rho_t(x)\, {\rm d}x\, {\rm d}t, \\
        \text{s.t.} \quad & \partial_t \rho_t+ \nabla \cdot((u_t+v_t)\rho_t) = \tfrac12 g^2_t \Delta \rho_t, \quad \rho_0(x) = \delta_0(x),\quad \rho_1 = \mu_\star.
    \end{aligned}
\end{equation}
By Lagrange duality, the optimality condition for~\eqref{eq:KL:min:sb:2} is that there exists a function $\lambda_t(x)$ such that $u_t(x)=g_t^2\nabla \lambda_t(x)$, and the pair $(\rho,\lambda)$ solves the coupled system of Fokker--Planck and Hamilton--Jacobi--Bellman equations:
\begin{equation}
\label{eq:el:1}
\begin{cases}
    \begin{aligned}
        &\partial_t \rho_t +  \nabla\cdot ((g^2_t\nabla \lambda_t+v_t) \rho_t ) = \tfrac{1}{2}g^2_t\Delta \rho_t\\
        &\partial_t\lambda_t + \tfrac{1}{2}g^2_t|\nabla \lambda_t|^2 + \nabla \lambda_t \cdot v_t = -\tfrac{1}{2}g^2_t\Delta \lambda_t,
    \end{aligned}
    \end{cases}
\end{equation}
to be solved with the boundary conditions $\rho_0(x) = \delta_0(x)$ and $\rho_1 = \mu_\star$.

Through the transformation 
\begin{equation}
    \label{eq:transf}
     \begin{aligned}
    &\Phi_t(x) = \exp(\lambda_t(x)), \quad &&\hat{\Phi}_t(x) = \rho_t(x) \exp(-\lambda_t(x)) \\
    \Leftrightarrow \quad
    &\lambda_t(x) = \log  \Phi_t(x), \quad &&\rho_t(x) = \Phi_t(x) \hat{\Phi}_t(x)
    \end{aligned}
\end{equation}
we can turn \eqref{eq:el:1} into
\begin{equation}
\label{eq:el:2}
\begin{cases}
    \begin{aligned}
        & \partial_t \Phi_t + \nabla \Phi_t \cdot v_t = -\tfrac{1}{2}g^2_t\Delta \Phi_t\\
        & \partial_t \hat{\Phi}_t + \nabla \cdot (\hat{\Phi}_t v_t) = \tfrac{1}{2}g^2_t\Delta \hat{\Phi}_t.
    \end{aligned}
\end{cases}
\end{equation}
to be solved with the boundary conditions $\Phi_0(x)\hat{\Phi}_0(x) = \delta_0(x)$ and $\Phi_1(x)\hat{\Phi}_1(x)= \mu_\star(x)$. 

In general, one does not have an explicit formula for $\Phi_t(x)$ and $\hat{\Phi}_t(x)$. However, \eqref{eq:el:2} can be solved for the particular boundary conditions of the F\"ollmer process.

To see how, let us set $\hat{\Phi}_0(x) = \delta_0(x)$. Since $v_t(x) = a_t x$, we get that $\hat{\Phi}_t(\cdot)$ is the density of 
\begin{equation}
    \label{eq:ref:foll-appendix}
    {\rm d}Y_t = a_t Y_t \,{\rm d}t + g_t\,{\rm d}W_t, \quad Y_{t=0} = 0 .
\end{equation}
Now, to satisfy the terminal condition at $t=1$, we must set $\Phi_1(x) = \mu_\star(x)/\hat{\Phi}_1(x)$. To conform with standard notation, let us denote this density ratio as $f$, i.e.,
\begin{equation}
    \label{eq:f:def}
    f(x) \equiv \Phi_1(x) = \frac{{\rm d}\mu_\star}{{\rm d}\mathrm{Law}(Y_1)}(x) .
\end{equation}
Using the Feynman--Kac formula, we can express the solution of the equation for $\Phi$ in \eqref{eq:el:2} for this terminal condition as
\begin{equation}
    \label{eq:phi:h:s}
    {\Phi}_t(x) = \mathbb{E}[f(Y_1)|Y_{t}=x] .
\end{equation} 
Furthermore, we can verify that $\Phi_0(0) = \mathbb{E}[f(Y_1)] = 1$ since $f$ is the density ratio. Therefore, the condition $\Phi_0(x)\hat{\Phi}_0(x) = \delta_0(x)$ is satisfied. Thus, the constructed $\hat{\Phi}_t$ and $\Phi_t$ are indeed the solutions to~\eqref{eq:el:2}; similar derivations can be found in \cite{wang2021deep}.

Coming back to the minimization problem~\eqref{eq:KL:min:sb:2}, the derivation above shows that the optimal drift is $b^\Fo_t(x) = a_t x + g^2_t \nabla \log \Phi_t(x)$, with $\Phi$ given in~\eqref{eq:phi:h:s}. It is common to introduce the notation $P_t^Y f(x) = \mathbb{E}[f(Y_1)|Y_{1-t}=x]$. With this notation, the SDE associated with the optimal drift can be written as
\begin{equation}
\label{eq:sde:foll:2}
{\rm d}X^\Fo_t = a_t X_t^\Fo \,{\rm d}t + g^2_t\nabla \log P_{1-t}^Y f(X_t^{\rm F})\,  {\rm d}t + g_t\,{\rm d}W_t, \quad X_{t=0}^{\Fo} = 0 .
\end{equation}
The solution to this SDE is the F\"ollmer process, and by construction its solutions satisfy $X_{t=1}^{\Fo} \sim \mu_\star$.

\section{Singularity Discussions on The Baseline Drift at Initial Time}
\label{sec:singular behaviors of baseline SDEs at initial time}
This section discusses the regularity of the drift in the baseline diffusion under more general assumptions beyond Assumption \ref{assum:bounded-support-smoothed-assumption}, and identifies a special condition $\dot\beta_0 = 0$.

Recall that under Assumption \ref{assum:bounded-support-smoothed-assumption}, the drift $b_t(x)$ is spatially Lipschitz, uniformly in time. In particular $b_t(x)$ is bounded and $\lim_{t \to 0^+} b_t(x)$ is well-defined at $t=0$. However, this is not true for general target distributions. To see this, we present two examples below.
Recall the setting $x_t  = \beta_t  x_\star + \sigma_t \sqrt{t}z, t\in[0,1]$, where $x_\star \sim \mu_\star$, $z\sim \mathsf {\sf N}(0,\mathrm{I}), z\perp x_\star$. Moreover, $b_t(x) =  \dot\beta_t  \mathbb{E}[x_\star|x_t = x] + \dot \sigma_t   \mathbb{E}[\sqrt{t}z|x_t = x].$ We denote $\rho^\star$ the density of $\mu_\star$ regarding the Lebesgue measure.
\begin{example}
\label{prop: dot beta for Gaussian}
    Suppose $\rho^\star$ is a Gaussian density ${\sf N}(m,C)$.   Then for any $x \in \mathbb{R}^d$, it holds that
    \begin{itemize}
        \item $\lim_{t\to 0}\mathbb{E}[x_\star|x_t = x] = \frac{\dot{\beta}_0}{\sigma^2_0}Cx + m $;
        \item $\lim_{t\to 0}\mathbb{E}[\sqrt{t}z|x_t = x] = \frac{1}{\sigma_0}x$;
        \item $\lim_{t\to 0} b_t(x) = \dot{\beta}_0 m + \big(\frac{\dot{\beta}^2_0}{\sigma^2_0}C + \frac{\dot{\sigma}_0}{\sigma_0}\mathrm{I}\big)x$.
    \end{itemize}
\end{example}
\begin{example}
\label{prop: dot beta for exponential}
    Let $d=1$ and $\rho^\star(x_\star) \propto \exp(-|x_\star|)$. Assume $\dot{\beta}_0 > 0$. Let $c = \frac{\sigma^2_0}{\dot{\beta}_0} > 0$. Then, for $|x|>c$, it holds that
    \begin{itemize}
        \item $\lim_{t\to 0} \beta_t \mathbb{E}[x_\star|x_t = x] = (1-\frac{c}{|x|})x$;
        \item $\lim_{t\to 0} \mathbb{E}[\sqrt{t}z|x_t = x] = \frac{c}{\sigma_0} \frac{x}{|x|}$;
        \item $\lim_{t\to 0} b_t(x) = \infty$.
    \end{itemize}
\end{example}
The first example of Gaussian distributions leads to a well defined limit at $t=0$. This is not surprising since Gaussian distributions are covered by Assumption \ref{assum:bounded-support-smoothed-assumption}. However, for exponential distributions, if $\dot\beta_0 >0$, then the drift $b_t(x)$ blows up when $t\to 0$. We also note that the singularity of the drift here may not correspond to a negative restoring force (as in the F\"ollmer drift discussed in Appendix \ref{appendix-singular-drift-discussion}), since
\begin{equation}
    \lim_{t \to 0} b_t(x) = \lim_{t \to 0} (\dot\beta_t  \mathbb{E}[x_\star|x_t = x] + \dot \sigma_t   \mathbb{E}[\sqrt{t}z|x_t = x]) = \lim_{t\to 0}\frac{\dot\beta_t}{\beta_t}(1-\frac{c}{|x|})x + \dot \sigma_0  \frac{c}{\sigma_0} \frac{x}{|x|} = +\infty
\end{equation}
for $|x|>c$. 
The above two examples indicate that the tail behaviors of the target distribution matter in determining the behavior of $b_t(x)$ when $t\to 0$. 
\begin{proof}[Proof of Example~\ref{prop: dot beta for Gaussian}]
    Note that if two random vectors $X, Y$ are joint Gaussians, then there is an explicit formula for the conditional expectation \[\E[X|Y] = \mathrm{Cov}(X,Y)\mathrm{Cov}(Y,Y)^{-1}(Y-\E[Y])+\E[X].\] Using this formula, we can calculate the following:
\begin{equation}
    \begin{aligned}
        &\E[x_\star|x_t = x] = \beta_t C\big(\beta^2_tC+\sigma^2_t t \,\mathrm{I}\big)^{-1}\big(x -\beta_t m\big)+m,\\
        &\E[z|x_t = x]= \sqrt{t}\,\sigma_t \big(\beta^2_tC+\sigma^2_t t \,\mathrm{I}\big)^{-1}\big(x-\beta_t m\big) .
    \end{aligned}
\end{equation}
As $\beta_t $ is differentiable, it holds that $\beta_t =\dot{\beta}_0 t+o(t)$ as $t \to 0$. Therefore, we have the following limits as $t\to 0$:
\begin{equation}
    \begin{aligned}
        &\lim_{t\to 0}\E[x_\star|x_t = x] = \frac{\dot{\beta}_0}{\sigma^2_0}Cx+m,\\
        &\lim_{t\to 0}\E[\sqrt{t}\,z|x_t = x]= \frac{1}{\sigma_0}x .
    \end{aligned}
\end{equation}
Combining them with the formula for $b_t(x)$ yields
\begin{equation}
    \lim_{t\to 0} b_t(x) = \dot{\beta}_0 m + \Big(\frac{\dot{\beta}^2_0}{\sigma^2_0}C + \frac{\dot{\sigma}_0}{\sigma_0}\mathrm{I} \Big)x .
\end{equation}
\end{proof}

\begin{proof}[Proof of Example~\ref{prop: dot beta for exponential}]
    To show this example, we state a lemma:
    \begin{lemma}
    \label{lemma-exponential}
        Let $X$ be a random variable in $\mathbb{R}$ with probability density $\rho(x) \propto \exp(-|x|)$. Let $Z \sim \sfN(0,1)$ be independent of $X$, and let $p_t, q_t$ be non-negative continuous functions satisfying $p_0=q_0=0$ and positive on $(0,\epsilon)$ for some $\epsilon > 0$. Consider any continuous function $y_t$ such that $\lim_{t\to 0} y_t = y_0$. If $\lim_{t \to 0} \frac{q^2_t}{p_t} = c <\infty$, then for $|y_0| > c$, $\lim_{t \to 0} p_t \E[X|p_t X + q_t Z = y_t] = (1-\frac{c}{|y_0|})y_0$.
    \end{lemma}
    We now prove the example by employing this lemma; the proof of the lemma is provided subsequently.
    
    By definition, $\E[x_\star|x_t = x] = \E[x_\star | \beta_t x_\star+\sqrt{t}\,\sigma_t z = x]$. Using the lemma with $p_t = \beta_t$, $q_t = \sqrt{t}\,\sigma_t$, $y_t \equiv x$ leads to \[\lim_{t\to 0} \beta_t \E[x_\star|x_t = x] = \Big(1-\frac{c}{|x|}\Big)x\, , \]
    where $c := \lim_{t \to 0}\frac{q^2_t}{p_t} = \lim_{t \to 0} \frac{t\sigma^2_t}{\beta_t} = \frac{\sigma^2_0}{\dot{\beta}_0}$. The above formula holds for $|x|>c$.
    Furthermore, using the relation 
    \[\mathbb{E}[\sqrt{t}\,z|x_t = x] = \frac{x - \beta_t  \mathbb{E}[x_\star|x_t = x]}{\sigma_t },\]
    we get the limit
    \[\lim_{t\to 0} \mathbb{E}[\sqrt{t}\,z|x_t = x] = \frac{c}{\sigma_0} \frac{x}{|x|}.\]
    Note that $b_t(x) =  \dot\beta_t  \mathbb{E}[x_\star|x_t = x] + \dot \sigma_t   \mathbb{E}[\sqrt{t}\,z|x_t = x]$. From the above result we get that \[\lim_{t \to 0} \dot\beta_t  \mathbb{E}[x_\star|x_t = x] = \lim_{t \to 0} \frac{\dot\beta_t }{\beta_t }\cdot  \beta_t \mathbb{E}[x_\star|x_t = x] = \infty\] as $\beta_0 = 0$, $\dot{\beta}_0 \neq 0$, and \[\lim_{t \to 0} \dot \sigma_0   \mathbb{E}[\sqrt{t}\,z|x_t = x] = \frac{c\dot \sigma_0}{\sigma_0} \frac{x}{|x|} .\]
    Therefore, $\lim_{t\to 0} b_t(x) = \infty$.
\end{proof}
\begin{proof}[Proof of Lemma~\ref{lemma-exponential}]
        By definition of conditional expectations, for $t \in (0,\epsilon)$, we have
    \begin{equation}
    \begin{aligned}
        &\E[X|p_t X + q_tZ = y_t] \\
        = &\frac{\int x \exp\big(-\frac{|y_t-p_tx|^2}{2q^2_t}-|x|\big)\, {\rm d}x}{\int \exp\big(-\frac{|y_t-p_tx|^2}{2q^2_t}-|x|\big)\, {\rm d}x}\\
        =& \frac{\int \frac{x}{p_t} \exp\big(-\frac{|y_t-x|^2}{2q^2_t}-\frac{|x|}{p_t}\big)\, {\rm d}x}{\int \exp\big(-\frac{|y_t-x|^2}{2q^2_t}-\frac{|x|}{p_t}\big)\, {\rm d}x}\\
        =& \frac{\int_0^{+\infty} \frac{x}{p_t} \exp\big(-\frac{|y_t-x|^2}{2q^2_t}-\frac{x}{p_t}\big)\, {\rm d}x + \int_{-\infty}^0 \frac{x}{p_t} \exp\big(-\frac{|y_t-x|^2}{2q^2_t}+\frac{x}{p_t}\big)\, {\rm d}x }{\int_0^{\infty} \exp\big(-\frac{|y_t-x|^2}{2q^2_t}-\frac{x}{p_t}\big)\, {\rm d}x + \int_{-\infty}^0 \exp\big(-\frac{|y_t-x|^2}{2q^2_t}+\frac{x}{p_t}\big)\, {\rm d}x} .
    \end{aligned}
    \end{equation}
    Note that \[-\frac{|y_t-x|^2}{2q^2_t}-\frac{x}{p_t} = -\frac{1}{2q^2_t}\Big(x-\big(y_t-\frac{q^2_t}{p_t}\big)\Big)^2-\frac{y_t}{p_t}+\frac{q^2_t}{2p^2_t} ,\]
    and
    \[-\frac{|y_t-x|^2}{2q^2_t}+\frac{x}{p_t} = -\frac{1}{2q^2_t}\Big(x-\big(y_t+\frac{q^2_t}{p_t}\big)\Big)^2+\frac{y_t}{p_t}+\frac{q^2_t}{2p^2_t} .\]
    We can calculate
    \begin{equation*}
    \begin{aligned}
        &\int_0^{+\infty} \frac{x}{p_t} \exp\Big(-\frac{|y_t-x|^2}{2q^2_t}-\frac{x}{p_t}\Big)\, {\rm d}x \\
        = &\frac{1}{p_t}\exp\Big(-\frac{y_t}{p_t}+\frac{q^2_t}{2p^2_t}\Big)\int_0^{+\infty} x \exp\Big(-\frac{1}{2q^2_t}\big(x-(y_t-\frac{q^2_t}{p_t})\big)^2\Big)\, {\rm d}x\\
        = & \frac{1}{p_t}\exp\Big(-\frac{y_t}{p_t}+\frac{q^2_t}{2p^2_t}\Big) \left(q_t\big(y_t-\frac{q^2_t}{p_t}\big)\sqrt{2\pi}\,F\Big(\frac{y_t-q_t^2/p_t}{q_t}\Big) + q^2_t\exp\Big(-\frac{1}{2q^2_t}\big(y_t-\frac{q^2_t}{p_t}\big)^2\Big)\right) ,
    \end{aligned}
    \end{equation*}
    where $F$ is the cumulative distribution function (CDF) of the standard normal, i.e., $F(x) = \mathbb{P}\{Y \leq x\}$ for $Y \sim \sfN(0,1)$. Similarly,
    \begin{equation*}
    \begin{aligned}
        &\int_{-\infty}^{0} \frac{x}{p_t} \exp\Big(-\frac{|y_t-x|^2}{2q^2_t}+\frac{x}{p_t}\Big)\, {\rm d}x \\
        = & \frac{1}{p_t}\exp\Big(\frac{y_t}{p_t}+\frac{q^2_t}{2p^2_t}\Big) \left(q_t\big(y_t+\frac{q^2_t}{p_t}\big)\sqrt{2\pi}\,F\Big(-\frac{y_t+q^2_t/p_t}{q_t}\Big) - q^2_t\exp\Big(-\frac{1}{2q^2_t}\big(y_t+\frac{q^2_t}{p_t}\big)^2\Big)\right) ,
    \end{aligned}
    \end{equation*}
    \begin{equation*}
    \begin{aligned}
        &\int_0^{\infty} \exp\Big(-\frac{|y_t-x|^2}{2q^2_t}-\frac{x}{p_t}\Big)\, {\rm d}x = \exp\Big(-\frac{y_t}{p_t}+\frac{q^2_t}{2p^2_t}\Big) q_t\sqrt{2\pi}\,F\Big(\frac{y_t-q^2_t/p_t}{q_t}\Big) ,
    \end{aligned}
    \end{equation*}
    \begin{equation*}
    \begin{aligned}
        \int_{-\infty}^{0} \exp\Big(-\frac{|y_t-x|^2}{2q^2_t}+\frac{x}{p_t}\Big)\, {\rm d}x = \exp\Big(\frac{y_t}{p_t}+\frac{q^2_t}{2p^2_t}\Big) q_t\sqrt{2\pi}\,F\Big(-\frac{y_t+q^2_t/p_t}{q_t}\Big) .
    \end{aligned}
    \end{equation*}
    Using the above equations, we get
     \begin{equation}
     \label{eqn: lemma2 eqn 213}
    \begin{aligned}
        &\E[X|p_t X + q_tZ = y_t] \\
        =& \frac{1}{p_t} \frac{\exp(-\frac{2y_t}{p_t})(y_t-\frac{q^2_t}{p_t})F(\frac{y_t-q^2_t/p_t}{q_t}) + (y_t+\frac{q^2_t}{p_t})F(-\frac{y_t+q^2_t/p_t}{q_t})}{\exp(-\frac{2y_t}{p_t})F(\frac{y_t-q_t^2/p_t}{q_t}) + F(-\frac{y_t+q^2_t/p_t}{q_t})} .
    \end{aligned}
    \end{equation}
    Let $\lim_{t \to 0} \frac{q^2_t}{p_t} = c <\infty$. Suppose $\lim_{t \to 0} y_t = y_0 > c$; the proof for $y_0 < -c$ is similar. Then, $\frac{y_t-q_t^2/p_t}{q_t} > 0$ for sufficiently small $t$. Thus, $F(\frac{y_t-q^2_t/p_t}{q_t}) \geq F(0) = \frac{1}{2}$. Therefore, using the estimate for the Gaussian CDF that $F(z)\leq \frac{1}{\sqrt{2\pi}}\frac{1}{|z|}\exp(-z^2/2)$ for $z<0$, we obtain
    \begin{equation*}
    \begin{aligned}
        \left|\frac{F(-\frac{y_t+q^2_t/p_t}{q_t})}{\exp(-\frac{2y_t}{p_t})F(\frac{y_t-q^2_t/p_t}{q_t})}\right|&\leq \frac{2}{\sqrt{2\pi}}\frac{q_t}{y_t+q^2_t/p_t}\exp\Big(-\frac{1}{2}\big(\frac{y_t+q^2_t/p_t}{q_t}\big)^2+\frac{2y_t}{p_t}\Big)\\
        & = \frac{2}{\sqrt{2\pi}}\frac{q_t}{y_t+q^2_t/p_t} \exp\Big(-\frac{1}{2}\big(\frac{y_t-q^2_t/p_t}{q_t}\big)^2\Big) .
    \end{aligned}
    \end{equation*}
    The last terms converge to zero as $t \to 0$ since $\lim_{t \to 0} q_t = 0$ and $\lim_{t \to 0} \frac{q^2_t}{p_t} = c <\infty$. Combining this fact with \eqref{eqn: lemma2 eqn 213} implies that
    \begin{equation*}
        \lim_{t \to 0} p_t\E[X|p_t X + q_tZ = y_t] = \lim_{t \to 0} \frac{(y_t-\frac{q^2_t}{p_t}) + (y_t+\frac{q^2_t}{p_t})\times 0}{ 1+ 0} = y_0 - c .
    \end{equation*}
    When $y_0 < -c$, we will get the limit to be $y_0 + c$. So in a unified way, the limit is $(1-\frac{c}{|y_0|})y_0$.
\end{proof}

\subsection{The condition $\dot{\beta}_0=0$}
In the following, we show that assuming $\dot{\beta}_0=0$ is \textit{sufficient and necessary} to obtain a well-defined limit for all exponential-tailed distributions. 

\begin{assumption}
    The target density $\rho^\star$ is exponential tailed so there exists constants $C_1,C_2 > 0$, such that
    \[\rho^\star(x)\leq C_2\exp(-C_1|x|) , \]
    for any $x \in \mathbb{R}^d$. 
    \label{assump:exponential-tails}
\end{assumption}
The assumption of exponential tails is needed for technical reasons, ensuring the validity of the step that involves the interchange of limits and integrations in the proof of Proposition \ref{th:decokmp}.
\begin{proposition}
\label{th:decokmp}
    Let Assumption \ref{assump:exponential-tails} hold. Consider the setting of interpolants in Theorem \ref{thm:1:b} and assume $\dot\beta_0 = 0$.
It holds that the limits at $t=0$ of $b_t(x)$ and $\nabla b_t(x)$ are finite:
\[\lim_{t \to 0} b_t(x) = \frac{\dot\sigma_0}{\sigma_0}x\quad \text{and} \quad \lim_{t \to 0} \nabla_x b_t(x) = \frac{\dot\sigma_0}{\sigma_0}\mathrm{I}.\]
\end{proposition}
\begin{proof}[Proof of Proposition~\ref{th:decokmp}]
By definition $x_t  = \beta_t  x_\star + \sigma_t \sqrt{t}z$ and $b_t(x) = \mathbb{E}[\dot\beta_t x_\star+ \dot \sigma_t  \sqrt{t} \,z|x_t  = x]$. This implies that 
\begin{equation}
    \label{eq:etaz:1}
    \forall (t,x) \in (0,1)\times \R^d \ \text{with} \ \sigma_t >0: \ \mathbb{E}[\sqrt{t}z|x_t = x] = \frac{x - \beta_t  \mathbb{E}[x_\star|x_t = x]}{\sigma_t }.
\end{equation}
Using the proof of Proposition \ref{prop:cond-exp-gradient}, we have the formula
\begin{equation}
    \mathbb{E}[x_\star|x_t = x] = \frac{\int_{\R^d} x_\star \rho^\star(x_\star) \exp(-\frac12M_t |x_\star|^2 +m_t x_\star \cdot x ) {\rm d}x_\star}{\int_{\R^d} \rho^\star(x_\star) \exp(-\frac12M_t |x_\star|^2 +m_t x_\star \cdot x) {\rm d}x_\star}\, .
\end{equation}
Here
    $M_t = \frac{\beta^2_t}{t \sigma^2_t}, m_t = \frac{\beta_t }{t \sigma^2_t}$.
  To establish the limit at $t=0$, notice that, since  $\beta\in C^1([0,1)$ and $\dot \beta_0 = 0$, we must have $\beta_t  = o(t)$ as $t\to 0$. Then using the fact that $\sigma\in C^1([0,1])$ and $\sigma_0>0$, we have $
    \lim_{t\to0} M_t = \lim_{t\to0} m_t = 0$.
As a result,
\[\lim_{t \to 0} \mathbb{E}[x_\star|x_t = x] = \E[x_\star] \quad \text{and} \quad \lim_{t \to 0} \mathbb{E}[\sqrt{t}z|x_t = x] = \frac{x}{\sigma_0}\, .\]
In the above derivation, we need to verify the interchange of limits and integrations. This is guaranteed by using Assumption \ref{assump:exponential-tails} and the Lebesgue dominated convergence theorem: for a fixed $x$, when $t$ is sufficiently small, the factor $\rho^\star(x_\star) \exp(-\frac12M_t |x_\star|^2 +m_t x_\star \cdot x)$  is dominated by $\rho^\star(x_\star) \exp(C_1|x_\star|/2)$, which is integrable as a function of $x_\star$ due to our Assumption \ref{assump:exponential-tails}. Consequently, $\lim_{t \to 0} b_t(x) = \frac{\dot\sigma_0}{\sigma_0}x$.

To analyze the limit of $\nabla_x b_t(x)$, we use the argument in the proof of Proposition \ref{prop:cond-exp-gradient} again to get
\begin{equation}
    \nabla_x \mathbb{E}[x_\star|x_t = x] = m_t\frac{P_t(x)Q_t(x) -R_t(x)R_t(x)^T}{|Q_t(x)|^2}
\end{equation}
where 
\begin{equation*}
    \begin{aligned}
        P_t(x) &= \int_{\R^d} x_\star x_\star^T \rho^\star(x_\star) \exp(-\frac12M_t |x_\star|^2 +m_t x_\star \cdot x ) {\rm d}x_\star\\
        Q_t(x) & = \int_{\R^d} \rho^\star(x_\star) \exp(-\frac12M_t |x_\star|^2 +m_t x_\star \cdot x) {\rm d}x_\star\\
        R_t(x) &= \int_{\R^d} x_\star \rho^\star(x_\star) \exp(-\frac12M_t |x_\star|^2 +m_t x_\star \cdot x) {\rm d} x_\star\, .
    \end{aligned}
\end{equation*}
Using the Lebesgue dominated convergence theorem, we get
\begin{equation}
    \lim_{t\to 0} \frac{P_t(x)Q_t(x) -R_t(x)R_t(x)^T}{|Q_t(x)|^2} = \text{Cov}(x_\star).
\end{equation}
Thus, $\lim_{t \to 0} \nabla_x \mathbb{E}[x_\star|x_t = x]= 0$ since $\lim_{t\to 0} m_t=0$. Using the formula in \eqref{eq:etaz:1}, we get $\lim_{t \to 0} \nabla_x\mathbb{E}[\sqrt{t}z|x_t = x] = \frac{1}{\sigma_0}\mathrm{I}$. Therefore, $\lim_{t\to 0} \nabla_x b_t(x) = \frac{\dot\sigma_0}{\sigma_0}\mathrm{I}$.
The proof is complete.
\end{proof}
Therefore, the condition $\dot{\beta}_0 = 0$ plays a unique role in the point-source setting. It is sufficient to guarantee bounded drift at the initial time for a broad class of exponential-tailed target distributions, and it is essentially necessary in the sense that relaxing it leads to unavoidable singular behavior. 

This condition also appears in the optimally tuned diffusions and the associated F\"ollmer processes, where it leads to an initial-time singularity that takes the form of a negative infinity restoring drift; see Remark~\ref{remark-assumption-g-beta} and Appendix~\ref{appendix-singular-drift-discussion}.

From a practical perspective, schedules satisfying $\dot{\beta}_0 = 0$ have been successfully employed in applications such as probabilistic forecasting of turbulent flows \cite{chen2024probabilistic} which have been observed to outperform schedules not satisfying this condition. In such settings, the KL-optimal diffusion coefficient $g_t = g_t^{\rm F}$ corresponding to the F\"ollmer process has also been observed to yield improved statistical performance over both the baseline choice $g_t = \sigma_t$ and deterministic ODE-based alternatives for certain observables. These empirical observations are consistent with the variational optimality established in the present work, and further illustrate the relevance of the theoretical analysis to practical generative modeling problems.

For completeness, we also include a numerical example below regarding the optimization stability arising from the condition $\dot{\beta}_0=0$.

\subsection{Experiments: Improved optimization stability with $\dot{\beta}_0=0$}
\label{sec-dotbeta-0-improved-training}
We present numerical experiments\footnote{The code is available at \url{https://github.com/yifanc96/GenerativeDynamics-NumericalDesign.git}.} for probabilistic forecasting of stochastic Navier-Stokes \cite{chen2024probabilistic} to demonstrate that the choice of $\dot{\beta}_0=0$ can lead to more stable optimization and better generation results. 

More precisely, consider the 2D Navier-Stokes equations with random forcing on the torus $\mathbb{T}^2 = [0,2\pi]^2$. In vorticity formulation, the governing equation is:
\begin{equation}
    \label{eq:2D_vorticity_NS}
    \mathrm{d}\omega + v \cdot \nabla\omega\,\mathrm{d}t = \nu \Delta\omega\,\mathrm{d}t - \alpha\omega\,\mathrm{d}t + \varepsilon\,\mathrm{d}\eta\, .
\end{equation}
The velocity field $v = \nabla^{\perp} \psi = (-\partial_y\psi, \partial_x\psi)$ derives from the stream function $\psi$, related to vorticity through $-\Delta \psi = \omega$. The stochastic forcing $\mathrm{d}\eta$ represents white-in-time noise applied to a finite set of Fourier modes. We set the physical parameters to $\nu=10^{-3}$, $\alpha=0.1$, and $\varepsilon = 1$, following \cite{chen2024probabilistic}. Under these conditions, equation~\eqref{eq:2D_vorticity_NS} is rigorously ergodic with a unique invariant measure~\cite{hairer2006ergodicity}. In contrast to the deterministically-driven Kolmogorov-flow example of Section~\ref{sec:numerics-ns}, this is an inherently probabilistic forecasting problem: the dynamics themselves are stochastic, so the conditional law $\omega_{t+\tau} \mid \omega_t$ is non-degenerate even with a fully observed current state, and probabilistic-forecasting criteria can be evaluated against the simulator's ground-truth ensemble.
\begin{figure}[ht]
    \centering
    \begin{overpic}[width=0.24\linewidth]{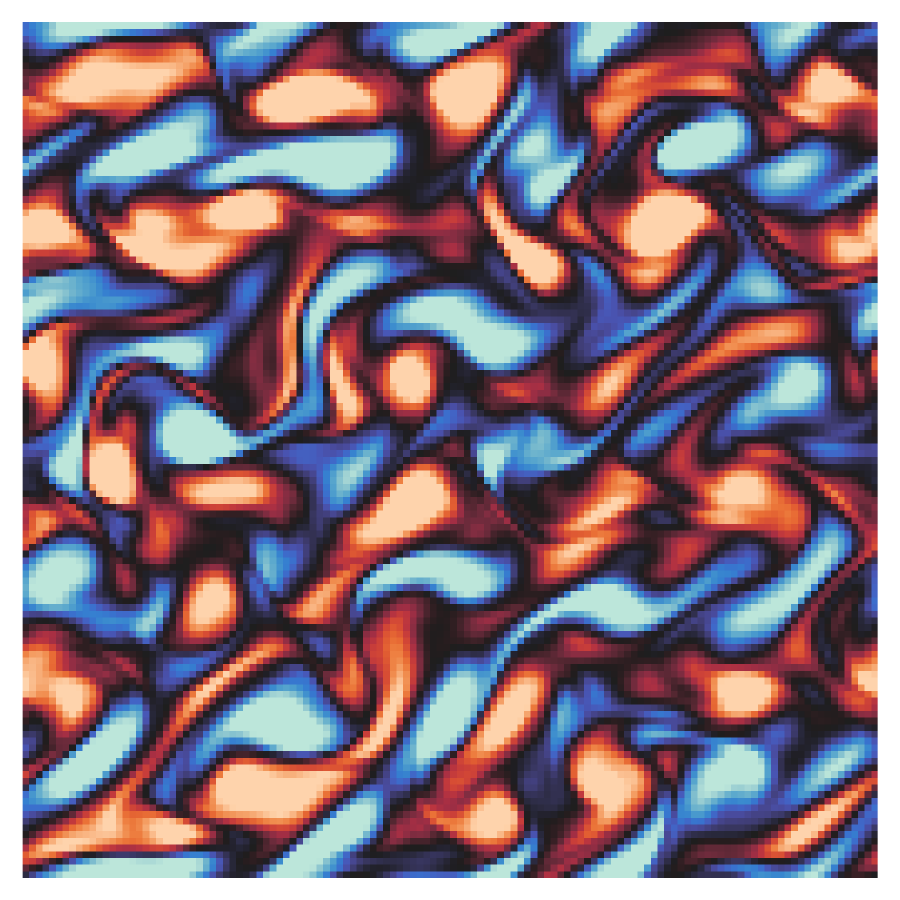}\put(45,-8){\small $t$}
    \end{overpic}
    \begin{overpic}[width=0.24\linewidth]{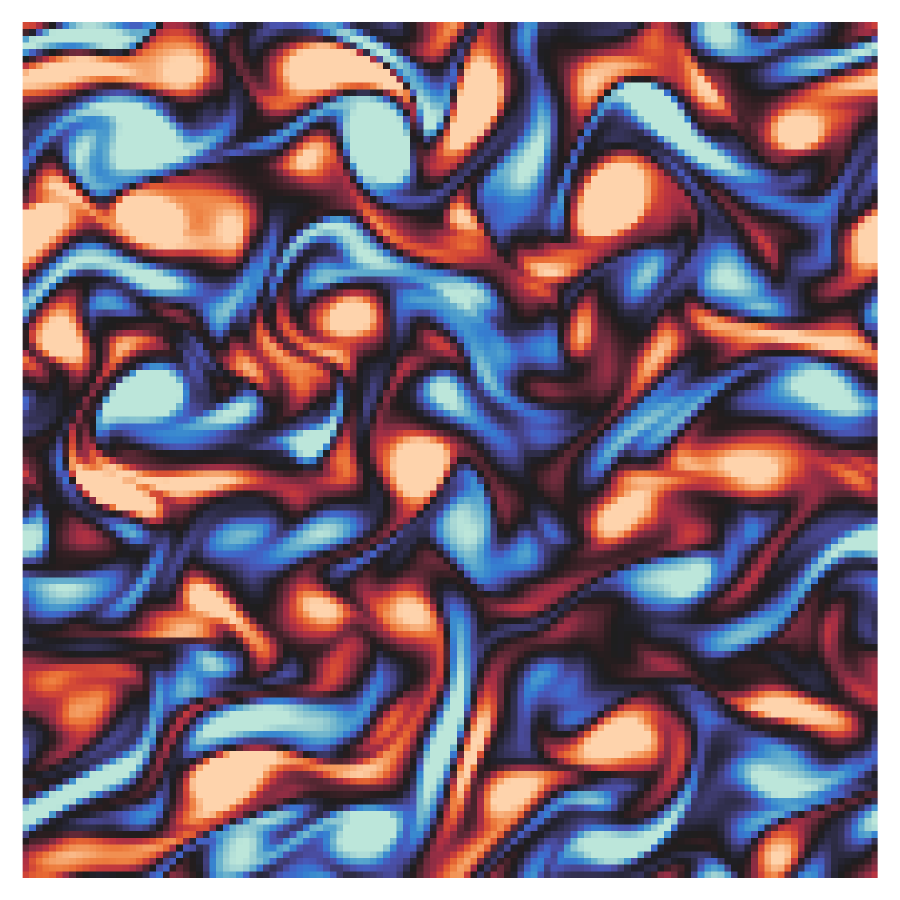}\put(40,-8){\small $t+1$}
    \end{overpic}
    \begin{overpic}[width=0.24\linewidth]{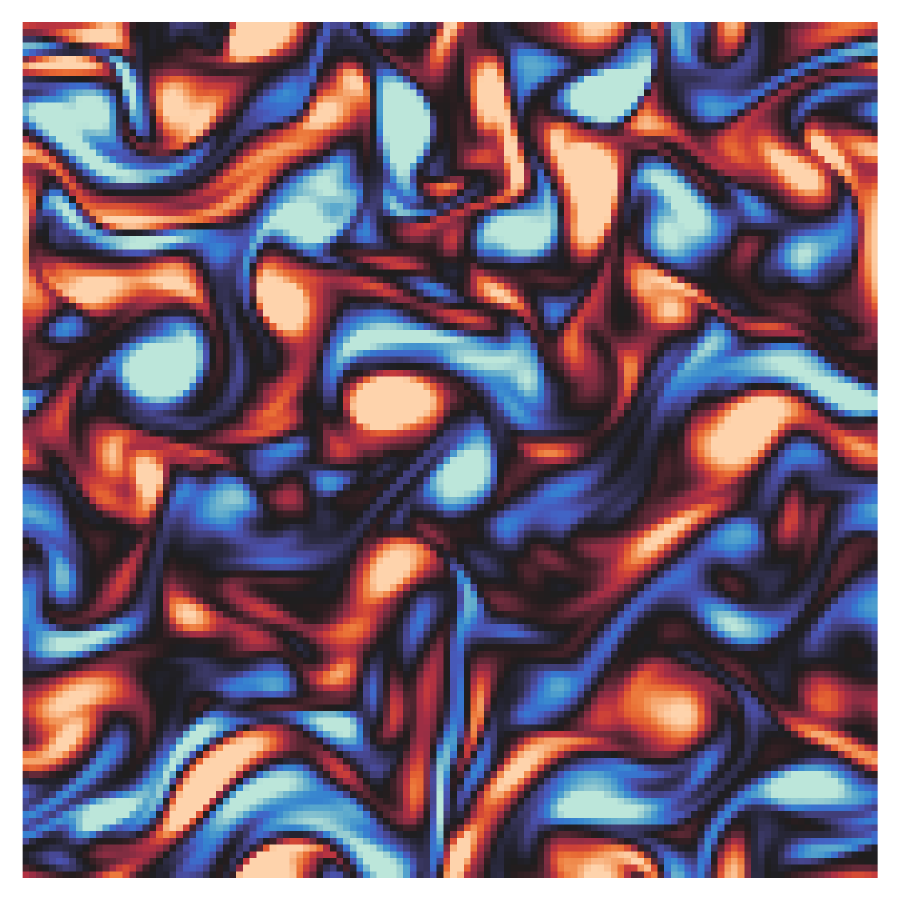}\put(40,-8){\small $t+2$}
    \end{overpic}
    \begin{overpic}[width=0.24\linewidth]{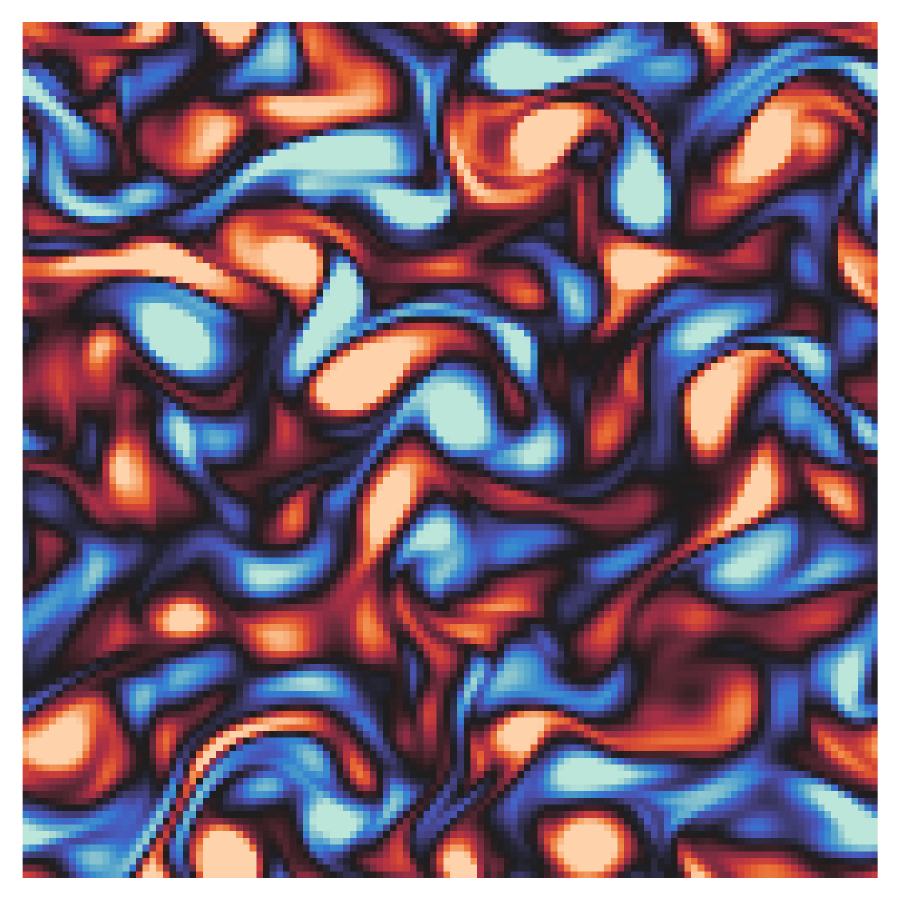}\put(40,-8){\small $t+3$}
    \end{overpic}
    \caption{A sample of trajectory of the stochastic Navier-Stokes equation}
    \label{fig:NStrajectory}
\end{figure}

We generate training data through long-term numerical simulation on a fine computational grid using the pseudo-spectral method. Following the simulation setup detailed in \cite{chen2024probabilistic}, we use a timestep of $\Delta t = 10^{-4}$ and a grid resolution of $256\times 256$, storing snapshots at regular intervals of $\Delta t = 0.5$. Our objective is to forecast the solution to \eqref{eq:2D_vorticity_NS} at time $t+\tau$ given its solution at time $t$, where $\tau = 1$, after the system has reached a statistically steady state. The autocorrelation function for such a lag is approximately $0.25$ for the setting we consider \cite{chen2024probabilistic}. A sample of trajectory is shown in Figure \ref{fig:NStrajectory}.

To reduce computational memory requirements, we downsample the dataset to a resolution of $128\times 128$. Let the corresponding data pairs be denoted by $(\omega_t, \omega_{t+\tau})$ for which we have $10^5$ pairs in the dataset. We build the stochastic interpolant $I_t = \beta_t x_1 + \sigma_t W_t$ where $x_1 = \omega_{t+\tau}-\omega_t$. We construct the objective function similar to \eqref{eqn-loss-for-interpolants}
\begin{equation}
    L_b[\hat b] = \int_0^1 \E \big[|\hat b_t(x_t; \omega_t) - ( \dot \beta_t  x_1+ \dot\sigma_t  \sqrt{t} \,z)|^2\big]  {\rm d}t\, .
\end{equation}
Note that $\omega_t$ is also an input to the drift function $\hat{b}_t$ since our goal here is to estimate the conditional distribution $\omega_{t+\tau}|\omega_t$. This is related but slightly different from the approach in \cite{chen2024probabilistic}; there, the last state is chosen as the point source, while here we use zero.
For all experiments, we employ a UNet architecture~\cite{ho2020denoising} to approximate the drift field. Given a new $\omega_t$, forecasting is done by integrating from $t=0$ to $t=1$ the SDE
\begin{equation}
    {\rm d}\hat{X}_t = \hat{b}_t(\hat{X}_t; \omega_t) {\rm d}t + \sigma_t  {\rm d}W_t, \quad \hat{X}_{t=0} = 0\, ,
\end{equation}
and then the forecast is obtained by $\hat{X}_1 + \omega_t$.

Below, we compare the choices $\beta_t = t, t^2, \frac{2t^2}{t+1}, \frac{3t^2}{2t+1}$; the latter three satisfy $\dot\beta_0 = 0$ while the former does not. In all cases, we choose $\sigma_t = 1 - t$. Below we show the optimization curves of the objective function $L_b$ based on optimization stepsize $l_r = 10^{-4}$ and cosine scheduler. The top and middle rows of Figure \ref{fig:ns-example} show the curve of objective functions and gradient norms during optimization. Clearly, $\beta_t = t$ leads to more noisy gradient norms compared to the other three choices which satisfy $\dot\beta_0 = 0$.
\begin{figure}[ht]
    \centering
    \includegraphics[width=0.24\linewidth]{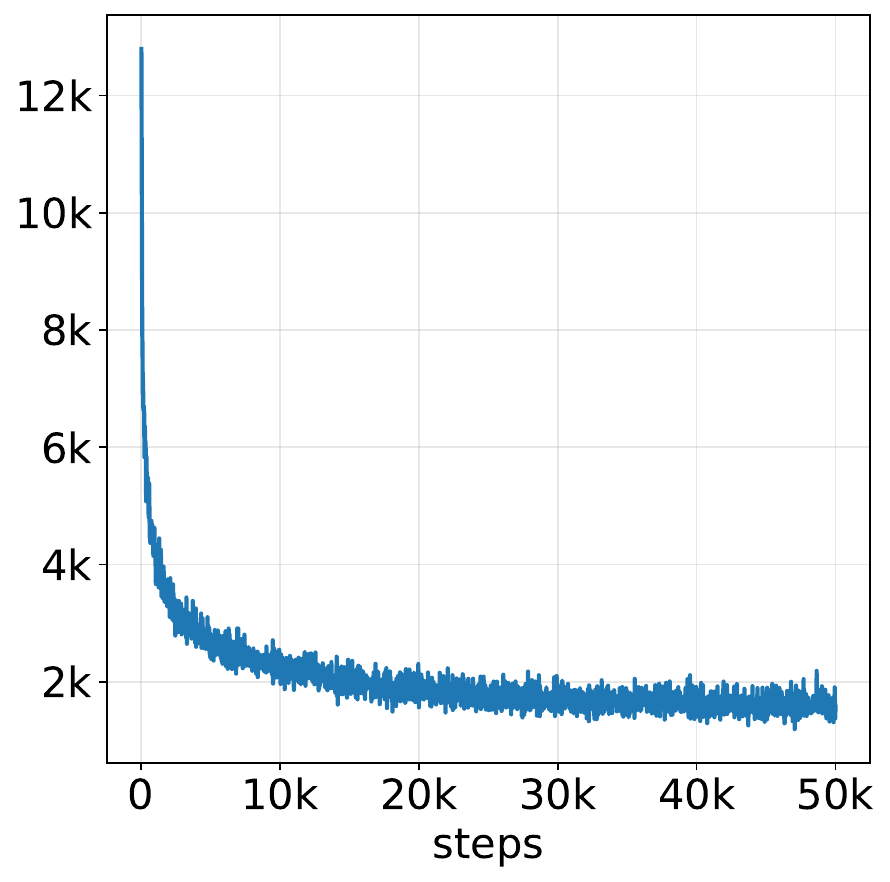}
    \includegraphics[width=0.24\linewidth]{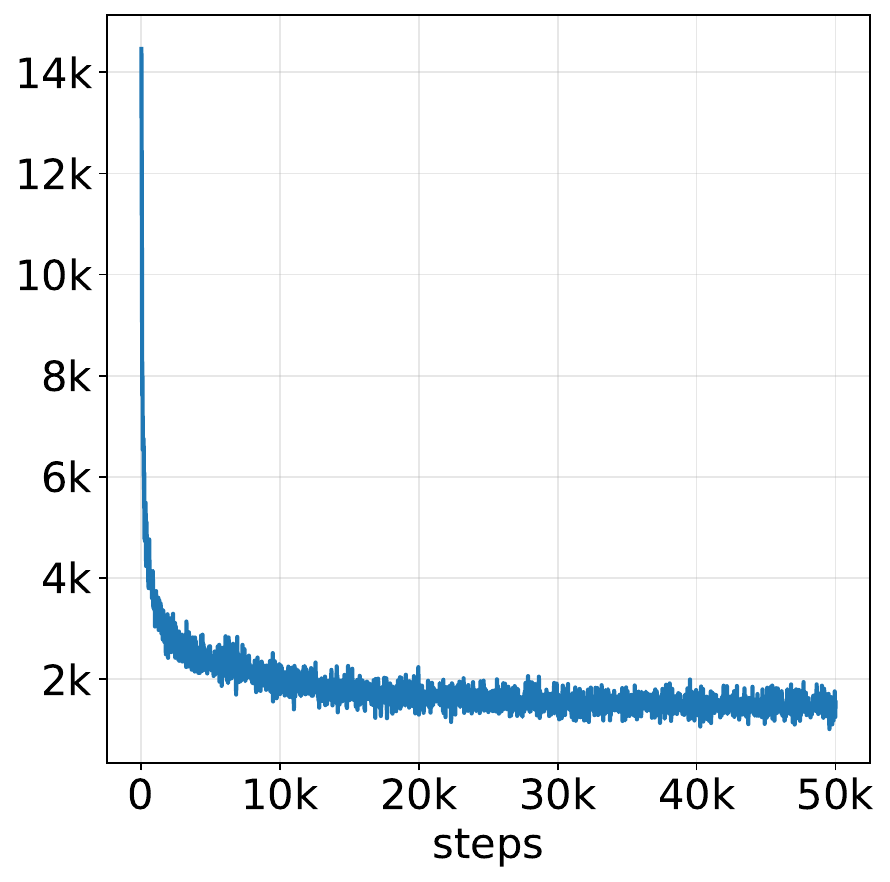}
    \includegraphics[width=0.24\linewidth]{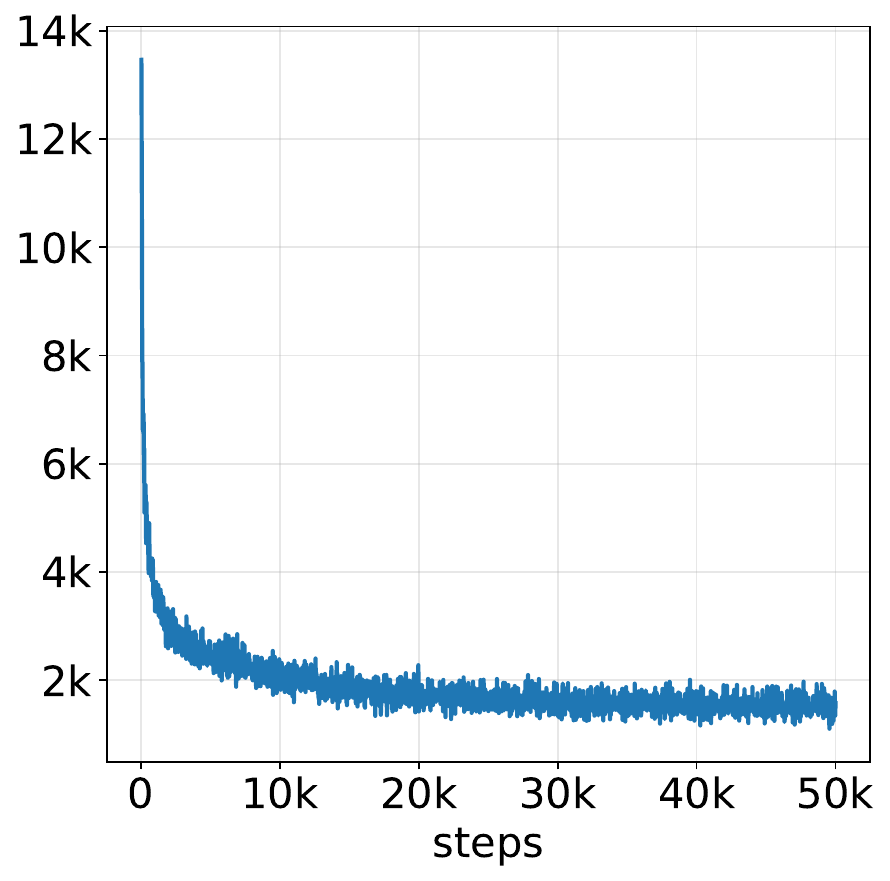}
    \includegraphics[width=0.24\linewidth]{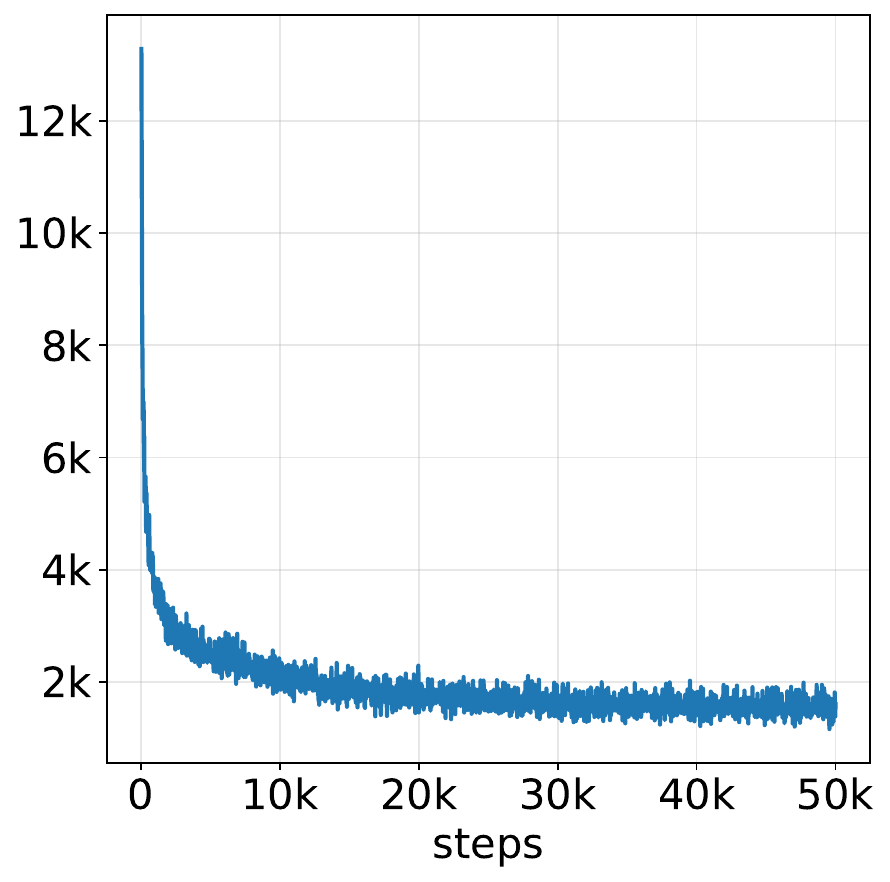}
    \includegraphics[width=0.24\linewidth]{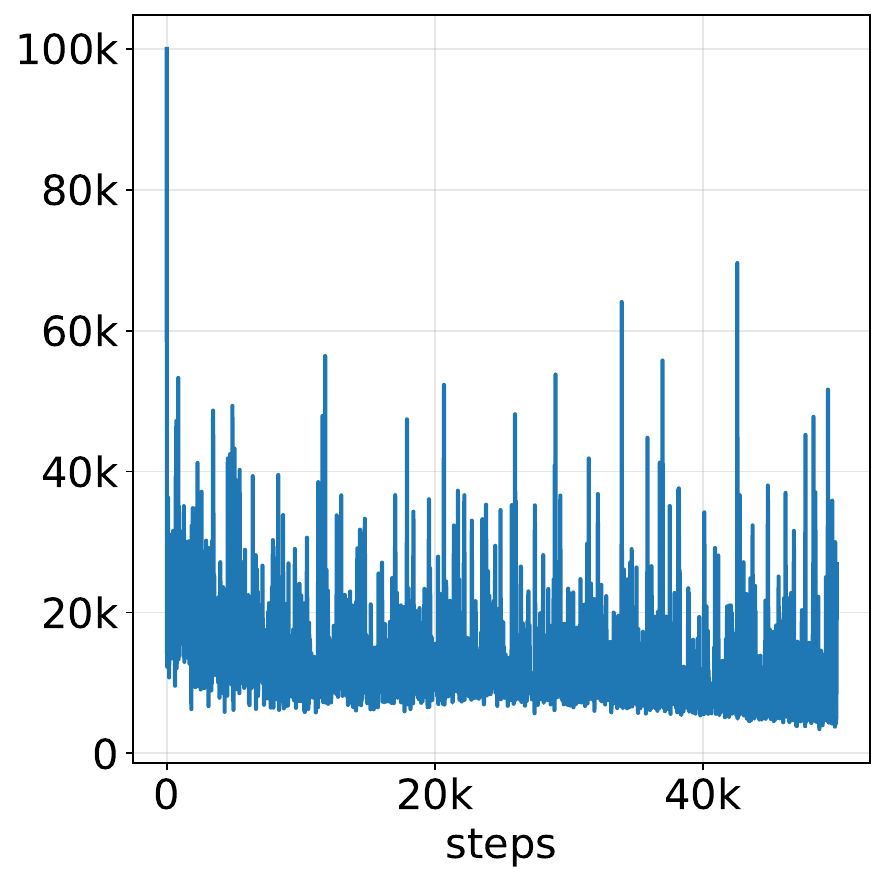}
    \includegraphics[width=0.24\linewidth]{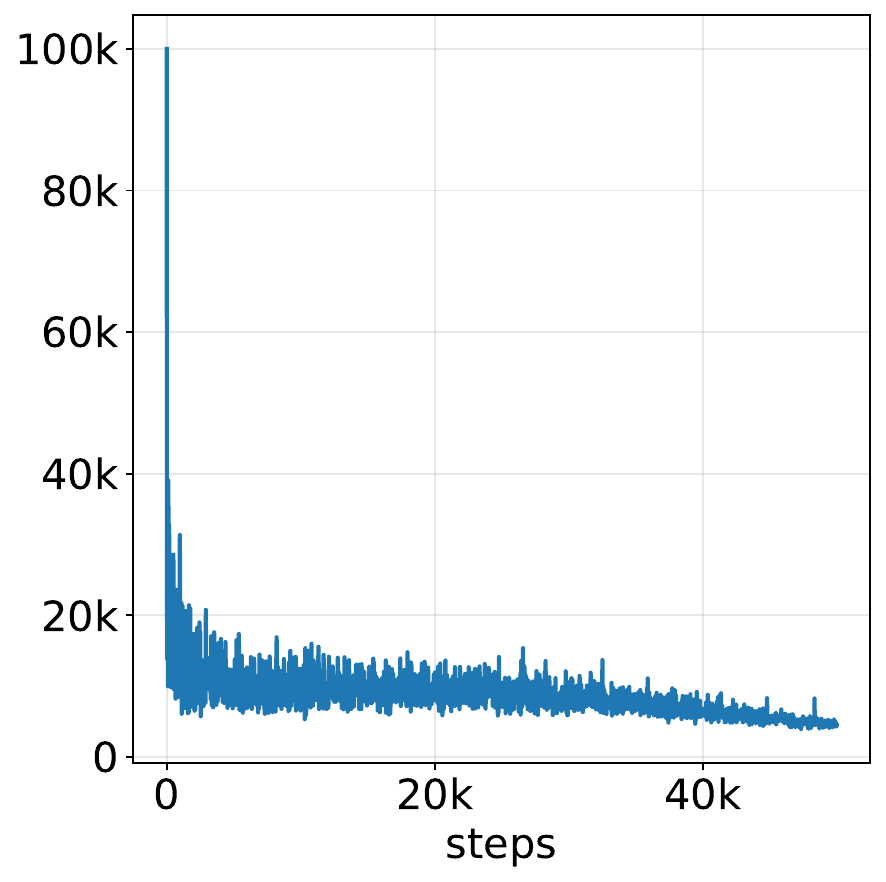}
    \includegraphics[width=0.24\linewidth]{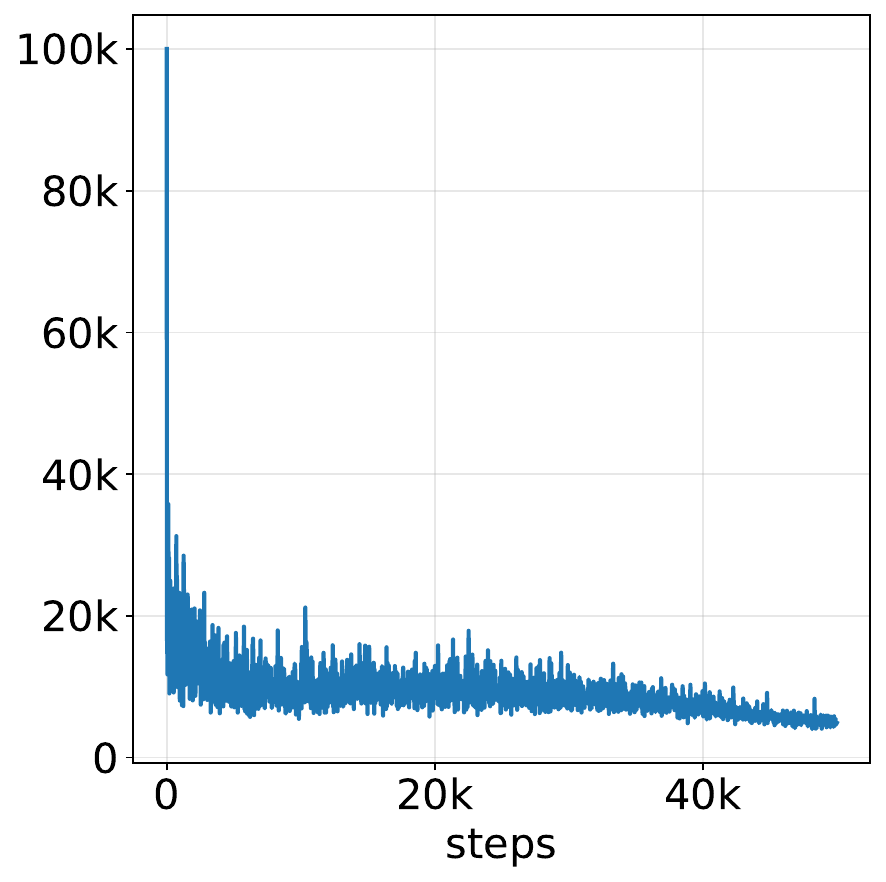}
    \includegraphics[width=0.24\linewidth]{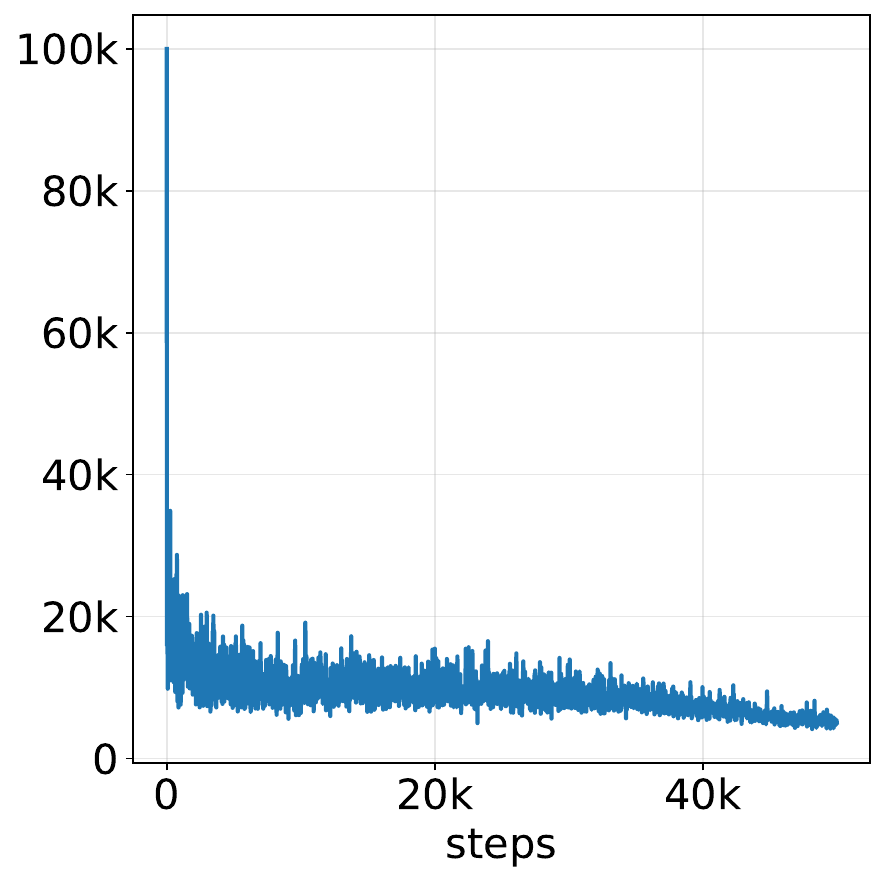}
     \begin{overpic}[width=0.24\linewidth]{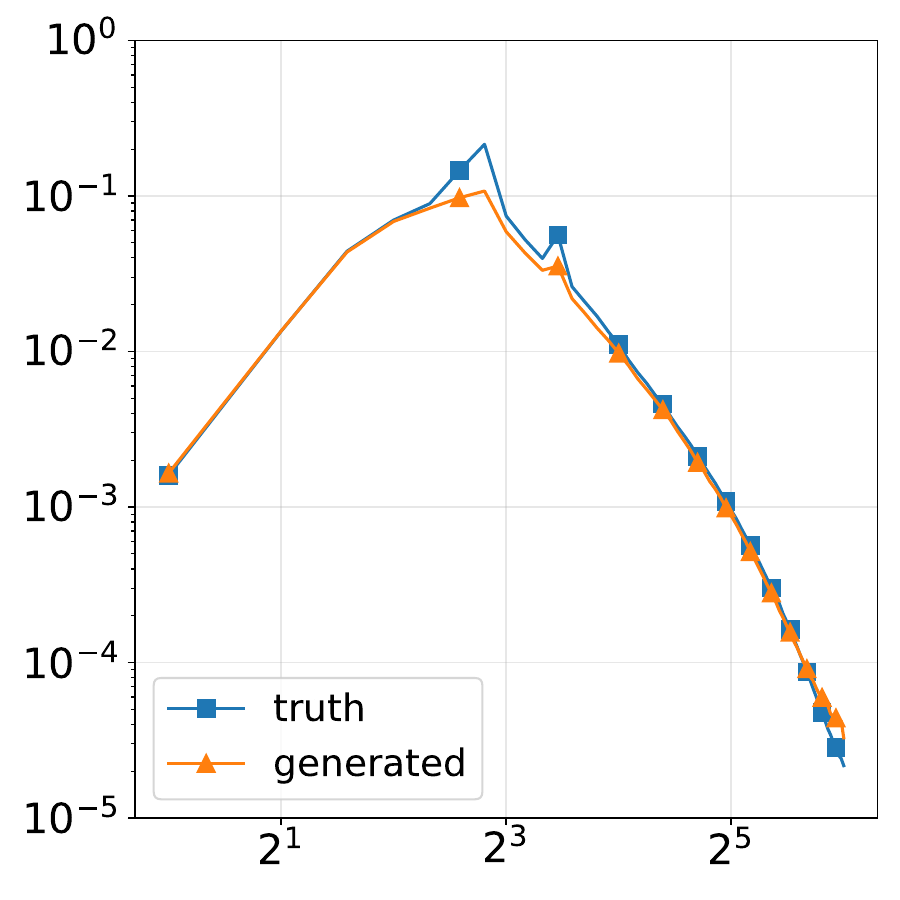}\put(36,-10){$\beta_t = t$}
     \end{overpic}
    \begin{overpic}
        [width=0.24\linewidth]{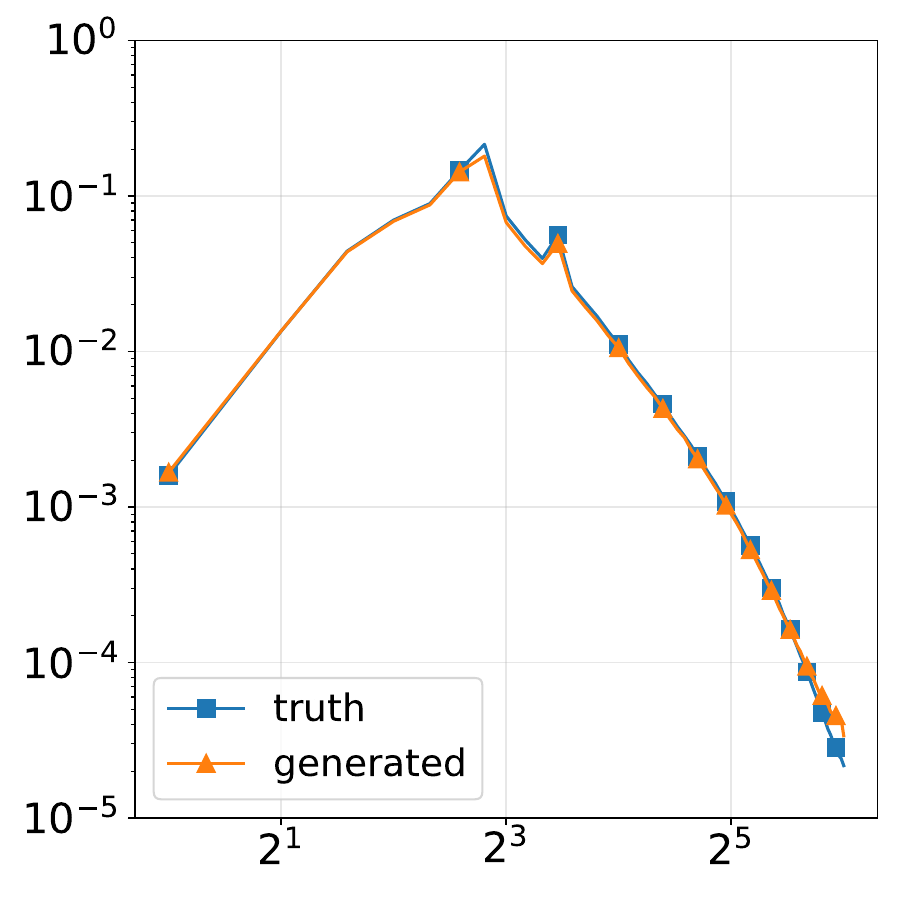}\put(36,-10){$\beta_t = t^2$}
    \end{overpic}
    \begin{overpic}[width=0.24\linewidth]{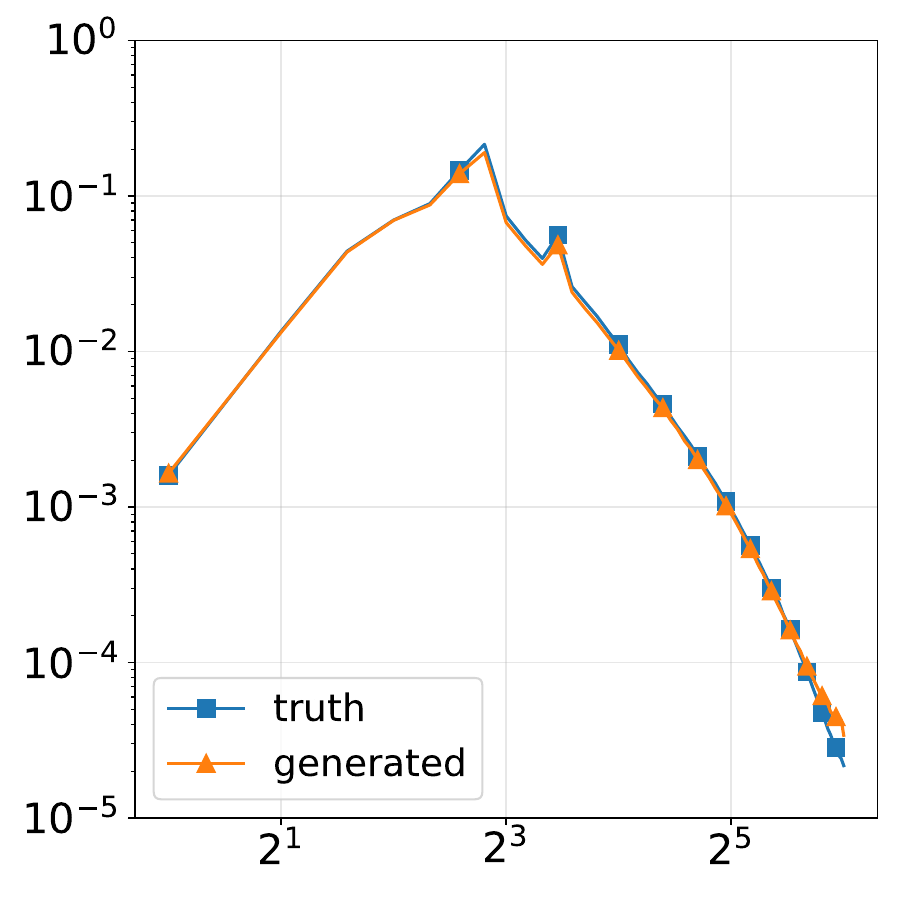}\put(36,-10){$\beta_t = \frac{2t^2}{t+1}$}
    \end{overpic}
    \begin{overpic}[width=0.24\linewidth]{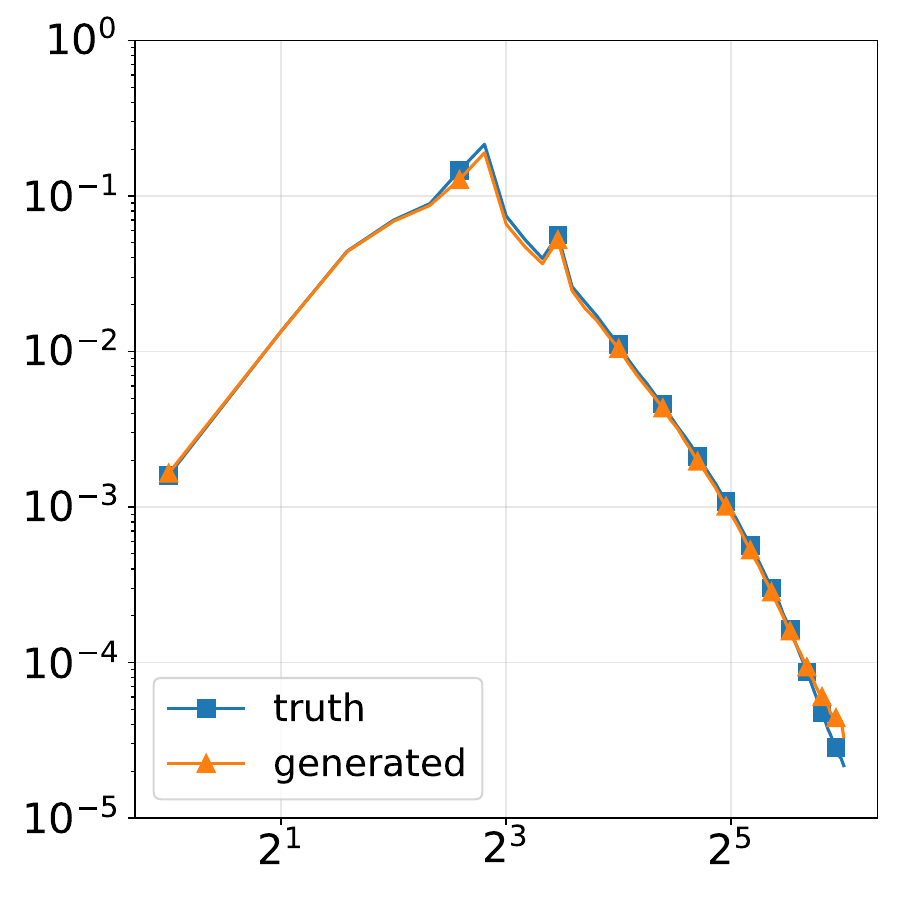}\put(36,-10){$\beta_t = \frac{3t^2}{2t+1}$}
    \end{overpic}
    \vspace{1.5em}
    \caption{\textbf{Top row}: objective functions during optimization. \textbf{Middle row}: gradient norms during optimization. \textbf{Bottom row}: enstrophy spectrum of generated samples versus truth.}
    \label{fig:ns-example}
\end{figure}

We also show the enstrophy spectrum of the estimated forecast; the forecast samples are obtained by applying $100$ Euler-Maruyama (EM) steps to the SDE. For a 2D sample $u$, the spectrum is computed as
\[E(k) = \sum_{k\leq |m|_2\leq k+1} |\hat{u}(m)|^2\, ,\]
where $\hat{u}(m)$ denotes the Fourier coefficients of $u$ at $m \in \mathbb{Z}^2_+ \backslash \{0\}$. We compute $E(k)$ by averaging over a sufficiently large ensemble of samples for each frequency $k$. In the bottom row of Figure \ref{fig:ns-example}, we show the enstrophy spectrum for different $\beta_t$. The choice $\beta_t$ leads to an inaccurate estimation, while the other three lead to similar performance in this case. This demonstrates a correlation between the optimization stability and the generation accuracy. We note that common tricks such as decreasing stepsize do not improve the case $\beta_t = t$ in our experience. Also, with more EM steps, the accuracy does not improve significantly. These results indicate that choosing $\dot\beta_0 = 0$ is useful to improve regularity and stability. 

\subsection{Improved forecasting results using $\dot \beta_0=0$ for the Kolmogorov flow example}
\label{sec:sqlin-posterior}

We close by revisiting the data-assimilation experiment of Section \ref{sec:numerics-ns} under the \emph{square-linear} schedule $\beta_t = t^2$, $\sigma_t = 1-t$, which satisfies the regularity condition $\dot\beta_0 = 0$ studied in Section~\ref{sec-dotbeta-0-improved-training}. All other components of the setup---dataset, U-Net architecture, optimizer, observation model~\eqref{eq:ns-obs}, six held-out test pairs, and forecast lag $\tau = 0.5$---are kept identical to Section~\ref{sec:numerics-ns}, so that any difference in the reported numbers is attributable to the schedule alone.

\vspace{0.2em}
\paragraph{\bf Schedule-specific quantities.} For $\beta_t = t^2$, $\sigma_t = 1-t$, the F\"ollmer diffusion coefficient of Theorem~\ref{prop-min-KL} is
\begin{equation}
\label{eq:gF-sqlin}
    g_t^{\rm F} = \sqrt{(1-t)(3-t)}\, ,
\end{equation}
which satisfies $g_0^{\rm F} = \sqrt{3}$ and $g_1^{\rm F} = 0$. The denoiser $\hat x_1(x,t) = \E[x_\star \mid x_t = x, \tilde\omega_t]$ is obtained as in Section~\ref{sec:numerics-ns}, by combining the baseline-drift identity $b_t = 2t\,\E[x_\star\mid x_t] - \sqrt{t}\,\E[z\mid x_t]$ (from $\dot\beta_t = 2t$, $\dot\sigma_t = -1$) with the linear constraint $x = t^2\,\E[x_\star\mid x_t] + (1-t)\sqrt{t}\,\E[z\mid x_t]$ from the interpolant. Solving the resulting $2\times 2$ system yields the closed-form denoiser
\begin{equation}
\label{eq:tweedie-sqlin}
    \hat x_1(x,t) = \frac{x + (1-t)\,\hat b_t(x\mid\tilde\omega_t)}{t(2-t)} \, .
\end{equation}
The score-form decomposition of the prior drift~\eqref{eq:bg-score-decomp} likewise generalises:
\begin{equation}
\label{eq:bg-score-sqlin}
    \hat b_t^g(x\mid\tilde\omega_t) = \frac{2}{t}\,x + C_g(t)\,\nabla_x\log\rho_t(x\mid\tilde\omega_t), \qquad C_g(t) = \frac{g_t^2 + (1-t)(3-t)}{2} \, ,
\end{equation}
and the identity $C_g \equiv g_t^2 \Leftrightarrow g = g_t^{\rm F}$ from Section~\ref{sec:numerics-ns} continues to hold for this schedule. The posterior SDE therefore takes the same form as~\eqref{eq:ns-guided-sde} with $C_g(t)$ given by~\eqref{eq:bg-score-sqlin} and $\hat x_1$ given by~\eqref{eq:tweedie-sqlin}.

\vspace{0.2em}
\paragraph{\bf Sampler.} We integrate the posterior SDE on the truncated interval $[0.05,\, 0.95]$ with $32$ trajectories per test pair and $100$ Euler--Maruyama steps. The larger left endpoint $t_{\min} = 0.05$ (as opposed to $10^{-2}$ in Section~\ref{sec:numerics-ns}) is required because the F\"ollmer drift~\eqref{eq:bg-score-sqlin}--\eqref{eq:gF-sqlin} for this schedule develops a genuine $1/t$ singularity as $t\to 0^+$, in contrast to the linear-linear case whose drift remains bounded. See Appendix~\ref{appendix-singular-drift-discussion} for a detailed analysis of this singularity.

\vspace{0.2em}
\paragraph{\bf Results.} Table~\ref{tab:ns-posterior-sqlin} reports the relative posterior-mean RMSE under the square-linear schedule, averaged over the same six held-out test pairs as Table~\ref{tab:ns-posterior}, alongside the linear-linear numbers reproduced from that table for direct comparison. Two observations stand out:
\begin{itemize}
\item Across all three diffusion coefficients, the square-linear schedule improves on the linear-linear one. This is consistent with the optimization-stability advantage of $\dot\beta_0 = 0$ documented in Section~\ref{sec-dotbeta-0-improved-training}: a more accurately estimated drift produces a more accurate generative path measure, and posterior sampling inherits this improvement.
\item The F\"ollmer schedule still attains the smallest RMSE among the three diffusion coefficients, with a $\sim\!31\%$ reduction over its linear-linear counterpart ($0.275 \to 0.189$). The ranking of the three coefficients (F\"ollmer $<$ constant $<$ baseline) is preserved.
\end{itemize}
Combining the improved training stability from $\dot\beta_0 = 0$ with F\"ollmer tuning of the diffusion coefficient helps, and the resulting posterior samples are the most accurate of the four configurations considered in this paper.

\begin{table}[h]
\centering
\small
\begin{tabular}{lcc}
\toprule
diffusion coefficient & linear-linear ($\beta_t = t$) & square-linear ($\beta_t = t^2$) \\
\midrule
Baseline, $g_t = \sigma_t$   & $0.313 \pm 0.090$          & $0.225 \pm 0.040$ \\
Constant, $g_t = 1$          & $0.297 \pm 0.084$          & $0.194 \pm 0.025$ \\
F\"ollmer, $g_t = g_t^{\rm F}$ & $\mathbf{0.275 \pm 0.091}$ & $\mathbf{0.189 \pm 0.025}$ \\
\bottomrule
\end{tabular}
\caption{Relative posterior-mean RMSE on the Kolmogorov-flow data-assimilation task of Section~\ref{sec:numerics-ns} under two schedules: the linear-linear schedule $\beta_t = t$, $\sigma_t = 1-t$ (left column, reproduced from Table~\ref{tab:ns-posterior}) and the square-linear schedule $\beta_t = t^2$, $\sigma_t = 1-t$ for which $\dot\beta_0 = 0$ (right column). Mean $\pm$ standard deviation across the same six held-out test pairs $(\tilde\omega_t, \omega_{t+\tau})$ as Table~\ref{tab:ns-posterior}; $16$ posterior samples per test pair.}
\label{tab:ns-posterior-sqlin}
\end{table}

\end{document}